\documentclass[12pt]{amsart}
\usepackage[utf8]{inputenc}
\usepackage{kantlipsum}
\usepackage[pagebackref=true, colorlinks=true,linkcolor=blue, citecolor={red}]{hyperref}
\usepackage{lmodern}\usepackage[T1]{fontenc}
\usepackage{amsthm, amsmath,amssymb,amsfonts,euscript,graphicx,hyperref,indentfirst}
\usepackage{enumerate}
\usepackage[all,cmtip]{xy}
\usepackage{relsize,exscale}
\usepackage{textcomp}
\calclayout

\newtheorem{theorem}{Theorem}[section]
\newtheorem{lemma}[theorem]{Lemma}
\newtheorem{proposition}[theorem]{Proposition}
\newtheorem{corollary}[theorem]{Corollary}
\newtheorem{conjecture}[theorem]{Conjecture}

\newcommand{\thistheoremname}{}
\newtheorem{genericthm}[theorem]{\thistheoremname}
\newenvironment{namedthm}[1]
  {\renewcommand{\thistheoremname}{#1}%
   \begin{genericthm}}
  {\end{genericthm}}

\theoremstyle{definition}

\newtheorem{example}[theorem]{Example}
\newtheorem{question}[theorem]{Question}

\theoremstyle{remark}
\newtheorem{remark}[theorem]{Remark}

\makeatletter
\def\l@subsection{\@tocline{2}{0pt}{2.5pc}{5pc}{}} 
\makeatother

\numberwithin{equation}{section}

\author{Khashayar Filom}
\address{Department of Mathematics,
  University of Michigan;
 Ann Arbor, MI 48109-1043,
USA}
\email{filom@umich.edu}

\begin{document}

\begin{abstract}
The monotonicity of entropy is investigated for real quadratic rational maps on the real circle $\Bbb{R}\cup\{\infty\}$ based on the natural partition of the corresponding moduli space $\mathcal{M}_2(\Bbb{R})$ into its monotonic, covering, unimodal and bimodal regions. Utilizing the theory of  polynomial-like mappings, we  prove that the level sets of the real entropy function $h_\Bbb{R}$ are connected in the $(-+-)$-bimodal region and a portion of the unimodal region in $\mathcal{M}_2(\Bbb{R})$. Based on the numerical evidence, we conjecture that the monotonicity holds throughout the unimodal region, but we conjecture that it fails in the region of $(+-+)$-bimodal maps.
\end{abstract}

\title{Monotonicity of entropy for real quadratic rational maps}
\maketitle

\renewcommand{\baselinestretch}{0.9}\normalsize
\small
\tableofcontents
\renewcommand{\baselinestretch}{1.0}\normalsize
\normalsize

\section{Introduction}\label{S1}

The present article discusses the entropy of a real quadratic rational map $f\in\Bbb{R}(z)$ on the extended real line $\hat{\Bbb{R}}:=\Bbb{R}\cup\{\infty\}$ 
\begin{equation}\label{real entropy}
h_{\Bbb{R}}(f):=h_{\rm{top}}\left(f\restriction_{\hat{\Bbb{R}}}:\hat{\Bbb{R}}\rightarrow\hat{\Bbb{R}}\right).
\end{equation}
We investigate how this \textit{real entropy} varies in families of real quadratic rational maps. The real entropy defines a continuous function 
\begin{equation}\label{real entropy function}
h_\Bbb{R}:\mathcal{M}_2(\Bbb{R})-\mathcal{S}(\Bbb{R})\rightarrow\left[0,\log(2)\right]
\end{equation}
on the Zariski open subset $\mathcal{M}_2(\Bbb{R})-\mathcal{S}(\Bbb{R})$ of the \textit{moduli space}
$\mathcal{M}_2(\Bbb{R})$ of real quadratic rational maps, where $\mathcal{S}(\Bbb{R})$, the locus determined by real maps admitting non-trivial Möbius symmetries, is excluded so that the function $h_\Bbb{R}$ is single-valued; see \S\ref{S2.1} for details. The main question regarding this function is about the nature of its level sets called  the \textit{isentropes}:
\begin{question}\label{monotonicity formulation}
Restricted to a connected component of $\mathcal{M}_2(\Bbb{R})-\mathcal{S}(\Bbb{R})$, are the level sets of the function 
$h_{\Bbb{R}}$  connected?
\end{question}
The space $\mathcal{M}_2(\Bbb{R})-\mathcal{S}(\Bbb{R})$ has three connected components which are distinguished via the topological degree of $f$ on $\hat{\Bbb{R}}$ that can be $+2$, $0$, or $-2$.  The real entropy $h_{\Bbb{R}}$ is $\log(2)$ whenever the degree is  $\pm 2$, a situation that occurs precisely when the critical points of $f$ are complex conjugate; see Proposition \ref{Julia circle} for details. Question \ref{monotonicity formulation} is thus interesting only for the level sets in the connected component of $\mathcal{M}_2(\Bbb{R})-\mathcal{S}(\Bbb{R})$  consisting of maps with real critical points.  We call this component  the {\em component of degree zero maps} hereafter. This component  contains the set $\left\{\left\langle z^2+c\right\rangle\right\}_{-2\leq c\leq\frac{1}{4}}$ of the Möbius conjugacy classes of real quadratic polynomials with connected Julia sets as a line segment lying in the moduli space $\mathcal{M}_2(\Bbb{R})$ (which, as discussed below, could be identified with the plane $\Bbb{R}^2$). It is classically known that  the entropy function   $h_\Bbb{R}$ is monotonic on this polynomial segment \cite{MR970571,MR762431,MR1351519}.   
(See \cite{MR1239171} for more on this along with an unpublished proof by Sullivan.)
We shall prove the following generalization:  
\begin{theorem}\label{temp1}
Restricted to the part of the moduli space  $\mathcal{M}_2(\Bbb{R})$ where the critical points are real and the maps have three real fixed points, the level sets of  $h_\Bbb{R}$ are connected.
\end{theorem}

The domain of the real entropy function in \eqref{real entropy function} and the region under consideration in the preceding theorem may be explained in more details: The space $\mathcal{M}_2(\Bbb{R})$ is the set of $\Bbb{R}$-points of the moduli space $\mathcal{M}_2$  of quadratic rational maps 
(which, as a variety, admits a model over $\Bbb{Q}$).
It coincides with the space 
${\rm{PGL}}_2(\Bbb{C}).\left({\rm{Rat_2}}(\Bbb{R})\right)\big/{\rm{PGL}}_2(\Bbb{C})$  
of Möbius conjugacy classes of real quadratic rational maps that itself can be identified with the plane $\Bbb{R}^2$ via a coordinate system $(\sigma_1,\sigma_2)$ defined in terms of multipliers of fixed points 
(\cite[\S10]{MR1246482}), a plane in which $\mathcal{S}(\Bbb{R})$ is a rational cuspidal curve (\cite[\S5]{MR1246482}). 
A real quadratic map $f\in\Bbb{R}(z)$ either has complex conjugate critical points and thus induces a two-sheeted covering 
$f\restriction_{\hat{\Bbb{R}}}:\hat{\Bbb{R}}\rightarrow\hat{\Bbb{R}}$,
 or is with real critical points in which case the circle map above is of degree zero and restricts to an interval map 
$f\restriction_{f(\hat{\Bbb{R}})}:f(\hat{\Bbb{R}})\rightarrow f(\hat{\Bbb{R}})$ 
for which 
\begin{equation}\label{real entropy 1}
h_{\Bbb{R}}(f)=h_{\rm{top}}\left(f\restriction_{\hat{\Bbb{R}}}:\hat{\Bbb{R}}\rightarrow\hat{\Bbb{R}}\right)
=h_{\rm{top}}\left(f\restriction_{f(\hat{\Bbb{R}})}:f(\hat{\Bbb{R}})\rightarrow f(\hat{\Bbb{R}})\right).
\end{equation}
This dichotomy will be established in Proposition \ref{Julia circle}. Then conditioning on the shape (i.e. whether the graph of $f$ starts with an increase or with a decrease) and the modality of these interval maps  results in a natural partition of $\mathcal{M}_2(\Bbb{R})\cong\Bbb{R}^2$ into 
 \textit{degree} $\pm2$, \textit{increasing}, \textit{decreasing}, \textit{unimodal}, $(-+-)$-\textit{bimodal}
and $(+-+)$-\textit{bimodal} regions \cite{MR1246482}; see Figure \ref{fig:main} adapted from Milnor's paper  (or the more zoomed-out version in Figure \ref{fig:colored}) for this partition. Of course, outside the latter three, $h_\Bbb{R}$ is either identically zero or identically $\log(2)$; Question \ref{monotonicity formulation}  is thus interesting only in unimodal and bimodal regions of the moduli space. The region mentioned in Theorem \ref{temp1} is colored in green in Figure \ref{fig:colored} and, aside from a tiny portion of the monotone increasing region, consists of the whole $(-+-)$-bimodal region along with parts of unimodal and $(+-+)$-bimodal regions that lie strictly below a certain line in the plane $\mathcal{M}_2(\Bbb{R})\cong\Bbb{R}^2$. This description will be established in \S\ref{S4} after a careful discussion on the number of real fixed points. The proof of Theorem \ref{temp1} then relies on the Douady and Hubbard theory of \textit{polynomial-like mappings} from \cite{MR816367}, a technical result from \cite{MRUhre} invoked to control the straightening in a family and finally, the monotonicity of entropy for real quadratic polynomials as established in \cite{MR970571,MR762431,MR1351519}. 

Theorem \ref{temp1}  implies that restrictions of $h_\Bbb{R}$ to the  $(-+-)$-bimodal region or to the part of the unimodal region where  maps have three real fixed points both have connected level sets; see Corollary \ref{monotonicity 1'}. 
An immediate question then arises: Are isentropes connected in the entirety of the unimodal region (i.e. the middle rectangular part of Figure \ref{fig:colored} comprising of both green and yellow parts)? Supported by the numerically generated entropy plots in \S\ref{S5}, we pose the following conjecture: 
\begin{conjecture}\label{monotonicity 2}
The level sets of the restriction of  $h_\Bbb{R}$ to the unimodal region of the moduli space
$\mathcal{M}_2(\Bbb{R})$ are connected.
\end{conjecture}
In \S\ref{S6.2}, we elaborate more on this conjecture and formulate other more tractable conjectures implying it; see Proposition \ref{conjectures}. The techniques are of a different flavor and are reminiscent of the treatment of the monotonicity problem for cubic polynomials in \cite{MR1351522,MR1736945}. Conjecture \ref{monotonicity 2} is discussed in  \cite{2020arXiv200903797G} by invoking a machinery developed in \cite{MR4115082}, and also  in \cite{2020arXiv200910147B} with a different approach.
\begin{figure}[ht!]
\center
\includegraphics[width=15cm, height=8.1cm]{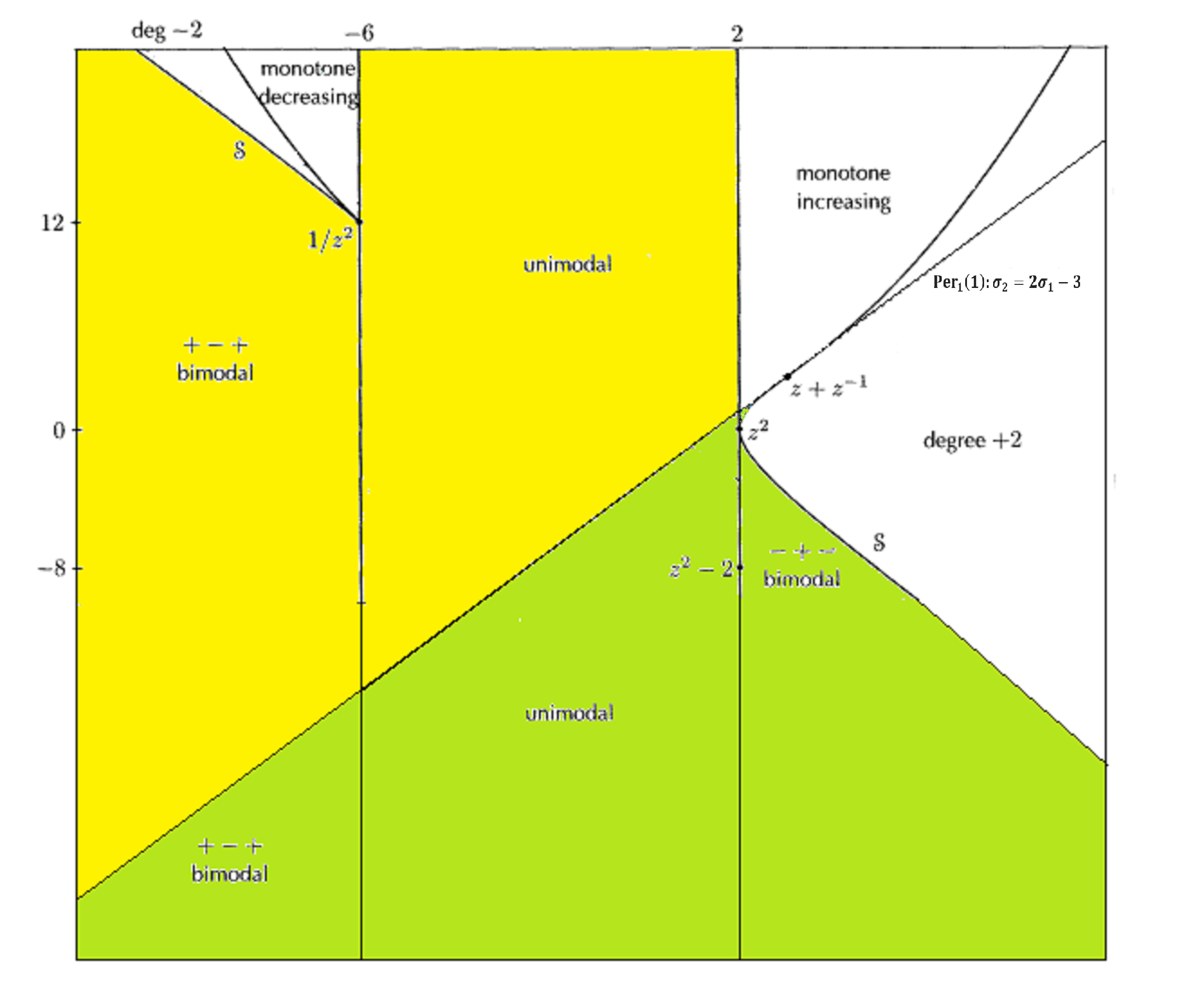}
\caption{A picture of the  real moduli space $\mathcal{M}_2(\Bbb{R})$ in the $(\sigma_1,\sigma_2)$-coordinate system 
(cf. \S\ref{S2.1}) and its division into seven regions based on the topological type of the dynamics induced on $\hat{\Bbb{R}}$ adapted from \cite{MR1246482} with slight changes (cf. Figure \ref{fig:main}). One has $h_\Bbb{R}\equiv 0$ over the monotonic regions and $h_\Bbb{R}\equiv\log(2)$ over the degree $\pm2$ regions. The critical points are not real for maps from the latter two. The line separating yellow and green parts is the line ${\rm{Per}}_1(1)$ where a real fixed point becomes multiple.  Throughout the green part strictly below this line the maps have three real fixed points and the level sets of the restriction of $h_\Bbb{R}$ are connected; see Theorem \ref{temp1} or its reformulation as Theorem \ref{monotonicity 1}.}
\label{fig:colored}
\end{figure}

The empirical data in \S\ref{S5} moreover suggest that the monotonicity fails in the $(+-+)$-bimodal region of the moduli space.
\begin{conjecture}\label{temp2}
The level sets of the restriction of  $h_\Bbb{R}$ to the $(+-+)$-bimodal region of the moduli space
$\mathcal{M}_2(\Bbb{R})$ can be disconnected.
\end{conjecture}
Borrowing the terminology used in \cite{MR1806289}, the maps in Theorem \ref{temp1} exhibit a ``polynomial-like behavior'':
once a real quadratic rational map has three real fixed points, at least one of them must be attracting (Lemma \ref{attracting fixed point}), and  outside the basin of that fixed point the map can be quasi-conformally perturbed to a real polynomial of the same real entropy (Theorem \ref{straightening}). In contrast,  Conjecture \ref{temp2} is about an ``essentially non-polynomial behavior'' where all three fixed points are repelling: Throughout the yellow part of the $(+-+)$-bimodal region in Figure \ref{fig:colored} (which is missing from Theorem \ref{temp1}), all fixed points (real or complex) are repelling. This ``non-polynomial'' attribution necessitates completely different techniques for investigating Conjecture \ref{temp2} as demonstrated in \cite{1930-5311_2020_16_225}.

\textbf{Motivation.}
A complex rational map 
$f:\hat{\Bbb{C}}\rightarrow\hat{\Bbb{C}}$ of degree $d\geq 2$ is of topological entropy $\log(d)$ and admits a unique  measure of maximal entropy $\mu_f$ whose support is the Julia set $\mathcal{J}(f)$ \cite{MR736568,MR736567}. By contrast, when $f$ is with real coefficients, the circle $\hat{\Bbb{R}}:=\Bbb{R}\cup\{\infty\}$ is invariant under $f$ and 
the entropy of the restriction 
$$h_{\Bbb{R}}(f)=h_{\rm{top}}\left(f\restriction_{\hat{\Bbb{R}}}:\hat{\Bbb{R}}\rightarrow\hat{\Bbb{R}}\right)$$ 
can take all values in $\left[0,\log(d)\right]$. The behavior of this real entropy as $f$ varies in families is then worthy of study.\\
\indent
There is an extensive literature on Milnor's conjecture on the \textit{monotonicity of entropy} for families of real polynomials \cite{MR3289915}. The first result of this type, appeared in  \cite{MR970571}, is due to Milnor and Thurston
where they prove that in the quadratic family 
\begin{equation}\label{quadratic family}
\left\{x\mapsto \frac{x^2-a}{2}
:[-(1+\sqrt{1+a}),1+\sqrt{1+a}]\rightarrow [-(1+\sqrt{1+a}),1+\sqrt{1+a}]\right\}_{-1\leq a\leq 8}
\end{equation}
 the entropy increases with the parameter $a$ (also see \cite{MR762431,MR1351519}).  
The proof in \cite{MR970571} has two ingredients: the \textit{kneading theory} -- a combinatorial tool for studying piecewise-monotone mappings (developed in the same paper) -- and a special case of the \textit{Thurston rigidity for  post-critically finite polynomials}.  For polynomial interval maps of higher degrees,  the parameter space is not one-dimensional and so the monotonicity of a function on it has to be interpreted as the connectedness of its level sets. The next development was the case of bimodal cubic polynomial maps 
\cite{MR1351522,MR1736945}. The proof relies on delicate planar topology arguments and analyzing certain real curves, called \textit{bones},  in the parameter space  that are defined by imposing a periodic condition on a critical point.  Given the analogy between quadratic rational maps and cubic polynomials, it is no surprise that similar ideas come up in our treatment of the monotonicity problem for real quadratic rational maps; see \S\ref{S6.2}. \\
\indent
The monotonicity conjecture on the connectedness of isentropes for polynomial maps  of an arbitrary degree $d$
was established in  \cite{MR3264762} by  Bruin and van Strien (and also in \cite{MR3999686} by a different method). It must be mentioned that in their work, as well as the aforementioned works of Milnor and Thurston  \cite{MR970571} or Milnor and Tresser \cite{MR1736945}, one deals with degree $d$ polynomials with real non-degenerate critical points that restrict to boundary-anchored interval maps of modality $d-1$. These properties are all apparent in the family \eqref{quadratic family} where $d=2$. This is in contrast with our context since, given a real quadratic rational map $f$ with real critical points, the interval map 
$f\restriction_{f(\hat{\Bbb{R}})}:f(\hat{\Bbb{R}})\rightarrow f(\hat{\Bbb{R}})$ is not boundary-anchored -- a fact that makes its kneading theory more complicated -- and may be bimodal as well; although, unlike the case of cubic polynomial interval maps, its entropy is at most $\log(2)$, strictly less than the maximum entropy that a bimodal map can realize which is $\log(3)$. A key idea in both \cite{MR3264762} and \cite{MR1736945} is to associate with the polynomial interval map under consideration a piecewise-linear ``combinatorial model'' from an appropriate family of \textit{stunted sawtooth maps} that has the same kneading data; and moreover, the connectedness of isentropes is easier to verify for the stunted sawtooth family. Our treatment lacks such a model due to the aforementioned difficulties; cf. the discussion in \S\ref{S7}.  \\
\indent 
The first step to pose a question about the entropy behavior in a family is to parametrize the space of maps under consideration.  In the usual setting of polynomial interval maps of the highest possible modality (e.g. treatments of the monotonicity conjecture in \cite{MR970571,MR1351522,MR1736945,MR3264762}) 
one can rely on the fact that such interval maps are uniquely determined up to an affine conjugacy by their critical values \cite[Theorem 3.2]{MR1736945}. But for real rational maps, the natural space to work with is the locus determined by real maps in the space $\mathcal{M}_d(\Bbb{C}):={\rm{Rat}}_d(\Bbb{C})/{\rm{PSL}}_2(\Bbb{C})$
of degree $d$ holomorphic dynamical systems on the Riemann sphere modulo the conformal conjugacy.
There is no analogous parametrization of this moduli  space in terms of critical values. As a matter of fact, we shall encounter families of interval maps whose dimension is strictly larger than the modality of maps in the family, e.g. two-parameter families of unimodal interval maps; see Proposition \ref{unimodal-bimodal}. To the best of our knowledge, the monotonicity problem for such families has not been fully investigated, at least not to the extent of the research on polynomial maps of the highest possible modality in the aforementioned references on the monotonicity problem. Nevertheless, \cite{MR3289915, MR3264762} 
allude to a monotonicity result for a two-parameter family of boundary-anchored unimodal maps given by real quartic polynomials with a pair of non-real critical points.
\\
\indent 
In the broader context of real rational maps, the research on the monotonicity question as well as other questions concerning the isentropes is still in its early stages. The current article, along with \cite{filom2021real}, can be considered as the first attempts toward formulating the monotonicity problem in this context and addressing the difficulties that arise when one passes from polynomial maps to rational maps.

\textbf{Outline.}
In the present paper, we solely focus on the case of $d=2$ where $\mathcal{M}_2(\Bbb{C})$ can be famously identified with $\Bbb{C}^2$, with the subspace $\mathcal{M}_2(\Bbb{R})$ of classes with real representatives being  the underlying real plane $\Bbb{R}^2$.
We have devoted \S\ref{S2} to the background material on spaces $\mathcal{M}_2(\Bbb{C})$, $\mathcal{M}_2(\Bbb{R})$ and to setting up  the real entropy function $h_\Bbb{R}$ that will put Question \ref{monotonicity formulation} on a firm footing. \\
\indent
In \S\ref{S3}, we concentrate on the dense open subset of classes of \textit{hyperbolic} maps in $\mathcal{M}_2(\Bbb{R})$ where $h_\Bbb{R}$ is locally constant; see Theorem \ref{entropy constant over hyperbolic}. We will completely describe the real dynamics in a prominent hyperbolic component namely, the \textit{escape locus}; see Proposition \ref{escape locus}.     \\
\indent
Generating entropy plots in the moduli space $\mathcal{M}_2(\Bbb{R})$ poses a practical difficulty as one needs to have parametrized families of real quadratic rational maps in hand before applying any algorithm for calculating entropy. To this end, a  real parameter space will be introduced in \S\ref{S4} which is more convenient for computer implementations and admits a 
finite-to-one map into 
$\mathcal{M}_2(\Bbb{R})$; cf. \eqref{the transformation}. Several observations will be made in \S\ref{S4} that will result in a better understanding of the dynamics and in excluding certain parts of the parameter space that are inconsequential to Question \ref{monotonicity formulation} in the sense that in those regions the function $h_\Bbb{R}$ is identically $0$ or $\log(2)$ and the induced dynamics on the real circle can be explicitly described. These observations lead to two important conclusions:
\begin{itemize}
\item The number of real fixed points plays an important role in the analysis of the real dynamics; see Lemma \ref{attracting fixed point} and observation \hyperref[f]{4.f}. 
\item The $(-+-)$-bimodal situation can be reduced to the unimodal one by observation \hyperref[j]{4.j}. Therefore, the study of the induced dynamics on $\hat{\Bbb{R}}$ essentially reduces to studying certain two-parameter families of unimodal and  $(+-+)$-bimodal interval maps; see Proposition \ref{unimodal-bimodal}.
\end{itemize}

In \S\ref{S5}, we first numerically generate contour plots for the entropy function in this new parameter space and then project them into the moduli space. 
As mentioned above, we only need to deal with two-parameter families of $(+-+)$-bimodal and unimodal interval maps.  The entropy plots obtained in \S\ref{S5} suggest that the monotonicity fails for the former while it holds for the latter. The first observation substantiates  Conjecture \ref{temp2}. We will elaborate a bit more on this conjecture in \S\ref{S7}. The second one will be partially established in \S\ref{S6} as the proof of Theorem \ref{temp1} by invoking the straightening theorem where a  real attracting fixed point is present, i.e. the $(-+-)$-bimodal region and a certain part of unimodal region. The monotonicity in the whole unimodal region is the content of Conjecture \ref{monotonicity 2} whose analysis is closely related to certain post-critical curves in that region. Imitating the ideas developed in \cite{MR1351522,MR1736945}, in \S\ref{S6.2} we argue that this conjecture is implied by other conjectures on the nature of these curves and the kneading theory of maps that they pass through. \\
\indent 
The moduli space of rational maps with only two critical points bears a significant resemblance to the moduli space of quadratic rational maps \cite{MR1806289}. Hence many of these ideas can potentially be developed 
in that context as well. This will be briefly discussed in \S\ref{S8}.


\textbf{Notation and terminology.} 
As for notations, $\hat{\Bbb{C}}$ and $\hat{\Bbb{R}}$ denote the compactifications $\Bbb{C}\cup\{\infty\}$
and $\Bbb{R}\cup\{\infty\}$ of $\Bbb{C}$ and $\Bbb{R}$, and we use $z$ and $x$ for the coordinates on them respectively.  
The open disk $\left\{|z|<r\right\}$ is denoted by $\Bbb{D}_r$ and for the open unit disk $\Bbb{D}_1$ the notation $\Bbb{D}$ is used instead. The Mandelbrot set in the complex plane is denoted by $\mathbf{M}$.\\
\indent
For notations related to the moduli space of rational maps, we mainly follow \cite{MR1246482}: the moduli space 
of quadratic rational maps is the variety $\mathcal{M}_2(\Bbb{C})={\rm{Rat}}_2(\Bbb{C})\big/{\rm{PGL}}_2(\Bbb{C})$ where ${\rm{Rat}}_2(\Bbb{C})\subset\Bbb{P}^5(\Bbb{C})$ is the space of degree two rational maps and the Möbius conjugacy class of a map $f\in{\rm{Rat}}_2(\Bbb{C})$ is denoted by $\langle f\rangle\in\mathcal{M}_2(\Bbb{C})$. The \textit{symmetry locus}  $\mathcal{S}(\Bbb{C})$  is the subvariety determined by rational maps $f$ for which the group ${\rm{Aut}}(f)$ of Möbius transformations commuting with $f$ is non-trivial. The curve ${\rm{Per}}_n(\lambda)\subset\mathcal{M}_2(\Bbb{C})$ by definition consists of quadratic maps that admit an $n$-cycle of multiplier $\lambda$ \cite[p. 41]{MR1246482}. To speak of the corresponding sets of $\Bbb{R}$-points, we use ${\rm{Rat}}_2(\Bbb{R})$, $\mathcal{M}_2(\Bbb{R})$, $\mathcal{S}(\Bbb{R})$ and ${\rm{Per}}_n(\lambda)(\Bbb{R})$
respectively.
\\
\indent
A real-valued function is called \textit{monotonic} if its level sets are connected. It is called monotonic over a subset of its domain if the corresponding restriction is monotonic. \\
\indent 
A self-map of an interval is called \textit{boundary-anchored} if it takes boundary points to boundary points. A \textit{lap}  is defined to be a maximal monotonic subinterval. The number of laps (i.e. modality$+1$) is called the \textit{lap number}.  The entropy of a multimodal  interval map is famously the same as the exponential growth rate  of lap numbers of its iterates \cite[Theorem 1]{MR579440}.\\
\indent
Given a rational map $f$, $\mathcal{J}(f)$ denotes its Julia set and $\bar{f}$ is the map obtained from applying the complex conjugation to its coefficients. For the broader class of complex-valued functions on $\hat{\Bbb{C}}$, we use the notation $\tilde{f}$ instead for the map obtained from conjugating $z\mapsto f(z)$ with $z\mapsto\bar{z}$; cf. \eqref{involution}. We call a rational map $f$ to be \textit{post-critically finite} or \textit{PCF} if its critical points are eventually periodic. 

\textbf{Acknowledgment.} The author would like to thank Laura DeMarco for proposing this project on the entropy behavior of real quadratic rational maps and also for many fruitful conversations. I also would like to thank Corinna Wendisch for her help in the early stages of this project,  Eva Uhre for generously sharing her thesis \cite{MRUhre}, and Yan Gao for helpful discussions. I am grateful to Kevin Pilgrim for his careful reading of the first draft of this paper and numerous helpful suggestions. \\
\indent
The code for drawing post-critical curves (the so called bones) in \S\ref{S6.2} along with the codes for generating entropy contour plots in \S\ref{S5} have been written in MATLAB\texttrademark.  The latter are based on algorithms introduced in 
\cite{MR1002478, MR1151977}. All codes are available from author's \href{https://github.com/FilomKhash/Real-quadratic-rational-maps}{GitHub}. The bifurcation plots of \S\ref{S7} have been generated with \href{https://sourceforge.net/projects/detool/}{Dynamics Explorer}.

\section{The moduli space of real quadratic rational maps}\label{S2}
The main goal of the section is to develop a framework for studying the moduli space $\mathcal{M}_2(\Bbb{R})$ 
and its natural dynamical partition alluded to in \S\ref{S1} along with the real entropy function $h_\Bbb{R}$ defined on a Zariski open subset of $\mathcal{M}_2(\Bbb{R})$. This will enable us to formulate the monotonicity question as appeared in Question \ref{monotonicity formulation}. The discussion here can be mostly considered as a special case of the more general one in \cite[\S2]{filom2021real}
concerning  the real entropy function  on the moduli space of real rational maps of an arbitrary degree $d\geq 2$.
  
\subsection{Background on the moduli space $\mathcal{M}_2$}\label{S2.1}
The moduli space 
$\mathcal{M}_2:={\rm{Rat}}_2/{\rm{PSL}}_2$ 
of quadratic rational maps has been studied thoroughly.  In \cite{MR1246482}, Milnor treats this space chiefly as the complex orbifold $\mathcal{M}_2(\Bbb{C})={\rm{Rat}}_2(\Bbb{C})/{\rm{PSL}}_2(\Bbb{C})$ which is the space of  Möbius conjugacy classes of quadratic rational maps and identifies it with the plane $\Bbb{C}^2$ by dynamical means. Silverman later constructed the general moduli space $\mathcal{M}_d:={\rm{Rat}}_d/{\rm{PSL}}_2$  
as an affine integral scheme over $\Bbb{Z}$ by means of the geometric invariant theory and showed that in particular, $\mathcal{M}_2$
is isomorphic to the affine scheme $\Bbb{A}^2_\Bbb{Z}$ \cite{MR1635900}. 

The identification $\mathcal{M}_2(\Bbb{C})=\Bbb{C}^2$ is not hard to explain. A quadratic rational map $f\in\Bbb{C}(z)$
admits three fixed points (counted with multiplicity) whose multipliers $\mu_1,\mu_2,\mu_3$ satisfy the 
\textit{fixed point formula} 
\begin{equation}\label{fixed point formula}
\frac{1}{1-\mu_1}+\frac{1}{1-\mu_2}+\frac{1}{1-\mu_3}=1.
\end{equation} 
Forming symmetric functions of these multipliers  
\begin{equation}\label{symmetric functions}
\sigma_1=\mu_1+\mu_2+\mu_3,\quad \sigma_2=\mu_1\mu_2+\mu_2\mu_3+\mu_3\mu_1,\quad \sigma_3=\mu_1\mu_2\mu_3,
\end{equation}
the formula above amounts to $\sigma_3=\sigma_1-2$ and therefore, there is an embedding 
$$
\langle f\rangle\mapsto \left(\sigma_1(f),\sigma_2(f),\sigma_3(f)\right)
$$
of $\mathcal{M}_2(\Bbb{C})$ into the affine space $\Bbb{C}^3$ as the hyperplane $z=x-2$. The first two components $\sigma_1,\sigma_2$ will be repeatedly used to identify $\mathcal{M}_2(\Bbb{C})$ with the plane $\Bbb{C}^2$.

\subsection{Normal forms}\label{S2.2}
The paper \cite{MR1246482} utilizes various normal forms to investigate quadratic rational maps. This includes the 
\textit{fixed-point normal form}
\begin{equation}\label{fixed-point normal form}
\frac{z^2+\mu_1z}{\mu_2z+1}\quad (\mu_1\mu_2\neq 1)
\end{equation}
where $\mu_1, \mu_2$ are the multipliers of fixed points $0, \infty$. There is a single conjugacy class missed here which is that of the map 
$z\mapsto z+\frac{1}{z}$ with exactly one fixed point.  There is also the \textit{critical normal form} 
\begin{equation}\label{critical normal form}
\frac{\alpha z^2+\beta}{\gamma z^2+\delta}\quad (\alpha\delta-\beta\gamma=1)
\end{equation}
in which the critical points are $0,\infty$. Finally, there is the \textit{mixed normal form} 
\begin{equation}\label{mixed normal form}
\frac{1}{\mu}\left(z+\frac{1}{z}\right)+a\quad (\mu\neq 0)
\end{equation}
in which the critical points $\pm 1$ and the fixed point $\infty$ of multiplier $\mu\neq 0$ are specified. For future applications, we record the coordinates $\left(\sigma_1,\sigma_2\right)$ of the conjugacy class of this map in the moduli space $\mathcal{M}_2(\Bbb{C})\cong\Bbb{C}^2$ (\cite[Appendix C]{MR1246482}):
\begin{equation}\label{coordinates}
\begin{cases}
\sigma_1=\mu(1-a^2)-2+\frac{4}{\mu}\\
\sigma_2=\left(\mu+\frac{1}{\mu}\right)\sigma_1-\left(\mu^2+\frac{2}{\mu}\right)
\end{cases}.
\end{equation}

\subsection{The real moduli space $\mathcal{M}_2(\Bbb{R})$}\label{S2.3}
Let us now concentrate on the case of $f(z)\in\Bbb{R}(z)$ being a real quadratic rational map. There are descriptions of $\sigma_i$'s in terms of coefficients of the rational map $f$, and thus the conjugacy class of $f$ determines a point of $\mathcal{M}_2(\Bbb{R})$. 
As we shall see shortly, conversely any point of $\mathcal{M}_2(\Bbb{C})\cong\Bbb{C}^2$ with real coordinates can be represented by a real map.\footnote{That is to say, the \textit{field of moduli} $\Bbb{R}$ is a \textit{field of definition} too. This may fail in higher degrees; see  \cite[\S2]{filom2021real} for a detailed discussion on this issue.} We are interested in the dynamics of $f\restriction_{\hat{\Bbb{R}}}:\hat{\Bbb{R}}\rightarrow\hat{\Bbb{R}}$ which is invariant only under real conjugacies, i.e. conjugacy by elements of ${\rm{PGL}}_2(\Bbb{R})$. Thus there is  the question of whether multiple real conjugacy classes can lie in a single complex conjugacy class or not. The next proposition addresses all these issues and expands upon the discussion at the beginning of \cite[\S10]{MR1246482}:
\begin{proposition}\label{classification}
An orbit of the action of ${\rm{PGL}}_2(\Bbb{C})$ on the space of quadratic rational maps ${\rm{Rat}}_2(\Bbb{C})$ contains an element of ${\rm{Rat}}_2(\Bbb{R})$ if and only if  the corresponding point in $\mathcal{M}_2(\Bbb{C})\cong\Bbb{C}^2$ lies in the real locus
$\mathcal{M}_2(\Bbb{R})\cong\Bbb{R}^2$. Furthermore, the aforementioned ${\rm{PGL}}_2(\Bbb{C})$-orbit has a real representative from one of the following families:
\begin{equation} \label{families}
\{z^2+c\}_{c>\frac{1}{4}}; \quad
\left\{\frac{1}{\mu}\left(z+\frac{1}{z}\right)+a\right\}_{\mu\in\Bbb{R}-\{0\}, a\geq 0}; \quad
\left\{\frac{1}{\mu}\left(z-\frac{1}{z}\right)+b\right\}_{\mu\in\Bbb{R}-\{0\}, b\geq 0}.
\end{equation}
These families remain disjoint under the action of ${\rm{PGL}}_2(\Bbb{R})$. Also it might be the case that a 
${\rm{PGL}}_2(\Bbb{C})$-orbit contains two of these maps which are not ${\rm{PGL}}_2(\Bbb{R})$-conjugate. This occurs precisely for points of the symmetry locus, that is for a conjugacy class such as 
$\left\langle\frac{1}{\mu}\left(z+\frac{1}{z}\right)\right\rangle=\left\langle\frac{1}{\mu}\left(z-\frac{1}{z}\right)\right\rangle\,(\mu\in\Bbb{R}-\{0\})$ where maps $\frac{1}{\mu}(z\pm\frac{1}{z})$ are conjugate only over $\Bbb{C}$. 
\end{proposition}

\begin{proof}
{Pick a map $f(z)\in{\rm{Rat}}_2(\Bbb{C})$ for which $\sigma_1,\sigma_2$ and therefore $\sigma_3$ (as defined in \eqref{symmetric functions}) are real. Multipliers of fixed points are roots of the real cubic equation $x^3-\sigma_1x^2+\sigma_2x-\sigma_3=0$. Thus there is a fixed point whose multiplier is a real number, say $\mu\in\Bbb{R}$. When  $\mu=0$, there is a super-attracting fixed point and thus  $f(z)$ is conjugate to a  polynomial of the form $z^2+c$ where $\sigma_1=2, \sigma_2=4c$. Hence $c\in\Bbb{R}$. In other cases where the original $\mu$ is non-zero or $c\leq\frac{1}{4}$ in the polynomial $z^2+c$, there has to be more than one real fixed point and thus there is a real fixed point of non-zero multiplier. So aside from conjugacy classes of maps in $\{z^2+c\}_{c>\frac{1}{4}}$, there is an element of $\langle f\rangle$
in  the mixed normal form \eqref{mixed normal form}  $\frac{1}{\mu}\left(z+\frac{1}{z}\right)+a$ where $\mu\in\Bbb{R}-\{0\}$. But then the first equation of \eqref{coordinates} and $\sigma_1\in\Bbb{R}$ indicate that $a^2\in\Bbb{R}$. 
We should have either $a\in\Bbb{R}$ or $a\in{\rm{i}}\Bbb{R}$. In the former possibility, after conjugation with $z\mapsto -z$ if necessary, we can assume $a\geq 0$ and so we are dealing with a map from the second family in \eqref{families}. In the latter case, we conjugate with $z\mapsto {\rm{i}}z$ to get to a map of the form 
$\frac{1}{\mu}\left(z-\frac{1}{z}\right)+b$ from the third family where $b={\rm{i}}a\in\Bbb{R}$. With the same argument, there is no loss of generality to assume $b\geq 0$. This concludes the proof of the first part of the proposition. \\
\indent
For the second part, notice that transformations $x\mapsto x^2+c$ or $x\mapsto \frac{1}{\mu}(x+\frac{1}{x})+a$ of $\hat{\Bbb{R}}$  are dynamically quite different from $x\mapsto \frac{1}{\mu'}(x-\frac{1}{x})+b$  as the first two have two critical points  on $\hat{\Bbb{R}}$ while the other one is an unramified two-sheeted covering of the circle due to the absence of critical points; so they cannot be conjugated even via a homeomorphism of $\hat{\Bbb{R}}$.  Lastly, a map of the form $\frac{1}{\mu}\left(z+\frac{1}{z}\right)+a$ is not ${\rm{PGL}}_2(\Bbb{R})$-conjugate with a polynomial $z^2+c$ with $c>\frac{1}{4}$ since, unlike the former, this polynomial does not admit any fixed point of non-zero real multiplier. \\
\indent
Finally, suppose  for a map $f$ from one of the families in \eqref{families} the complex conjugacy class 
$\langle f\rangle$  contains two distinct real conjugacy classes. This implies the existence of a Möbius transformation 
$\alpha\in {\rm{PGL}}_2(\Bbb{C})-{\rm{PGL}}_2(\Bbb{R})$ for which $\alpha\circ f\circ\alpha^{-1}\in\Bbb{R}(z)$. But then taking complex conjugates implies that 
$\bar{\alpha}\circ f\circ\overline{\alpha^{-1}}=\alpha\circ f\circ\alpha^{-1}$.  
Hence the non-identity Möbius map  $\alpha^{-1}\circ\bar{\alpha}$ lies in ${\rm{Aut}}(f)$. We conclude that 
$\langle f\rangle$ is on the symmetry locus in $\mathcal{M}_2(\Bbb{C})$ -- the locus determined by maps which admit non-trivial automorphisms. It has been established in \cite[\S5]{MR1246482} that such maps have to be conjugate to a map of the form $k\left(z+\frac{1}{z}\right)$ where $k\in\Bbb{C}-\{0\}$. Again, $z\mapsto k\left(z+\frac{1}{z}\right)$ should have a fixed point of real multiplier. But other than $\infty$ which is of multiplier $\frac{1}{k}$, other fixed points are 
$\pm\sqrt{\frac{k}{1-k}}$ with the common multiplier  $2k-1$. Hence $k\in\Bbb{R}$ and $\langle f\rangle=\left\langle k(z+\frac{1}{z})\right\rangle$.
 }
\end{proof}

Proposition \ref{classification} allows us to formulate the monotonicity question concerning connectedness of the isentropes in the moduli space $\mathcal{M}_2(\Bbb{R})$ of real quadratic maps which is identified with the plane $\Bbb{R}^2$ via 
$\left(\sigma_1,\sigma_2\right)$.  Here, any point  can be presented by a real quadratic map  uniquely up to the action ${\rm{PGL}}_2(\Bbb{R})$ unless the point is on the symmetry locus $\mathcal{S}(\Bbb{R})$.
The dynamics of the restriction to the circle $\hat{\Bbb{R}}$ is uniquely determined up to a real Möbius conjugacy, and so all elements of the class $\langle f\rangle$ are of the same real entropy (as defined in \eqref{real entropy}). 
Hence, the following is a well defined  real entropy function on the complement of $\mathcal{S}(\Bbb{R})$ in
 $\mathcal{M}_2(\Bbb{R})$:
\begin{equation}\label{function}
\begin{cases}
h_{\Bbb{R}}:\mathcal{M}_2(\Bbb{R})-\mathcal{S}(\Bbb{R})\rightarrow \left[0,\log(2)\right]\\
\left(\sigma_1,\sigma_2\right)\mapsto h_{\rm{top}}\left(f\restriction_{\hat{\Bbb{R}}}:\hat{\Bbb{R}}\rightarrow\hat{\Bbb{R}}\right)
\end{cases}
\end{equation}
where $f\in\Bbb{R}(z)$ is a real representative of the conjugacy class 
$\left(\sigma_1,\sigma_2\right)\in\mathcal{M}_2(\Bbb{C})\cong\Bbb{C}^2$. This function is moreover continuous in the analytic topology by the result of \cite{MR1372979}.

The following example attests that for getting a single-valued entropy function on $\mathcal{M}_2(\Bbb{R})$,  it is absolutely necessary to exclude the symmetry locus.

\begin{example}\label{symmetry locus entropy}
Given $\mu\in\Bbb{R}-\{0\}$, real quadratic rational  maps $\frac{1}{\mu}\left(z\pm\frac{1}{z}\right)$   are conjugate via 
$z\mapsto{\rm{i}}z$ but exhibit quite different dynamical behavior on the real circle. The critical points of 
$\frac{1}{\mu}\left(z-\frac{1}{z}\right)$ are not real, so  
it induces a degree two covering 
$x\mapsto\frac{1}{\mu}\left(x-\frac{1}{x}\right)$ of $\hat{\Bbb{R}}$
whose entropy is therefore $\log(2)$.\footnote{
In general,  degree $d$ coverings of a circle are always of topological entropy $\log(d)$ \cite[Theorem 1$^{'}$]{MR579440}.}
On the other hand, the topological  entropy of 
$x\mapsto\frac{1}{\mu}\left(x+\frac{1}{x}\right)$ vanishes: 
for $|\mu|\leq 1$ every orbit is attracted to the fixed point $\infty$ of multiplier $\mu$; for $\mu>1$ orbits in the invariant interval $(0,\infty)$ tend to the attracting fixed point $\frac{1}{\sqrt{\mu-1}}$ while those in the invariant interval 
$(-\infty,0)$ tend to the attracting fixed point $-\frac{1}{\sqrt{\mu-1}}$; and finally, for $\mu<-1$ there is no finite real fixed point and any point of $\hat{\Bbb{R}}$, other than the fixed point $\infty$ and its preimage $0$,  converges under iteration to the $2$-cycle consisting of $\pm\frac{1}{\sqrt{-\mu-1}}$ whose multiplier 
is $\left(\frac{2+\mu}{\mu}\right)^2<1$. 
\end{example}

\subsection{Partitioning $\mathcal{M}_2(\Bbb{R})$}\label{S2.4}
\begin{figure}[ht!]
\center
\includegraphics[width=15cm]{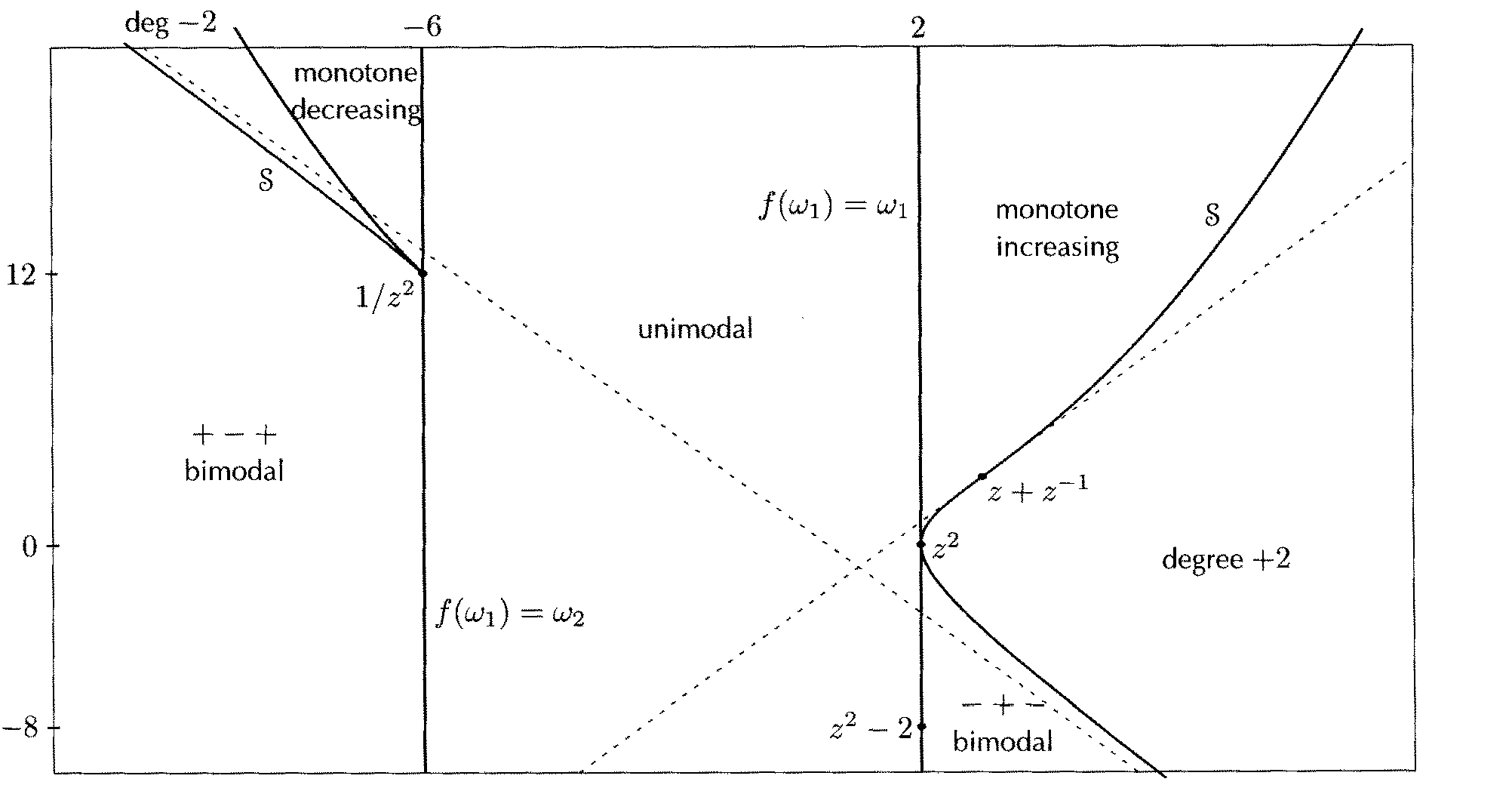}
\caption{The real moduli space $\mathcal{M}_2(\Bbb{R})$ as illustrated in  \cite[Figure 15]{MR1246482}. The post-critical lines $\sigma_1=-6,2$ ($\omega_1,\omega_2$ denote the critical points of $f$ here), the dotted lines ${\rm{Per}}_1(\pm 1)$, the real symmetry locus $\mathcal{S}(\Bbb{R})$ and the partition into seven regions according to the various types of the dynamics induced on $\hat{\Bbb{R}}$ are shown. The component of degree zero maps in 
$\mathcal{M}_2(\Bbb{R})-\mathcal{S}(\Bbb{R})$ is the union of monotonic, unimodal and bimodal regions that overlap only along the lines $\sigma_1=-6,2$.}
\label{fig:main}
\end{figure}

Going back to quadratic maps, let us briefly comment on the domain of $h_{\Bbb{R}}$ in \eqref{function}. Here is a parametric equation (adapted from \cite[p. 47]{MR1246482}) for the symmetry locus 
$\mathcal{S}(\Bbb{C})=\left\{\left\langle k(z+\frac{1}{z})\right\rangle \mid k\in\Bbb{C}-\{0\}\right\}$ in the $\left(\sigma_1,\sigma_2\right)$-plane:
\begin{equation}\label{symmetry locus}
\begin{cases}
\sigma_1=4k-2+\frac{1}{k}\\
\sigma_2=4k^2-4k+5-\frac{2}{k}
\end{cases}.
\end{equation}
As $k$ varies in $\Bbb{R}-\{0\}$, one gets a parametrization of the set of real points on this curve. There are two real components, one with $\sigma_1\leq -6$ parametrized with $k<0$ and the other with $\sigma_1\geq 2$ parametrized with 
$k>0$.  Thus $\mathcal{M}_2(\Bbb{R})-\mathcal{S}(\Bbb{R})$ has three connected components 
in the real $\left(\sigma_1,\sigma_2\right)$-plane; we shall see shortly in Proposition \ref{Julia circle} that only one of them matters:  the component not containing covering maps of degrees $\pm 2$ but determined by maps for which the topological degree on $\hat{\Bbb{R}}$ is zero. We henceforth refer to this component as the component of degree zero maps in $\mathcal{M}_2(\Bbb{R})-\mathcal{S}(\Bbb{R})$. 
This is all demonstrated in Figure \ref{fig:main}  that illustrates the real moduli space $\mathcal{M}_2(\Bbb{R})$ in $\left(\sigma_1,\sigma_2\right)$-coordinates. There are conspicuous vertical lines 
$\sigma_1=2,\sigma_1=-6$ cut off by the post-critical relations $f(c)=c$ and $f(c_1)=c_2$ respectively. The curves 
$${\rm{Per}}_{1}(1): 2\sigma_1-\sigma_2=3,\quad {\rm{Per}}_{1}(-1)={\rm{Per}}_{2}(1):2\sigma_1+\sigma_2=1,$$ 
are defined by the existence of certain parabolic cycles and are visible as dotted skew lines in Figure \ref{fig:main}; the detailed derivation of their equations can be found in \cite[\S 3]{MR1246482}.\\
\indent
The number of real fixed points will play a vital role in our analysis of the dynamics induced on the real circle; see observation \hyperref[f]{4.f} and Lemma \ref{attracting fixed point} in \S\ref{S4}. The transition occurs as one crosses the line ${\rm{Per}}_{1}(1)$
or the symmetry locus $\mathcal{S}(\Bbb{R})$. The former amounts to a qualitative change in the dynamics of the ambient map on the Riemann sphere, whereas the latter is simply due to the ambiguity in the choice of real representatives for points on $\mathcal{S}(\Bbb{R})$. The more delicate partition of the component of degree zero maps to five other regions apparent in Figure \ref{fig:main} will be addressed later in this subsection.

Given $f\in\Bbb{R}(z)$,  the entropy of the restriction 
$f\restriction_{\hat{\Bbb{R}}}$ solely depends on how the forward-invariant circle $\hat{\Bbb{R}}$  meets the Julia set $\mathcal{J}(f)$ of the ambient map  because  $\mathcal{J}(f)\cap \hat{\Bbb{R}}$ contains  the non-wandering set of $f\restriction_{\hat{\Bbb{R}}}$ aside from those Fatou points on $\hat{\Bbb{R}}$ which are periodic or in a rotation domain; and of course such non-wandering points of  $f\restriction_{\hat{\Bbb{R}}}$ do not contribute to its entropy.  
Consequently, having a criterion for a  closed subset such as $\hat{\Bbb{R}}$ to contain the Julia set would be convenient. 
\begin{lemma}\label{criterion}
Let $f$ be a rational map and $C$ a closed subset satisfying $f^{-1}(C)\subseteq C$ and $|C|\geq 3$. 
Then 
$\mathcal{J}(f)\subseteq C$.
\end{lemma}

\begin{proof}
An immediate consequence of Montel's theorem.
\end{proof}

\begin{proposition}\label{Julia circle}
Let $f$ be a real quadratic rational map of degree two. The restriction
$f\restriction_{\hat{\Bbb{R}}}:\hat{\Bbb{R}}\rightarrow\hat{\Bbb{R}}$
is surjective if and only if critical orbits do not collide with the circle $\hat{\Bbb{R}}=\Bbb{R}\cup\{\infty\}$ or equivalently, if and only if the restriction is a covering map. In such a situation the Julia set is either the whole circle $\hat{\Bbb{R}}$ or a Cantor set contained in it and in particular, $h_\Bbb{R}(f)=\log(2)$.
\end{proposition}

\begin{proof}
If there is no critical point located on $\hat{\Bbb{R}}$, the proper map 
$f\restriction_{\hat{\Bbb{R}}}:\hat{\Bbb{R}}\rightarrow\hat{\Bbb{R}}$ would be a local homeomorphism and hence a covering map. The degree must be $\pm 2$ as for any $p\in\hat{\Bbb{R}}$, $f^{-1}(p)\cap \hat{\Bbb{R}}$ is either empty or of size two. Since $\deg f=2$, this in particular implies that $f$ is surjective and $f^{-1}(\hat{\Bbb{R}})=\hat{\Bbb{R}}$; so that the forward iterates of critical points cannot belong to $\hat{\Bbb{R}}$ either. As the closed subset $\hat{\Bbb{R}}$ of $\hat{\Bbb{C}}$ is backward-invariant,  Lemma \ref{criterion} implies that $\mathcal{J}(f)$ is included in  $\hat{\Bbb{R}}$. So 
$h_\Bbb{R}(f)=h_{\rm{top}}\left(f\restriction_{\hat{\Bbb{R}}}:\hat{\Bbb{R}}\rightarrow\hat{\Bbb{R}}\right)\geq
h_{\rm{top}}\left(f\restriction_{\mathcal{J}(f)}:\mathcal{J}(f)\rightarrow\mathcal{J}(f)\right)=\log(2)$ 
is the highest possible value $\log(2)$.
The subset $\mathcal{J}(f)$,  just like any Julia set, is perfect and compact. If it is not totally disconnected, it has to contain a connected subset of $\hat{\Bbb{R}}$ with more than one element and therefore a non-degenerate open arc of this circle. So we might pick an open subset $U$ of $\hat{\Bbb{C}}$ whose intersection with the $f$-invariant subset $\hat{\Bbb{R}}$ is contained in $\mathcal{J}(f)$. But then  Montel's theorem implies that the union $\bigcup_n f^n(U)$
misses at most two points of the Riemann sphere. Taking the intersection with the backward-invariant set $\hat{\Bbb{R}}$, we conclude that the compact subset $\mathcal{J}(f)$ of $\hat{\Bbb{R}}$ omits at most two point of this circle and hence coincides with it.\\
\indent
In the presence of critical points and thus critical values on $\hat{\Bbb{R}}$, $f(\hat{\Bbb{R}})$ has to be a proper subset of $\hat{\Bbb{R}}$. Assume the contrary. If a critical point of $f$ is real, the same must be true for the other as they are roots of a quadratic equation with real coefficients. So all critical points and values lie on $\hat{\Bbb{R}}$ and $f$ induces an unramified covering 
$\hat{\Bbb{R}}-\text{critical points}\rightarrow \hat{\Bbb{R}}-\text{critical values}$. Both the range and the domain have  two components each homeomorphic with the contractible space $(0,1)$. Thus, $f$ maps the components of the domain onto the components of the range homeomorphically and is hence bijective. This cannot happen since $f$ is two-to-one over an open subset of $\hat{\Bbb{R}}$.
\end{proof}
\noindent
This proposition is  useful especially because it implies that in Question \ref{monotonicity formulation} only the component of degree zero maps matters; the component where the critical points are real and,  
instead of circle maps $f\restriction_{\hat{\Bbb{R}}}:\hat{\Bbb{R}}\rightarrow\hat{\Bbb{R}}$, 
we are dealing with interval maps $f\restriction_{f(\hat{\Bbb{R}})}:f(\hat{\Bbb{R}})\rightarrow f(\hat{\Bbb{R}})$  since in this situation $f(\hat{\Bbb{R}})$ would be a proper closed arc of the circle $\hat{\Bbb{R}}$. The interval map is of  the same entropy as the points of  $\hat{\Bbb{R}}-f(\hat{\Bbb{R}})$ are wandering and thus do not contribute to the topological entropy of $f\restriction_{\hat{\Bbb{R}}}$; hence the reformulation \eqref{real entropy 1} of \eqref{real entropy}.

The following observation will come in handy in studying the dynamics of such interval maps 
$f\restriction_{f(\hat{\Bbb{R}})}:f(\hat{\Bbb{R}})\rightarrow f(\hat{\Bbb{R}})$:
\begin{lemma}\label{Schwarzian}
The Schwarzian derivative of a real quadratic rational map with real critical points is negative. 
\end{lemma}

\begin{proof}
Let $f\in{\rm{Rat}}_2(\Bbb{R})$ be such a map. According to Proposition \ref{classification}, the system 
$f\restriction_{\hat{\Bbb{R}}}:\hat{\Bbb{R}}\rightarrow\hat{\Bbb{R}}$ is ${\rm{PGL}}_2(\Bbb{R})$-conjugate to a map of the form either  $\frac{1}{\mu}(x+\frac{1}{x})+a$ or $x^2+c$
where all the coefficients are real.  Recall that such a conjugacy does not change the sign of the Schwarzian derivative, so it suffices to verify that the Schwarzian derivatives of $\frac{1}{\mu}(x+\frac{1}{x})+a$ and $x^2+c$ are always negative: they are given by $-\frac{6}{(x^2-1)^2}$ and $-\frac{3}{2x^2}$ respectively.
\end{proof}

\begin{figure}[ht!]
\center
\includegraphics[width=15cm]{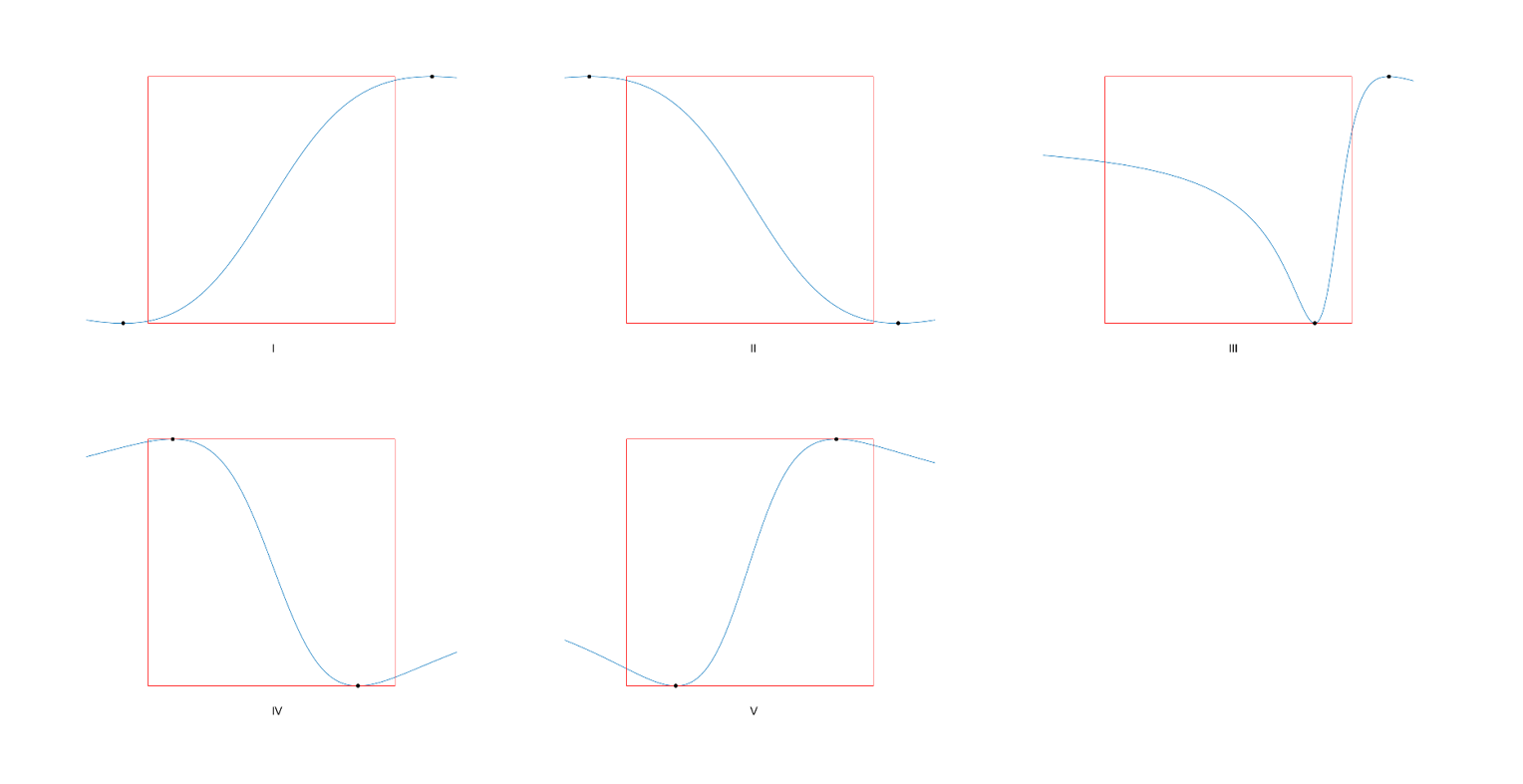}
\caption{Different types of dynamics that a real quadratic map $f$ with real critical points can induce on the compact invariant interval $f(\hat{\Bbb{R}})$. Here, the red square is built on the interval $f(\hat{\Bbb{R}})\subset\Bbb{R}$ and the critical points are marked in black on the graph of $f$. Based on how they are situated with respect to $f(\hat{\Bbb{R}})$, we call the corresponding point $\langle f\rangle$ to be in the
 I) monotone increasing, II) monotone decreasing, III) unimodal, 
IV) $(+-+)$-bimodal or V) $(-+-)$-bimodal region of  the moduli space $\mathcal{M}_2(\Bbb{R})$. 
See Figure \ref{fig:main} for these regions.}
\label{fig:graphs}
\end{figure}

In the component of degree zero maps in $\mathcal{M}_2(\Bbb{R})-\mathcal{S}(\Bbb{R})$, the positions of the two critical points  of $f$ on $\hat{\Bbb{R}}$ with respect to $f(\hat{\Bbb{R}})$ determine the number of laps. So, aside from the cases where $f\restriction_{\hat{\Bbb{R}}}:\hat{\Bbb{R}}\rightarrow\hat{\Bbb{R}}$ is a covering map of degree $\pm 2$, five dynamically distinct behaviors can  occur: the interval map $f\restriction_{f(\hat{\Bbb{R}})}:f(\hat{\Bbb{R}})\rightarrow f(\hat{\Bbb{R}})$
is either monotone increasing, monotone decreasing, unimodal, $(+-+)$-bimodal or 
$(-+-)$-bimodal; see Figure \ref{fig:graphs} for these possibilities. This  classification, outlined in \cite[\S10]{MR1246482}, is apparent in Figure \ref{fig:main}.

\textbf{Caution.} To avoid any confusion, we should point out that we will treat these ``regions'' as closed subsets of 
$\mathcal{M}_2(\Bbb{R})-\mathcal{S}(\Bbb{R})$; so the polynomial line $\sigma_1=2$ and the $f(c_1)=c_2$
line $\sigma_1=-6$ are in more than one of these regions. In case that we want to avoid the boundary lines between them, we use the term ``open region'' instead.

According to Proposition \ref{Julia circle}, the real entropy is ${\rm{log}}(2)$ in covering components. Among the other cases, the real entropy vanishes for monotonic maps. Hence, from the entropy perspective, the only interesting regions of the moduli space $\mathcal{M}_2(\Bbb{R})$ are unimodal and bimodal regions. 

 In view of this partitioning of the real moduli space, it would be useful for our purposes to have a general way to get from the monotonicity of a continuous function to the monotonicity of its restrictions and vice versa. This is the content of the  following straightforward point-set topology lemma that concludes this section:
\begin{lemma}\label{point-set topology}
Let $X$ be a topological space and $h:X\rightarrow\Bbb{R}$ a continuous function. 
\begin{enumerate}
\item If $h$ is monotonic, then so is its restriction $h\restriction_{A}$ to a closed subspace $A$ provided that the smaller restriction $h\restriction_{\partial A}$ is monotonic.
\item Assuming that $X$ is written as the union $A\cup B$ of two arbitrary subsets,  the monotonicity of $h:X\rightarrow\Bbb{R}$  can be inferred from that of 
$h\restriction_A$ and $h\restriction_B$ 
provided that $h(A\cap B)=h(A)\cap h(B)$. 
\end{enumerate}
\end{lemma}

\begin{proof}
For the first claim, we should show that for any real number $c$ the space $h^{-1}(c)\cap A$ is connected. If this space is written as $A_1\sqcup A_2$ with $A_1,A_2$ open (and hence closed) subsets of $A$ and $h^{-1}(c)$, then either $A_1\cap\partial A$ or 
$A_2\cap\partial A$ should be empty as $h^{-1}(c)\cap\partial A$ is connected. If for instance $A_1\cap\partial A=\emptyset$, then $A_1\subseteq{\rm{int}}(A)$; so the open subset $A_1$ of $h^{-1}(c)\cap A$ would be contained in 
$h^{-1}(c)\cap{\rm{int}}(A)$  which is itself open in $h^{-1}(c)$.  Hence $A_1$ is a  simultaneously open and closed subset of  the connected subspace $h^{-1}(c)$. Therefore, it is either vacuous or contains the whole $h^{-1}(c)\cap A$.\\
\indent
For the second part, just notice that a level set $h^{-1}(c)$ may be written as the union of connected subspaces 
$h^{-1}(c)\cap A, h^{-1}(c)\cap B$. If one of them is vacuous, we are done. Otherwise, $c\in h(A)\cap h(B)$ and thus, by the assumption, the connected subspaces $h^{-1}(c)\cap A$ and $h^{-1}(c)\cap B$  should have a point of $A\cap B$ in common; therefore, their union $h^{-1}(c)$ must be connected as well. 
\end{proof}

\section{Hyperbolic components}\label{S3}
In studying the entropy behavior of families, the notion of hyperbolicity appears naturally; e.g., the entropy of the quadratic polynomial family \eqref{quadratic family}, although increasing, is constant over the ``hyperbolic windows''. We establish an analogous result for quadratic rational maps in this section. In particular, the constancy of $h_\Bbb{R}$ over the 
\textit{escape component} will be useful later due to the fact that it reduces the discussion to the part of $\mathcal{M}_2(\Bbb{R})$ where the Julia set is connected, and this sets the stage for invoking the theory of polynomial-like maps in \S\ref{S6}. 

\subsection{Entropy behavior over real hyperbolic components}\label{S3.1}
There is a discussion in \cite[\S7]{MR1246482} on \textit{hyperbolic components} of the moduli space of quadratic rational maps. Recall that a rational map is called hyperbolic if each of its critical orbits converges to an attracting periodic orbit; see 
\cite[Theorem 19.1]{MR2193309} for equivalent characterizations.  It turns out that there are four different classes of hyperbolic quadratic maps:
\begin{itemize}
\item\textbf{Type B: Bitransitive.} Both critical orbits converge to the same attracting periodic orbit but critical points are in immediate basins of different points of this orbit. 
\item\textbf{Type C: Capture.} Only one critical point lies in the immediate basin of an attracting periodic point and the other critical orbit eventually lands there.
\item\textbf{Type D: Disjoint Attractors.} The critical orbits converge to distinct periodic orbits.
\item\textbf{Type E: Escape.} Both critical orbits converge to the same attracting fixed point. 
\end{itemize}
  The paper \cite{MR1047139} investigates the topological types of hyperbolic components in the 
\textit{critically marked moduli space} $\mathcal{M}_2^{cm}(\Bbb{C})$. The corresponding topological types in the unmarked space $\mathcal{M}_2(\Bbb{C})$ can be deduced from that: every component of type \textbf{B}, \textbf{C} or \textbf{D} is homeomorphic to the open disk in $\Bbb{R}^4$; cf. \cite[p. 53]{MR1246482}. The escape case is in stark difference with the other types of hyperbolic components: there is a single hyperbolic component of type \textbf{E} denoted by $E\subset\mathcal{M}_2(\Bbb{C})$ which, unlike other components, does not admit a natural \textit{center}, i.e. a critically periodic map; and is homeomorphic to $\Bbb{D}\times (\Bbb{C}-\overline{\Bbb{D}})$ where $\Bbb{D}$ is the open unit disk in the plane \cite[Lemma 8.5]{MR1246482}.\\
\indent
A natural question related to monotonicity now arises: What is the behavior of $h_\Bbb{R}$ on the intersection of a hyperbolic component with the real locus? Is such a \textit{real hyperbolic component} necessarily included in a single isentrope?  
 The real locus $\mathcal{M}_2(\Bbb{R})\cong\Bbb{R}^2$ is the set of fixed points of an involution on 
$\mathcal{M}_2(\Bbb{C})\cong\Bbb{C}^2$ which takes $\left(\sigma_1,\sigma_2\right)$ to $(\bar{\sigma}_1,\bar{\sigma}_2)$, a transformation which on the level of rational maps is induced by conjugating coefficients in ${\rm{Rat}}_2(\Bbb{C})$. 
For future references, we record this involution as acting not only on rational maps but on any continuous transformation
 $h:\hat{\Bbb{C}}\rightarrow\hat{\Bbb{C}}$. The involution will henceforth be denoted by
\begin{equation}\label{involution}
h\mapsto \left(\tilde{h}:z\mapsto \overline{h(\bar{z})}\right).
\end{equation}
Of course, for a rational map $h\in\Bbb{C}(z)$, this is just the rational map $\bar{h}$ obtained from conjugating the coefficients. It is easy to check that $h\mapsto\tilde{h}$ respects the composition and the ring structure of the set of $\Bbb{C}$-valued functions on the Riemann sphere;
takes homeomorphisms to homeomorphisms; and finally, commutes with differential operators $\frac{\partial}{\partial z}$
and $\frac{\partial}{\partial\bar{z}}$:
\begin{equation}\label{differentiation}
\frac{\partial\tilde{h}}{\partial z}=\widetilde{\frac{\partial h}{\partial z}},\quad
\frac{\partial\tilde{h}}{\partial\bar{z}}=\widetilde{\frac{\partial h}{\partial\bar{z}}}.
\end{equation}

\begin{proposition}\label{classification of components}
A non-empty intersection with $\mathcal{M}_2(\Bbb{R})$  of a hyperbolic component of type \textbf{B}, \textbf{C} or \textbf{D} in $\mathcal{M}_2(\Bbb{C})$ is connected while the intersection of the hyperbolic component  $E\subset\mathcal{M}_2(\Bbb{C})$ of 
type \textbf{E} with the real locus has two connected components. In the former situation, the center of the component belongs to $\mathcal{M}_2(\Bbb{R})$ as well.
\end{proposition}
\begin{proof}
Following the suggestion on \cite[p. 15]{MR1181083}, one can invoke the \textit{Smith theory} from  algebraic topology to address this question:  by  \cite[chap. III, Theorem 7.11]{MR0413144},   the set of fixed points of an involution acting on a space homotopy equivalent to the sphere $S^n$ (resp. to a point) has the mod $2$ \v Cech cohomology of an $r$-sphere (resp. of a point) where $-1\leq r\leq n$, and $r=-1$ means there is no fixed point. 
Consequently, the real locus of any hyperbolic component of type \textbf{B}, \textbf{C} or \textbf{D} which intersects $\mathcal{M}_2(\Bbb{R})$ is connected due to the fact that these complex components are (as mentioned above) homeomorphic to a disk and hence contractible. On the other hand, for the type \textbf{E}  component 
$E\cong \Bbb{D}\times (\Bbb{C}-\overline{\Bbb{D}})$ which is of the homotopy type of $S^1$, the fixed point set  $E\cap\mathcal{M}_2(\Bbb{R})$ can have at most two connected components.  
We shall see shortly in Theorem \ref{entropy constant over hyperbolic} that $h_\Bbb{R}$ must be constant on any connected component of $E\cap\mathcal{M}_2(\Bbb{R})$.  One can easily check that different entropy values actually come up: 
the complement of the Mandelbrot set with respect to  the real line consists of two rays $\left(\frac{1}{4},\infty\right)$ and $(-\infty,-2)$; the Julia set of $x^2+c$ is completely real when $c<-2$ and is disjoint from the real axis for $c>\frac{1}{4}$ (compare with \cite[Problems 4-e and 4-f]{MR2193309} and Proposition \ref{escape locus}) with the values of real entropy given by $\log(2)$ and $0$   respectively. Consequently, the number of connected components of $E\cap\mathcal{M}_2(\Bbb{R})$ is precisely two. \\
\indent
For the last part, notice that the involution \eqref{involution} preserves the hyperbolicity and the post-critical finiteness of rational maps. Hence, given a hyperbolic component $\mathcal{U}\subset\mathcal{M}_2(\Bbb{C})$, its image $\widetilde{\mathcal{U}}$ under the involution is another complex hyperbolic component that should coincide with the original one if intersects it or equivalently, if $\mathcal{U}\cap\mathcal{M}_2(\Bbb{R})\neq\emptyset$. The complex conjugate of the center of $\mathcal{U}$ (if exists) is a PCF map in $\tilde{\mathcal{U}}$,
and hence when $\mathcal{U}=\tilde{\mathcal{U}}$  it should be the same as the center of $\mathcal{U}$ and thus real due to the fact that a hyperbolic component of $\mathcal{M}_2(\Bbb{C})$ contains at most one PCF map. 
\end{proof}

Now, we address the values of $h_\Bbb{R}$ on real hyperbolic components. We are only concerned with those hyperbolic components that intersect the component of degree zero maps in $\mathcal{M}_2(\Bbb{R})-\mathcal{S}(\Bbb{R})$ as 
$h_\Bbb{R}\equiv \log(2)$ outside that. 
\begin{theorem}\label{entropy constant over hyperbolic}
The function $h_\Bbb{R}$ is constant over each connected component of the intersection of a complex hyperbolic component  with the component of degree zero maps in $\mathcal{M}_2(\Bbb{R})-\mathcal{S}(\Bbb{R})$ with a value which is the logarithm of an algebraic number. 
\end{theorem}

\begin{proof}
{Consider the intersection of a hyperbolic component in $\mathcal{M}_2(\Bbb{C})$ with the real locus (or a connected component of such an intersection in the case of the component $\textbf{E}$) and denote this open subset by $\mathcal{U}$. 
The value of $h_\Bbb{R}$ at any point of $\mathcal{U}-\mathcal{S}(\Bbb{R})$ has to be the logarithm of an algebraic number; see Lemma \ref{algebraic} below.  By continuity, $h_\Bbb{R}$ maps any connected component of the intersection of $\mathcal{U}$ with the component of degree zero maps to an interval while the values attained by $h_\Bbb{R}$
on this connected component are logarithms of algebraic numbers. We conclude that this interval is a singleton. Finally, we need to show that if $\mathcal{S}(\Bbb{R})$ crosses
 $\mathcal{U}$, then these constant values of $h_\Bbb{R}$ on connected components of the intersection of $\mathcal{U}$
 with the component of degree zero maps coincide. This is due to the fact that $h_\Bbb{R}\to 0$ as one tends to $\mathcal{S}(\Bbb{R})$ along maps of degree zero; cf.  Example \ref{symmetry locus entropy}. So $h_\Bbb{R}$ would vanish over the intersection of $\mathcal{U}$ with the component of degree zero maps once the intersection is disconnected. }
\end{proof}

The proof above relied on the well known fact that the entropy of a hyperbolic map is the logarithm of an algebraic number.
\begin{lemma}\label{algebraic}
Let $I$ be an interval $[a,b]$ or a circle. Let $f:I\rightarrow I$ be a continuous multimodal self-map of $I$ whose turning points are attracted by periodic orbits. Then the number $\exp\left(h_{\rm{top}}(f)\right)$ is algebraic.
\end{lemma}
\noindent
A proof can be found for instance in \cite{filom2021real}.

\subsection{Real escape components}\label{S3.2}
The escape locus plays an important role in studying the dynamics of quadratic rational maps since, in the absence of parabolic fixed points, the Julia set of such a map is connected unless it belongs to the escape locus in which case the dynamics on the Julia set can be modeled by a one-sided $2$-shift \cite[Lemma 8.2]{MR1246482}. Switching to the real part of $E$, as mentioned in the discussion before Proposition \ref{classification}, $E\cap\mathcal{M}_2(\Bbb{R})$ has two components associated with real entropy values $0$ and $\log(2)$. Thinking about the complement of the Mandelbrot set in the polynomial line $\sigma_1=2$, 
it is clear that in Figure \ref{fig:main} the $h_\Bbb{R}\equiv 0$ \textit{escape component} lies above the dotted line 
$${\rm{Per}}_1(1): \sigma_2=2\sigma_1-3$$
of parabolic parameters while 
the $h_\Bbb{R}\equiv \log(2)$ \textit{escape component} lies below that; also see   \cite[Figure 16]{MR1246482}.
Not only the real entropy, but even the dynamics on the real circle itself is easy to describe in each of these cases:
\begin{proposition}\label{escape locus}
For a real quadratic rational map in the escape locus, the Julia set is either entirely contained in $\hat{\Bbb{R}}$ or completely disjoint from it. 
\end{proposition}

\begin{proof}
{Let $f$ be  such a map. If it is a covering map, then the Julia set is contained in $\hat{\Bbb{R}}=\Bbb{R}\cup\{\infty\}$ by Proposition \ref{Julia circle}. So suppose both critical points are on the real axis. Thus both critical orbits converge to a real attracting fixed point $p$. Denote the immediate basin of $p$ for the map 
$f\restriction_{\hat{\Bbb{R}}}:\hat{\Bbb{R}}\rightarrow\hat{\Bbb{R}}$ by $I$. 
If this open subinterval of $\hat{\Bbb{R}}$ coincides with $\hat{\Bbb{R}}$, then all points of $\hat{\Bbb{R}}$ are Fatou and we are done. Otherwise, it is no loss of generality to take the immediate basin $I$ to be a real interval of the form
$(\alpha,\beta)\subset\Bbb{R}$. Clearly, each of $\alpha,\beta$ lands on an orbit of period at most two after at most one iteration. Invoking Lemma \ref{Schwarzian}, the Schwarzian derivative of $f$ is negative, so $(\alpha,\beta)$  contains at least one critical point and hence at least one of the critical orbits.  We now claim that both critical values are contained in 
$(\alpha,\beta)$. Otherwise, there is at most one critical value there, namely $v$, and the preimage of $(\alpha,\beta)-\{v\}$ under the two-sheeted ramified covering $f:\hat{\Bbb{C}}\rightarrow\hat{\Bbb{C}}$ must have four connected components; each of them must be homeomorphic to $(0,1)$ and must have the unique element of $f^{-1}(v)$ in its closure. Since the coefficients of $f$ are real, those components that are not contained in $\hat{\Bbb{R}}$ have to appear in conjugate pairs disjoint from $\hat{\Bbb{R}}$. 
There is exactly one such a pair as $f^{-1}\left((\alpha,\beta)-\{v\}\right)$ contains $(\alpha,\beta)-f^{-1}(v)$, and so cannot be  disjoint from $\hat{\Bbb{R}}$; and also cannot be completely contained in it since in that case we have four disjoint open real intervals whose closures share a point.  Consequently, the two real components of $f^{-1}\left((\alpha,\beta)-\{v\}\right)$ union the singleton set $f^{-1}(v)$ form an interval which is the real preimage 
$\left(f\restriction_{\hat{\Bbb{R}}}\right)^{-1}\left((\alpha,\beta)\right)$  of $I=(\alpha,\beta)$. This has to be contained in a component of the basin of $p$ under the map $f\restriction_{\hat{\Bbb{R}}}:\hat{\Bbb{R}}\rightarrow\hat{\Bbb{R}}$. But it already contains $p\in f^{-1}(p)$, so must be contained in the immediate basin $I=(\alpha,\beta)$. As a corollary, this interval is completely invariant for the map $f\restriction_{\hat{\Bbb{R}}}:\hat{\Bbb{R}}\rightarrow\hat{\Bbb{R}}$. This is a contradiction because the other critical orbit has to eventually land at this immediate basin while we have assumed the other critical value does not lie there.\\
\indent
It has been established so far that $I=(\alpha,\beta)$ has both of critical values. We claim that this indicates that the closed real subset ${\hat{\Bbb{R}}}-I$  is backward-invariant under the map $f:\hat{\Bbb{C}}\rightarrow\hat{\Bbb{C}}$, and so must contain the Julia set according to Lemma \ref{criterion}: As this closed interval does not contain any critical value of $f$, its preimage consists of two components homeomorphic to $[0,1]$ which, by the same argument, are either both real or complex conjugate and disjoint from $\hat{\Bbb{R}}$. The latter is impossible since at least one of the endpoints $\alpha$ or $\beta$ of ${\hat{\Bbb{R}}}-I$ has a real point in its preimage. Thus 
$f^{-1}\left({\hat{\Bbb{R}}}-I\right)\subset \hat{\Bbb{R}}$. As a matter of fact, $f^{-1}\left({\hat{\Bbb{R}}}-I\right)$ cannot intersect the $f$-invariant set $I$, and so is a subset of ${\hat{\Bbb{R}}}-I$. This concludes the proof.}
\end{proof}

\begin{remark}\label{higher degrees}
Proposition \ref{escape locus} is false for $d\geq 3$:  For $c>1$ and $d\geq 3$, $z=c$ is a repelling fixed point of the map $f(z)=z^d+c-c^d$ and hence the Julia set is not disjoint from the real axis while it cannot be totally real due to the fact that it is symmetric with respect to the rotation $z\mapsto{\rm{e}}^{\frac{2\pi{\rm{i}}}{d}}z$ which does not preserve $\hat{\Bbb{R}}$. For $c>1$ large enough, the unique finite critical value $c-c^d$ lies in the immediate basin of   the super-attracting fixed point at infinity, and the  map then belongs to the \textit{polynomial escape locus}. Notice that the induced real map $x\mapsto x^d+c-c^d$ has at most one turning point ($x=0$ for $d$ even) and its entropy is thus at most $\log(2)<\log(d)$.
\end{remark}

The description of the real dynamics in the real part of the escape locus in Proposition \ref{escape locus} results in a simple description of the boundary the real escape locus in the component of degree zero maps.
Milnor alludes to this  in  \cite[pp. 66,67]{MR1246482}. We include a short proof here:

\begin{proposition}\label{escape boundary}
In the component of degree zero maps in  $\mathcal{M}_2(\Bbb{R})-\mathcal{S}(\Bbb{R})$:
\begin{itemize}
\item the boundary of the component of the real escape locus
where $h_\Bbb{R}\equiv\log(2)$ lies on the post-critical curve $f^{\circ 2}(c)=f^{\circ 3}(c)$ ($c$ a critical point of $f$)
and the straight line ${\rm{Per}}_{1}(1)$;
\item the  $h_\Bbb{R}\equiv 0$ component of the intersection of the escape component $E\subset\mathcal{M}_2(\Bbb{C})$ with the component of degree zero maps   
 is one of the  quadrants cut by the straight lines 
$${\rm{Per}}_{1}(1): 2\sigma_1-\sigma_2=3,\quad {\rm{Per}}_{2}(1):2\sigma_1+\sigma_2=1;$$ 
therefore, in the natural $\left(\sigma_1,\sigma_2\right)$-coordinate system of $\mathcal{M}_2(\Bbb{R})$ the corresponding boundary is  piecewise-linear. 
\end{itemize}
\end{proposition}

\begin{proof}
By continuity of the real entropy, $h_\Bbb{R}=\log(2)$ on the boundary of the 
$h_\Bbb{R}\equiv \log(2)$ escape component. There is a classification of real rational maps of degree $d\geq 2$ at which $h_\Bbb{R}$ attains its maximum $\log(d)$ \cite{filom2021real}: In our context,  a real quadratic rational map $f$ with 
$h_\Bbb{R}(f)=\log(2)$ is either in the (parabolic or hyperbolic) shift locus, or its Julia set is the whole $\hat{\Bbb{R}}$, or the Julia set a closed subinterval of it on which $f$ restricts to a boundary-anchored unimodal map with surjective monotonic laps. Among the last two possibilities, in the former $f$ induces an unramified two-sheeted covering of $\hat{\Bbb{R}}$,
the case which is irrelevant here. In the latter case, due to the aforementioned properties of the unimodal interval map 
$f\restriction_{\mathcal{J}(f)}$, the unique critical value should be the prefixed boundary point; hence the relation 
$f^{\circ 2}(c)=f^{\circ 3}(c)$.  The boundary point $\langle f\rangle$ of course cannot be hyperbolic, so the only remaining possibility is being in the  parabolic shift locus where both critical points converge to a fixed point of multiplier $+1$ and multiplicity two; compare with 
\cite[Lemma 8.2]{MR1246482} and \cite[\S4]{MR1806289}.\\
\indent
Switching to the part of the intersection of $E$ with the component of degree zero maps in $\mathcal{M}_2(\Bbb{R})-\mathcal{S}(\Bbb{R})$ where $h_\Bbb{R}\equiv 0$; we have to prove that this is cut off by lines ${\rm{Per}}_1({\pm 1})$. In the top quadrant determined by these lines in Figure \ref{fig:main}, there is precisely one real fixed point since we are above the dotted line ${\rm{Per}}_1(1)$ in the degree zero component; moreover, the period-doubling bifurcation has not occurred as we have not hit the other dotted line ${\rm{Per}}_1(-1)={\rm{Per}}_2(1)$ yet. Hence this open quadrant of the component of degree zero maps is characterized by the existence of a unique real fixed point which is attracting and the absence of real $2$-cycles.\footnote{By Sharkovsky's theorem, this means that there is no periodic point of period larger than one. Also being attracting is somehow automatic here: Let $g:[a,b]\rightarrow [a,b]$ be a $C^1$ interval map with a unique fixed point $p\in[a,b]$ and without any $2$-cycle. Applying the intermediate value theorem to $g^{\circ 2}$ implies that $g^{\circ 2}(x)>x$
over $[a,p)$ while $g^{\circ 2}(x)<x$ over $(p,b]$; hence $\left(g^{\circ 2}\right)'(p)=\left(g'(p)\right)^2\leq 1$ and $p$ is thus non-repelling.}  
Clearly, a map $f\in{\rm{Rat}}_2(\Bbb{R})$ form the $h_\Bbb{R}\equiv 0$ escape component fits in this description: By Proposition \ref{escape locus}, the Julia set $\mathcal{J}(f)$ is away from $\hat{\Bbb{R}}$. The subsystem $f\restriction_{\mathcal{J}(f)}$ is conjugate to the one-sided shift on two symbols; and therefore, for any $n$ there are exactly  $2^n$ distinct fixed points of $f^{\circ n}$ in $\mathcal{J}(f)$ while, on the other hand, the number of fixed points of 
$f^{\circ n}$ on the whole Riemann sphere, counted with multiplicity, is $2^n+1$. We deduce that the only periodic point for the restriction of $f$ to the Fatou set, and thus the only periodic point for the smaller restriction  $f\restriction_{\hat{\Bbb{R}}}$, is the real attracting fixed point $p$ whose basin contains both critical points. 
Conversely,  if $f\restriction_{\hat{\Bbb{R}}}$ has an attracting fixed point $p$ and no other point of period one or two, the real immediate basin of $p$ cannot be a proper subinterval of $\hat{\Bbb{R}}$ since otherwise, at least one of the endpoints of the immediate basin will be of period one or two. Therefore, the orbit of every point of $\hat{\Bbb{R}}$ converges to $p$; in particular, both critical orbits of $f$ tend to $p$. Consequently, $\langle f\rangle$ lies in the desired real escape component. 
\end{proof}
\noindent
The intersection of the escape component with the polynomial line $\left\{(2,4c)=\left\langle z^2+c\right\rangle\mid c\in\Bbb{R}\right\}$ of the moduli space $\mathcal{M}_2(\Bbb{R})$ can be identified with the set 
$\Bbb{R}-\mathbf{M}=(-\infty,-2)\cup\left(\frac{1}{4},+\infty\right)$ of real parameters outside the Mandelbrot set. The real entropy is zero for $c>\frac{1}{4}$ and $\log(2)$ for $c<-2$. Thus, we also refer to the real escape components associated with entropy values $0$ and $\log(2)$ as the ``upper'' or ``lower'' real escape components; cf. 
\cite[Figure 16]{MR1246482}.

\begin{remark}\label{escape boundary'}
The curve $f^{\circ 2}(c)=f^{\circ 3}(c)$ -- which appeared in the preceding proposition -- is described by the condition of a critical value being prefixed but not fixed. Invoking the mixed normal form \eqref{mixed normal form}, this curve can be exhibited as 
$$
\left\{\left\langle\frac{1}{\mu}\left(z+\frac{1}{z}+2\right)\right\rangle\,\Big|\, \mu\in\Bbb{R}-\{0\}\right\}
$$  
where $c=-1$ serves as the desired critical point. Setting $a$ to be $\frac{2}{\mu}$ in \eqref{coordinates} yields its equation in the $\left(\sigma_1,\sigma_2\right)$-plane as  $\left(\sigma_1+2\right)\left(2\sigma_1+\sigma_2+4\right)=\sigma_1-2$.  The lower half of this hyperbola -- which is part of the boundary of the  $h_\Bbb{R}\equiv\log(2)$ escape component   -- intersects ${\rm{Per}}_1(1):\sigma_2=2\sigma_1-3$ at 
$\left\langle z+\frac{1}{z}+2\right\rangle=(-1,-5)$. The Julia set of $z+\frac{1}{z}+2$ is an interval as expected from the proof of Proposition \ref{escape boundary}; compare with \cite[Problem 10-e]{MR2193309}. 
In fact, it is not hard to check that
for $\mu\geq 1$ the Julia set of $\frac{1}{\mu}\left(z+\frac{1}{z}+2\right)$ is $[-\infty,0]$ and the orbits outside it either converge to the parabolic fixed point at infinity (when $\mu=1$) or to the attracting fixed point $\frac{1}{\sqrt{\mu}-1}$
(when $\mu>1$).
 \end{remark}

\section{A new parameter space}\label{S4}
The goal of this section is twofold: We first introduce a new normal form \eqref{new normal form} that on the real axis restricts to an interval map without any vertical asymptotes and is hence convenient for computer implementations. This results in a new parameter space -- illustrated in Figure \ref{fig:parameter space} -- that admits a finite-to-one map onto the component of degree zero maps in $\mathcal{M}_2(\Bbb{R})-\mathcal{S}(\Bbb{R})$. Later in \S\ref{S5}, we will apply certain entropy calculation algorithms to maps of this  form to obtain entropy contour plots in this parameter space. Projecting into $\mathcal{M}_2(\Bbb{R})$ then yields pictures of isentropes in unimodal and bimodal regions of the moduli space.  Secondly, also in the current section, we will make a series of observations based on this more tractable normal form. These observations result in Lemma \ref{attracting fixed point} -- which is vital for the proof of Theorem \ref{temp1} in \S\ref{S6.1} -- and the exclusion of certain parts of the moduli space where the induced dynamics on the real circle are known and the real entropy is identically $0$ or $\log(2)$. This culminates in  Proposition \ref{unimodal-bimodal} that reduces our investigation of the real entropy of quadratic rational maps to the study of certain two-parameter families of unimodal and $(+-+)$-bimodal interval maps. 

We begin by slightly modifying the mixed normal form in \eqref{mixed normal form}. Recalling Propositions \ref{classification} and \ref{Julia circle}, in studying the real entropy of quadratic rational maps one can concentrate only on maps of the form
$$\frac{1}{\mu}\left(z+\frac{1}{z}\right)+a \quad (\mu\in\Bbb{R}-\{0\}, a\geq 0).$$
We can put the critical values at $\pm 1$ and get rid of real vertical asymptotes by a simple conjugation:
\begin{equation}\label{change of variable}
\left(z\mapsto\frac{2}{\mu(z-a)}\right)\circ\left(z\mapsto \frac{1}{\mu}\left(z+\frac{1}{z}\right)+a\right)\circ\left(z\mapsto\frac{2}{\mu(z-a)}\right)^{-1}=\frac{2\mu z(a\mu z+2)}{\mu^2 z^2+(a\mu z+2)^2}.
\end{equation}
It is more convenient to denote $a\mu$ by $t$ and write the previous map as 
$$x\mapsto\frac{2\mu x(tx+2)}{\mu^2x^2+(tx+2)^2}$$ where $\mu t\geq 0$.  
Here, $\mu\neq 0$ is the multiplier of the fixed point $x=0$ and the critical points and values are  
$\frac{2}{\pm\mu-t}\mapsto\pm1$. Consequently, it suffices to deal with the following family of systems defined on a common compact interval:
\begin{equation}\label{new normal form}
\left\{x\mapsto\frac{2\mu x(tx+2)}{\mu^2x^2+(tx+2)^2}:[-1,1]\rightarrow [-1,1]\right\}_{\mu\in\Bbb{R}-\{0\}, \mu t\geq 0}.
\end{equation}
Based on how the points $\frac{2}{\pm\mu-t}$ are located with respect to $[-1,1]$, one gets to the corresponding partition of the $(\mu,t)$-parameter plane in Figure \ref{fig:parameter space}.
Substituting $a$ with $\frac{t}{\mu}$ in \eqref{coordinates}
results in the formula 
\begin{equation}\label{the transformation}
\begin{cases}
F:\left\{(\mu,t)\in\Bbb{R}^2\mid \mu\neq 0, \,\mu t\geq 0\right\}\rightarrow\mathcal{M}_2(\Bbb{R})\cong\Bbb{R}^2\\
F(\mu,t)=(\sigma_1,\sigma_2)\\
\sigma_1=\mu-2+\frac{4}{\mu}-\frac{t^2}{\mu} \quad \sigma_2=\left(\mu+\frac{1}{\mu}\right)\sigma_1-\left(\mu^2+\frac{2}{\mu}\right)
\end{cases}
\end{equation}
for the map that assigns to each member of the family \eqref{new normal form} its conjugacy class.  
According to Proposition \ref{classification}, this map $F$ from the first and the third quadrants  of the $(\mu,t)$-plane (minus the $t$-axis) to the real moduli space  $\mathcal{M}_2(\Bbb{R})\cong\Bbb{R}^2$ is  onto the complement of the open degree $\pm 2$ regions and the open ray $\left\{(2,\sigma_2)\mid \sigma_2>1\right\}$.

\begin{figure}[ht!]
\center
\includegraphics[width=14cm]{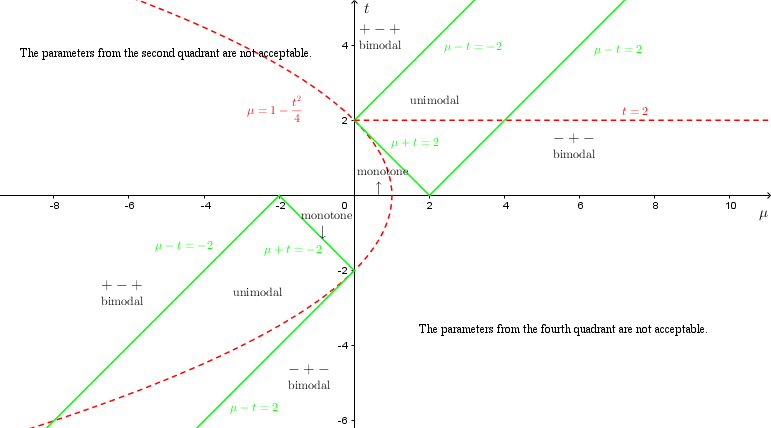}
\caption{The $(\mu,t)$-parameter space on which the map $F$  is defined in \eqref{the transformation}.}
\label{fig:parameter space}
\end{figure}

We proceed with several preliminary observations about the map from \eqref{the transformation}:
\begin{enumerate}[\bfseries 4.a]

\item \label{a}
Given the derivation of the normal form in \eqref{new normal form} from the mixed normal form \eqref{mixed normal form}, the cardinality of the fiber of $F$ above a point of $\mathcal{M}_2(\Bbb{R})-\mathcal{S}(\Bbb{R})$ is the number of different real fixed points of non-zero multiplier in the corresponding class.\footnote{The fixed-point normal form \eqref{fixed-point normal form} clearly indicates that away from the symmetry locus distinct fixed points come with different multipliers.}
In particular, $F$ is at most three-to-one. 
  
\item \label{b} 
The preimage under $F$ of the polynomial line $\sigma_1=2$ from Figure \ref{fig:main} consists of the line segment $\mu+t=2$ and rays $\mu-t=2$ in Figure \ref{fig:parameter space}. In fact, $F$ takes each of 
$\left\{(\mu,\mu-2)\mid \mu<0\right\}$ 
or
$\left\{(\mu,\mu-2)\mid \mu> 2\right\}$
bijectively onto the part of the vertical line $\sigma_1=2$ which is below the point $(2,0)=\langle z^2\rangle$
and takes 
$\left\{(\mu,2-\mu)\mid 0<\mu\leq 2\right\}$
to the closed segment of the polynomial line between 
$(2,0)=\langle z^2\rangle$
and
$(2,1)=\left\langle z^2+\frac{1}{4}\right\rangle$.

\item \label{c}
The preimage under $F$  of the line $\sigma_1=-6$ from Figure \ref{fig:main} -- with the post-critical description $f(c_1)=c_2$ -- consists of the line segment $\mu+t=-2$ and rays $\mu-t=-2$ in Figure \ref{fig:parameter space}. In fact, $F$ takes  
$\left\{(\mu,\mu+2)\mid \mu<-2\right\}$ 
and
$\left\{(\mu,-2-\mu)\mid -2<\mu<0 \right\}$
bijectively onto the parts of the vertical line  $\sigma_1=-6$
which are below and above the point  $(-6,12)=\langle \frac{1}{z^2}\rangle$ respectively.

\item \label{d} 
The map $F$ takes positive and negative rays of the $\mu$-axis to the symmetry locus $\mathcal{S}(\Bbb{R})$ in the 
Figure \ref{fig:main}.  
More precisely, $\left\{(\mu,0)\mid \mu\leq -2\right\}$ and $\left\{(\mu,0)\mid -2\leq\mu<0\right\}$
are  bijected respectively onto left and right branches of the component of 
$\mathcal{S}(\Bbb{R})$ that  passes through  $(-6,12)=\langle \frac{1}{z^2}\rangle$ and is adjacent to the degree $-2$ region while  $\left\{(\mu,0)\mid 0<\mu< 2\right\}$ and $\left\{(\mu,0)\mid \mu> 2\right\}$
are  bijected onto the halves of the other component of $\mathcal{S}(\Bbb{R})$ adjacent to the degree $+2$ region which are above or below the point  $(2,0)=\langle z^2\rangle$ respectively.

\item \label{e} 
The ray $\left\{(\mu,2)\mid \mu>0\right\}$ in Figure \ref{fig:parameter space} goes to the curve in $\mathcal{M}_2(\Bbb{R})$ defined by the critical orbit relation $f^{\circ 3}(c)=f^{\circ 2}(c)$, because when $t=2$:
$$-\frac{2}{\mu+t}\mapsto -1\mapsto \frac{2\mu(t-2)}{\mu^2+(t-2)^2}=0\mapsto 0.$$
The curve contains a part of the boundary of the $h_\Bbb{R}\equiv \log(2)$ escape component in $\mathcal{M}_2(\Bbb{R})$; see Proposition \ref{escape boundary} and Remark \ref{escape boundary'}.\footnote{The post-critical relation $f^{\circ 3}(c)=f^{\circ 2}(c)$ is satisfied along the line $t=-2$   too, this time for the other critical point 
$c=\frac{2}{\mu-t}$. But keep in mind that  in the third quadrant, near  $t=-2$, maps have precisely one real fixed point (observation \hyperref[f]{4.f}) while for the maps in the aforementioned real escape component all periodic points must be  real by Proposition \ref{escape locus}.}

\item \label{f} 
Solving for non-zero roots of 
$$\frac{2\mu x(tx+2)}{\mu^2x^2+(tx+2)^2}=x$$
results in the quadratic equation 
\begin{equation}\label{quadratic equation}
(\mu^2+t^2)x^2+(4t-2\mu t)x+4(1-\mu)=0
\end{equation}
whose discriminant is 
$16\mu^2\left(\frac{t^2}{4}+\mu-1\right)$.
The parabola $\mu=1-\frac{t^2}{4}$ in Figure \ref{fig:parameter space} is bijectively mapped onto the ray 
$\left\{(\mu+2,2\mu+1)\mid \mu\leq 1\right\}$ of the dotted line 
$${\rm{Per}}_1(1):\sigma_2=2\sigma_1-3$$ 
in Figure \ref{fig:main}.  Parameters inside the parabola are precisely those for which the origin is the only real fixed point of
$z\mapsto\frac{2\mu z(tz+2)}{\mu^2z^2+(tz+2)^2}$; in particular, the restriction of $F$  to the interior of the parabola $\mu=1-\frac{t^2}{4}$ is injective.  
\end{enumerate}
Observation \hyperref[a]{4.a} indicates that the cardinality of a fiber of $F$ is given by the number of real fixed points, a 
number which is determined in observation \hyperref[f]{4.f}: 
along the parabola  $\mu=1-\frac{t^2}{4}$ and the ray $\left\{(1,t)\mid t\geq 0\right\}$ there is a multiple fixed point; for parameters inside the parabola there is precisely one real fixed point which is simple and for the rest of $(\mu,t)$ parameters there are three real fixed points. It would be useful to take a closer look at the latter situation because it is exactly the subset of parameters over which the injectivity of $F$ fails.

\begin{lemma}\label{attracting fixed point}
A real quadratic map with real critical points and three distinct fixed points on $\hat{\Bbb{R}}$ -- in particular, any map from the $(-+-)$-bimodal region -- always has an attracting fixed point. Furthermore, the multiplier of this attracting fixed point can be assumed to be non-negative unless the topological type is $(-+-)$-bimodal.\footnote{Compare with \cite[Lemma 10.1]{MR1246482}.}
\end{lemma}
\begin{proof}
This is an immediate consequence of the fixed point formula \eqref{fixed point formula}: the multipliers $\mu_1,\mu_2,\mu_3$ of fixed points should be real numbers different from $1$ satisfying 
$$\frac{1}{1-\mu_1}+\frac{1}{1-\mu_2}+\frac{1}{1-\mu_3}=1.$$
If at least one $\mu_i\in\Bbb{R}-\{1\}$ belongs to $[0,1)$, we are done. Assuming the contrary, suppose for every 
$i\in\{1,2,3\}$ either $\mu_i<0$ or $\mu_i>1$. It is impossible to have the latter for all $\mu_i$'s since in that case the left-hand side of the equality above would be negative, and it is not possible that the former holds for all $\mu_i$'s either since the interval map
$f\restriction_{f(\hat{\Bbb{R}})}:f(\hat{\Bbb{R}})\rightarrow f(\hat{\Bbb{R}})$
can have at most two decreasing laps and thus at most two fixed points of negative multipliers. So without any loss of generality, one can assume that either $\mu_1,\mu_2>1,\mu_3<0$ or $\mu_1,\mu_2<0,\mu_3>1$. Rewriting the equality above as $\mu_3=\frac{2-(\mu_1+\mu_2)}{1-\mu_1\mu_2}$ and comparing signs implies that the first possibility cannot take place. Consequently, we need to have $\mu_1,\mu_2<0,\mu_3>1$ and now the presence of two real fixed points with negative multipliers requires the induced dynamics on $\hat{\Bbb{R}}$ to be $(-+-)$-bimodal. At least one of the fixed points of negative multiplier is attracting as otherwise we have $\mu_1,\mu_2\leq -1,\mu_3>1$ which yields  
$\mu_3>0>\frac{2-(\mu_1+\mu_2)}{1-\mu_1\mu_2}$ contradicting the aforementioned equality.\\
\indent
Finally, notice that a $(-+-)$-bimodal continuous self-map of a compact interval must have a fixed point at any of its laps by a simple application of the intermediate value theorem; cf. Figures \ref{fig:graph-1}, \ref{fig:graph-2}.
\end{proof}

\begin{figure}[ht!]
\centering
\includegraphics[width=9cm]{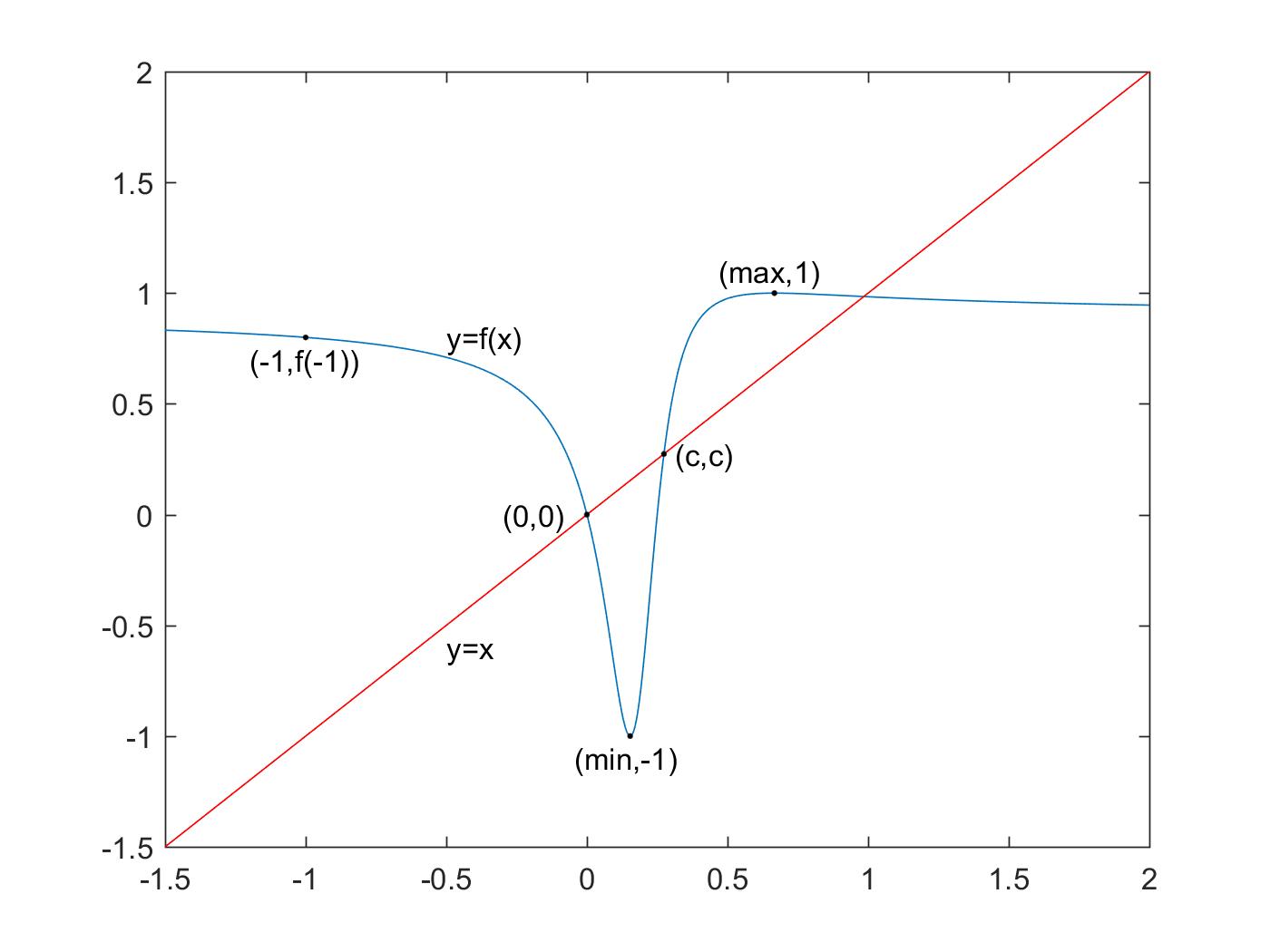}
\caption{A map of the form $f(x)=\frac{2\mu x(tx+2)}{\mu^2x^2+(tx+2)^2}$ with $\mu,t<0$ that admits three real fixed points.}
\label{fig:graph-1}
\end{figure}

\begin{figure}[ht!]
\centering
\includegraphics[width=9cm]{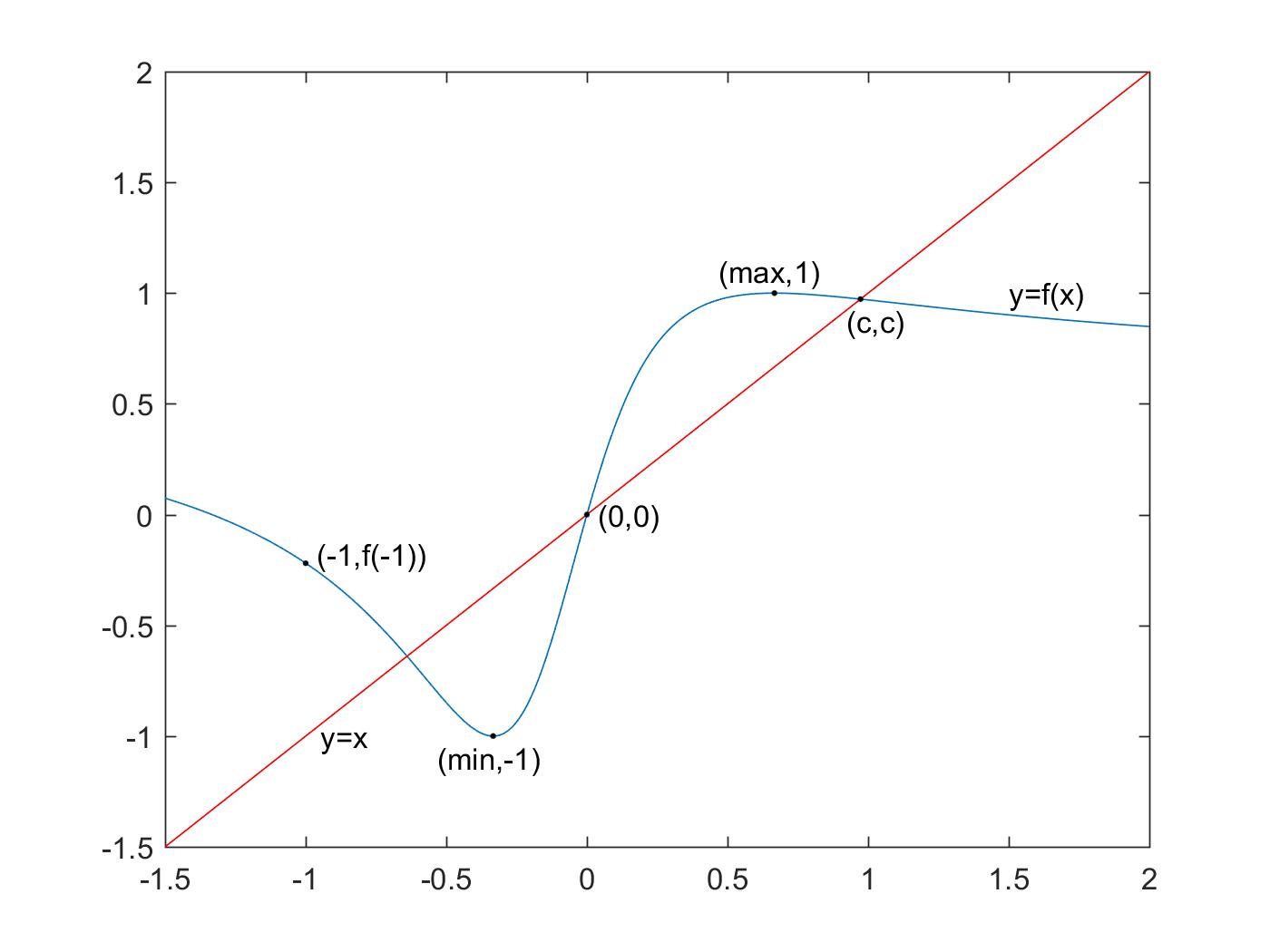}
\caption{A map of the form $f(x)=\frac{2\mu x(tx+2)}{\mu^2x^2+(tx+2)^2}$ with $\mu>0, 0<t<2$.}
\label{fig:graph-2}
\end{figure}

Equipped with this, we continue our observations about the map $F$:
\begin{enumerate}[\bfseries 4.a]
\setcounter{enumi}{6}

\item \label{g} 
In Figure \ref{fig:parameter space}, the image under $F$ of the part of the unimodal region of the third quadrant  which is outside the parabola is covered by the image of the first quadrant: by Lemma 
\ref{attracting fixed point}, a unimodal map $x\mapsto\frac{2\mu x(tx+2)}{\mu^2x^2+(tx+2)^2}$ there 
admits three real fixed points, one of them with non-negative multiplier.  If this multiplier is positive,  a suitable Möbius  change of coordinates of $\hat{\Bbb{R}}$  puts the fixed point at $\infty$ and yields  a mixed normal form 
$\frac{1}{\mu'}(z+\frac{1}{z})+a'$ with $\mu'>0,a'\geq 0$.  After a change of variable similar to \eqref{change of variable}, this then results in a function of the form 
$z\mapsto\frac{2\mu' z(t'z+2)}{\mu'^2z^2+(t'z+2)^2}$. When the multiplier is zero, the map is conjugate to a polynomial and  the sum of the other two multipliers is $2$; so again, there is a real fixed point of positive multiplier and the argument above remains valid. \\
\indent
The same holds for $(-+-)$-bimodal regions: a $(-+-)$-bimodal self-map of an interval possesses a fixed point in each of its laps (Figures \ref{fig:graph-1}, \ref{fig:graph-2}) and hence a fixed point of non-negative multiplier. Thus a point from the $(-+-)$-bimodal region of the moduli space can be written as $F(\mu,t)$ for a suitable point $(\mu,t)$ from the first quadrant. 
\end{enumerate} 

The observations made so far yield Figure \ref{fig:transformation}  illustrating how $F$ takes regions of the $(\mu,t)$-plane 
(Figure \ref{fig:parameter space}) to those of the moduli space $\mathcal{M}_2(\Bbb{R})$ (Figure \ref{fig:main}). 

\begin{figure}[htb!]
\center
\includegraphics[width=13cm, height=7.3cm]{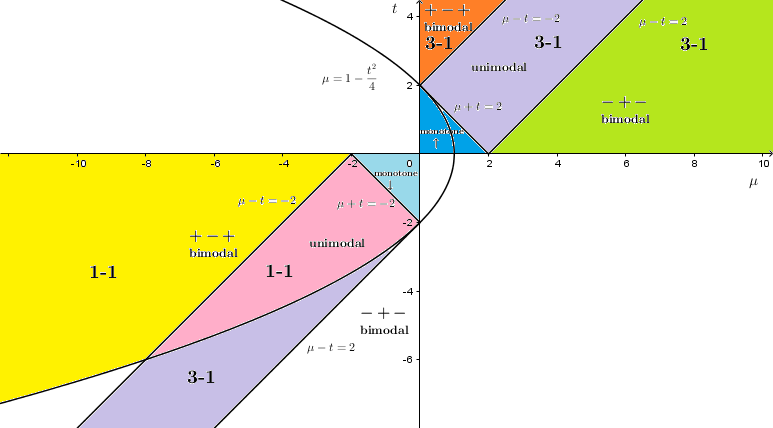}
\includegraphics[width=15cm, height=7.3cm]{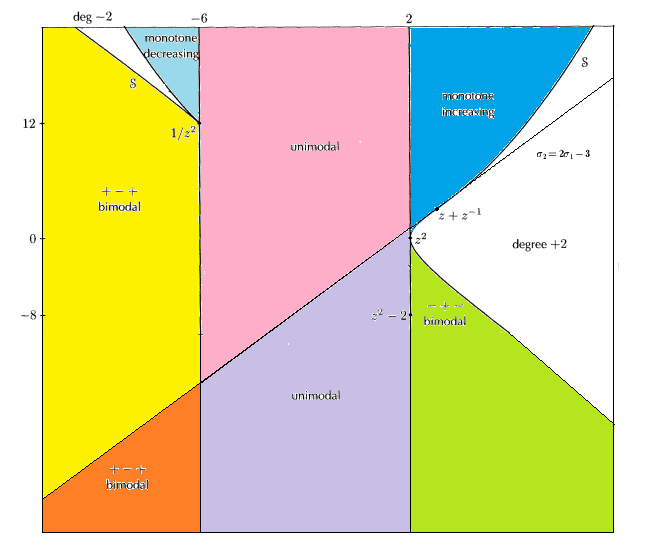}
\caption{The map $F$ from \eqref{the transformation} takes each colored  region in the first or third quadrant of the $(\mu,t)$-plane onto the region of the same color in the moduli space. Form the entropy point of view, one can safely focus only on unimodal and $(-+-)$-bimodal regions of the first quadrant and the non-monotone colored regions of the third quadrant inside the parabola: $h_\Bbb{R}\equiv\log(2)$ on $(+-+)$-bimodal parameters outside the parabola (observations \hyperref[h]{4.h}, \hyperref[i]{4.i}), and the $(-+-)$-bimodal region of the first quadrant is mapped onto the corresponding region of the moduli space (observation \hyperref[g]{4.g}).}
\label{fig:transformation}
\end{figure}

It is of course not necessary to consider the whole first and third quadrants of the $(\mu,t)$-plane in Figure \ref{fig:parameter space}: there are parts  of the $(\mu,t)$-parameter plane where the induced dynamics on the real circle is fairly easy to describe, i.e. the monotone regions or the components of the real escape locus;  the real entropy over these parts either vanishes or is identically $\log(2)$.  
We continue our observations to find more of such uninteresting regions.
\begin{enumerate}[\bfseries 4.a]
\setcounter{enumi}{7}
\item  \label{h}  Consider a point $(\mu,t)$ of the first quadrant away from the monotone increasing region; $\mu,t>0$ and $\mu+t>2$. 
The map $f(x)=\frac{2\mu x(tx+2)}{\mu^2x^2+(tx+2)^2}$ attains it absolute minimum $-1$ at 
${\rm{min}}:=-\frac{2}{\mu+t}\in [-1,0]$ while $f(0)=0$; cf. Figure \ref{fig:graph-2}. Hence if $f(-1)=\frac{2\mu (t-2)}{\mu^2+(t-2)^2}$ is positive, by the intermediate value theorem, $[-1,0]$ would be doubly covered by the subset 
$\left[-\frac{2}{t},0\right]=\left[-\frac{2}{t},\min\right]\cup\left[\min,0\right]$ of itself; hence must contain the whole Julia set; cf. Lemma \ref{criterion}. We conclude that  in Figure \ref{fig:parameter space} the real entropy is identically $\log(2)$ above the line $t=2$ ; in particular, for the $(+-+)$-bimodal maps one can concentrate only  on the  corresponding region of the third quadrant. Compare with observation \hyperref[e]{4.e}: the line $t=2$ goes to the boundary of the $h_\Bbb{R}\equiv\log(2)$ component  of the escape locus.

\item \label{i} 
Switching to the third quadrant, we claim that outside the parabola and away from the monotone region the entropy is constant outside a bounded region. For a point $(\mu,t)$ outside the parabola with $\mu,t<0$ and $\mu+t<-2$, the corresponding interval map 
$$
\begin{cases}
f:[-1,1]\rightarrow[-1,1]\\
f(x)=\frac{2\mu x(tx+2)}{\mu^2x^2+(tx+2)^2} 
\end{cases}$$
has three fixed points and takes its absolute minimum $-1$ at 
$\min:=-\frac{2}{\mu+t}\in (0,1)$. Hence it must have a positive real fixed point $c\in[\min,1]$; compare with Figure \ref{fig:graph-1}. By the intermediate value theorem, $f$ attains every value from $[-1,c]$ in $[\min,c]$. The same interval of values is realized over $[-1,\min]$ as well provided that $f(-1)=\frac{2\mu (t-2)}{\mu^2+(t-2)^2}>c$.
 In that case, the interval $[-1,c]$ would be doubly covered by a subset of itself and thus contains the Julia set, and so the real entropy would be $\log(2)$.  We next derive inequalities for $\mu,t$ that guarantee this. The non-zero fixed points are roots of the quadratic equation \eqref{quadratic equation} and thus, by Vieta's formulas, they are positive and the smaller one  is at most $\frac{\mu t-2t}{\mu^2+t^2}$. It suffices to take $\mu<-4$ due to the following inequalities:
\begin{equation*}
\begin{split}
&c\leq \frac{\mu t-2t}{\mu^2+t^2}<\frac{\frac{3}{2}\mu t}{\mu^2+t^2}
=\frac{1}{\frac{2}{3}\left(\frac{\mu}{t}+\frac{t}{\mu}\right)}\leq
\frac{1}{\frac{1}{2}\left(\frac{\mu}{t}+\frac{t}{\mu}\right)+\frac{1}{3}}=
\frac{2}{\frac{\mu}{t}+\left(\frac{t}{\mu}+\frac{2}{3}\right)}\\
&<\frac{2}{\frac{\mu}{t-2}+\frac{t-2}{\mu}}=\frac{2\mu (t-2)}{\mu^2+(t-2)^2}=f(-1).
\end{split}
\end{equation*}
In particular, $h_\Bbb{R}\equiv\log(2)$ on the part of $(+-+)$-bimodal region of the third quadrant that lies outside the parabola $\mu=1-\frac{t^2}{4}$ as $\mu\leq -8$ there; cf. Figure \ref{fig:parameter space}. As a matter of fact, $F$ takes this portion of the 
$(+-+)$-bimodal region of the third quadrant onto the part of the $(+-+)$-bimodal region of the moduli space which is below the line ${\rm{Per}}_1(1)$. This is contained in the $h_\Bbb{R}\equiv\log(2)$ component  of the escape locus; compare with
Proposition \ref{escape boundary}. Consequently, throughout the aforementioned part of the third quadrant of the $(\mu,t)$-parameter space, the dynamics on the non-wandering set of the real subsystem (aside from the fixed point that attracts the critical points) is that of the full $2$-shift.
\end{enumerate}

Finally, we exploit the dynamical constraint that Lemma \ref{attracting fixed point} puts on $(-+-)$-bimodal maps to argue that in the first quadrant of the $(\mu,t)$-plane (Figure \ref{fig:parameter space}) we are essentially dealing with unimodal maps.

\begin{enumerate}[\bfseries 4.a]
\setcounter{enumi}{9}
\item  \label{j} 
Pick two parameters $\mu,t>0$ with $\mu+t>2$ so that the self-map 
$f(x)=\frac{2\mu x(tx+2)}{\mu^2x^2+(tx+2)^2}$ of $[-1,1]$ is either unimodal or bimodal; the absolute minimum $-1$
is attained at $\min:=-\frac{2}{\mu+t}\in[-1,0]$; compare with Figure \ref{fig:graph-2}. Such a point $(\mu,t)$ lies outside the parabola $\mu=1-\frac{t^2}{4}$; therefore, by \hyperref[f]{4.f}, there are three real fixed points. According to observation \hyperref[h]{4.h}, there is no harm in restricting ourselves to the region below $t=2$; this rules out the $(+-+)$-bimodal case: 
the other critical point $\max:=\frac{2}{\mu-t}$ does not belong to $[-1,0]$. But if $t<2$, then $f(-1)=\frac{2\mu (t-2)}{\mu^2+(t-2)^2}<0$ and $f$ restricts to a unimodal self-map of $[-1,0]$ because the only critical point there, 
$-\frac{2}{\mu+t}$,  goes to $-1$. This subsystem  of 
$f\restriction_{[-1,1]}$
carries all the entropy: Notice that the other half $[0,1]$ of the domain is invariant as well due to the fact that $\mu,t>0$ and we just need to argue that the entropy of the complementary subsystem $f\restriction_{[0,1]}:[0,1]\rightarrow [0,1]$ vanishes. This is of course true when the aforementioned map is monotone (i.e. $(\mu,t)$ belongs to the unimodal region) and in the case that the other critical point $\max=\frac{2}{\mu-t}$ lies in $[0,1]$ (i.e. $(\mu,t)$ belongs to the $(-+-)$-bimodal region); one has $\mu>2$ and there exists  a unique positive fixed point $c>0$ (see Figure \ref{fig:graph-2}) which is not hard to verify that it is attracting: the multiplier can be written as 
$$
f'(c)=\left(\frac{2}{\frac{\mu x}{tx+2}+\frac{tx+2}{\mu x}}\right)'\Big|_{x=c}=
\frac{4\mu\left(\frac{1}{(\mu c)^2}-\frac{1}{(tc+2)^2}\right)}{\left(\frac{\mu c}{tc+2}+\frac{tc+2}{\mu c}\right)^2}.
$$
Notice that as $f(c)=c$, the term $\frac{\mu c}{tc+2}+\frac{tc+2}{\mu c}$ in the denominator is equal to $\frac{2}{c}$, and thus 
$f'(c)=\frac{1}{\mu}-\mu\left(\frac{c}{tc+2}\right)^2$. But 
$\left(\frac{c}{tc+2}\right)^2=\frac{2}{\mu(tc+2)}-\frac{1}{\mu^2}$
which has been obtained from simplifying the identity
$$\frac{f(c)}{c}=\frac{2\mu(tc+2)}{\mu^2c^2+(tc+2)^2}=1.$$
 Plugging in the previous equality yields
$f'(c)=\frac{2}{\mu}-\frac{2}{tc+2}$. We now have:
$$-1<\frac{-2}{tc+2}<\frac{2}{\mu}-\frac{2}{tc+2}<\frac{2}{\mu}<1;$$
so $c$ is attracting and the orbit under $f$ of every point in $(0,\infty)$ tends to $c$. 

\end{enumerate}

Now we can narrow down the $(\mu,t)$-parameter space in Figure \ref{fig:parameter space} to smaller domains. 
In the first quadrant of the $(\mu,t)$-plane, it suffices to deal with a family of unimodal self-maps of $[-1,0]$ determined by the inequalities  $\mu+t>2, t<2$ (observation \hyperref[j]{4.j}) which  amounts to maps outside the closure of the escape locus in the $(-+-)$-bimodal region of the moduli space and in the half of the unimodal region which lies below the line ${\rm{Per}}_1(1)$. The other half of the unimodal region of $\mathcal{M}_2(\Bbb{R})$ is the bijective image of the part of the unimodal region of the third quadrant which is inside the parabola; cf. Figure \ref{fig:transformation}. Finally, as for the $(+-+)$-bimodal parameters,
observations \hyperref[h]{4.h} and \hyperref[i]{4.i} indicate that all interesting entropy  behavior occurs
in the portion of the  $(+-+)$-bimodal region of the third quadrant which is inside the parabola. 
We have summarized all of these in the following:

\begin{proposition}\label{unimodal-bimodal}
The study of the  real dynamics and the entropy behavior of real quadratic rational maps reduces to studying 
the following  two-parameter families of interval maps in the sense that a real quadratic rational map conjugate to no member of these families is of real entropy $0$ or $\log(2)$. 
Each of the families below is parametrized over a domain from the $(\mu,t)$-parameter space in  Figure \ref{fig:parameter space}:
\begin{itemize}
\item the family 
\begin{equation}\label{family-1}
\mathbf{\mathcal{F}_1}:\left\{x\mapsto\frac{2\mu x(tx+2)}{\mu^2x^2+(tx+2)^2}:[-1,0]\rightarrow [-1,0]\right\}_{(\mu,t)\in\mathbf{U_1}}, 
\end{equation}
of unimodal interval maps parametrized 
over the domain
\begin{equation}\label{domain-1}
\mathbf{U_1}:=\left\{(\mu,t)\mid\mu,t>0,\, 2-\mu< t< 2\right\};
\end{equation}
\item the family 
\begin{equation}\label{family-2}
\mathbf{\mathcal{F}_2}:\left\{x\mapsto\frac{2\mu x(tx+2)}{\mu^2x^2+(tx+2)^2}:[-1,1]\rightarrow [-1,1]\right\}_{(\mu,t)\in\mathbf{U_2}}, 
\end{equation} of unimodal interval maps parametrized 
over the domain
\begin{equation}\label{domain-2}
\mathbf{U_2}:=\left\{(\mu,t)\mid \mu,t<0,\, t-2< \mu< |t+2|\right\};
\end{equation}
\item the family 
\begin{equation}\label{family-3}
\mathbf{\mathcal{F}_3}:\left\{x\mapsto\frac{2\mu x(tx+2)}{\mu^2x^2+(tx+2)^2}:[-1,1]\rightarrow [-1,1]\right\}_{(\mu,t)\in\mathbf{U_3}}, 
\end{equation}
of $(+-+)$-bimodal interval maps parametrized 
over the domain
\begin{equation}\label{domain-3}
\mathbf{U_3}:=\left\{(\mu,t)\,\Big|\,\mu,t<0,\, \mu<\min\left\{t-2,1-\frac{t^2}{4}\right\}\right\}.
\end{equation}
\end{itemize}
\end{proposition}
\noindent
Notice that, once projected to the real moduli space $\mathcal{M}_2(\Bbb{R})$, families $\mathbf{\mathcal{F}_1}$ and $\mathbf{\mathcal{F}_2}$ overlap since they share the half of the unimodal region which is below the line ${\rm{Per}}_1(1)$.

\section{Entropy plots}\label{S5}

We have run a couple of algorithms to compute entropy values for interval maps in the normal form \eqref{new normal form}. The implemented algorithms have been adapted from papers \cite{MR1151977}, \cite{MR1002478} and are based on the  comparison of the kneading data with that of a piecewise-linear map with the same number of laps whose entropy is known.  

In Figures \ref{fig:plot1} and \ref{fig:plot2}, we have applied the algorithm from \cite{MR1002478} to two-parameter families of unimodal interval maps appeared in \eqref{family-1} and \eqref{family-2}. The contour plots obtained in the $(\mu,t)$-plane are then projected into the unimodal and $(-+-)$-bimodal regions of the moduli space; see Figure \ref{fig:plot3}. These plots suggest the following that turns out to be a stronger version of Conjecture \ref{monotonicity 2} (see Proposition \ref{conjectures}):

\begin{conjecture}\label{monotonicity 3}
The level sets of the restriction of  $h_\Bbb{R}$ to the adjacent unimodal and $(-+-)$-bimodal regions of the moduli space
$\mathcal{M}_2(\Bbb{R})$ are connected.
\end{conjecture}

\begin{figure}[ht!]
\center
\includegraphics[width=10cm]{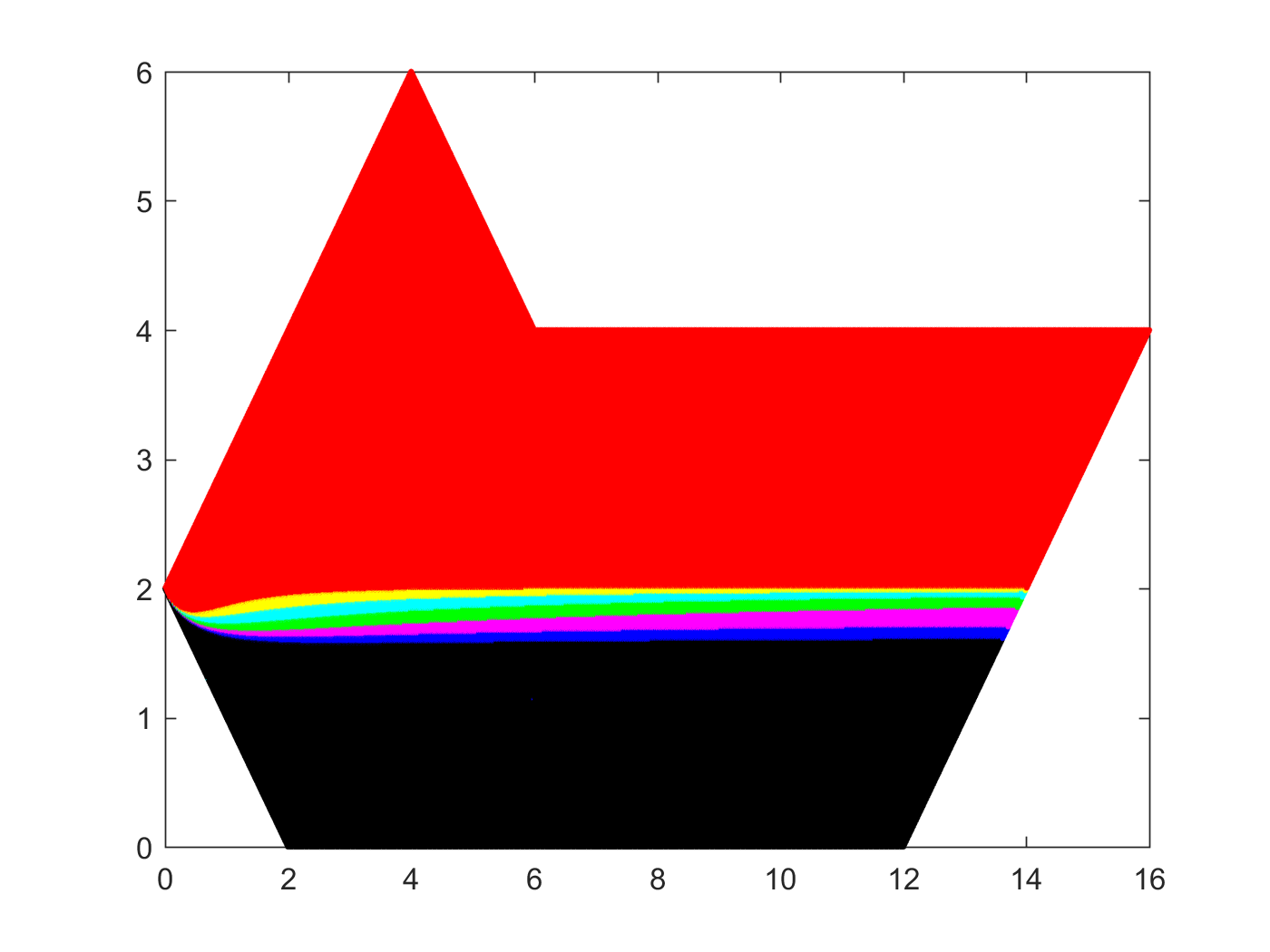}
\caption{A contour plot of the entropy function in the unimodal and $(-+-)$-bimodal regions of the first quadrant of the $(\mu,t)$-plane (Figure \ref{fig:parameter space}). Here the ordering of colors is black$<$blue$<$magenta$<$green$<$cyan$<$yellow$<$red and they  correspond to the partition $[0, 0.1)$, $[0.1, 0.25)$, $[0.25, 0.4)$, $[0.4, 0.48)$, $[0.48, 0.55)$, $[0.55, 0.65)$ and $[0.65, \log(2)]$ of  $[0,\log(2)\approx 0.7]$. The non-trivial entropy behavior occurs for the family $\mathcal{F}_1$ \eqref{family-1}
lying below the line $t=2$  since, according to observation \hyperref[h]{4.h}, the entropy is identically $\log(2)$ once $t\geq 2$.}
\label{fig:plot1}
\end{figure}

\begin{figure}[ht!]
\center
\includegraphics[width=9.5cm]{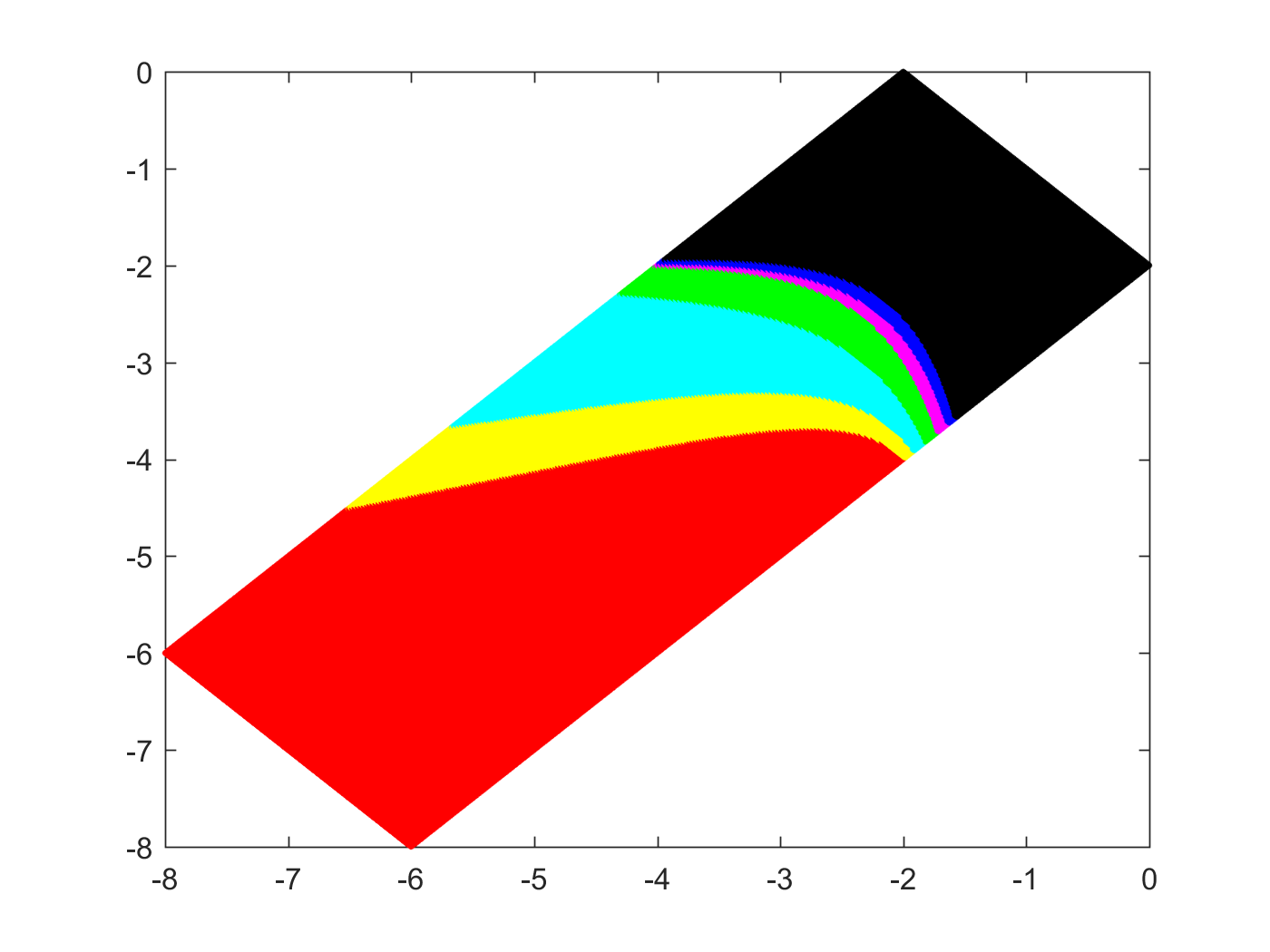}
\caption{An entropy contour plot in the unimodal region of the third quadrant of the $(\mu,t)$-plane (Figure \ref{fig:parameter space}); a region which is parametrized by the family $\mathcal{F}_2$ \eqref{family-2}. The coloring scheme is the same as that of Figure \ref{fig:plot1}.} 
\label{fig:plot2}
\end{figure}

\begin{figure}[ht!]
\center
\includegraphics[width=9.5cm]{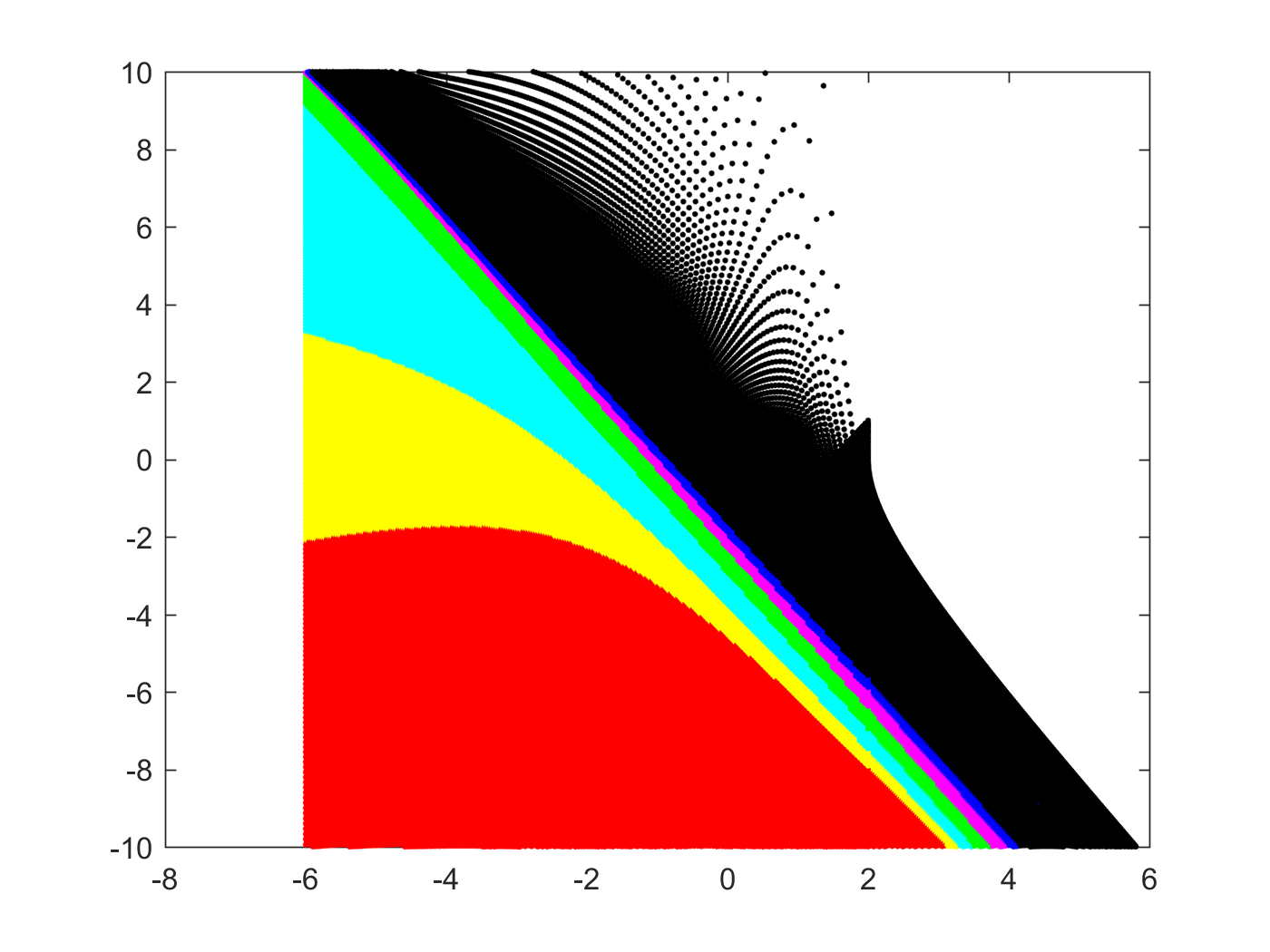}
\caption{An entropy contour plot in unimodal and $(-+-)$-bimodal regions of the moduli space (Figure \ref{fig:main}) obtained from projecting the contour plots illustrated in Figures \ref{fig:plot1}, \ref{fig:plot2} from the $(\mu,t)$-parameter space (Figure \ref{fig:parameter space}) to the moduli space via the map $F$ \eqref{the transformation}.}
\label{fig:plot3}
\end{figure}

\begin{remark}
One has $-6<\sigma_1<2$ in the open unimodal region of $\mathcal{M}_2(\Bbb{R})$ as it is apparent in Figure
\ref{fig:main}. Once $\sigma_1$ is within these bounds, $\sigma_2(\mu,t)$ from 
\eqref{the transformation} blows up as $\mu\to 0$. For this reason, in Figure \ref{fig:plot3} we have confined the projections of Figures \ref{fig:plot1}, \ref{fig:plot2} from the $(\mu,t)$-plane to the moduli space $\mathcal{M}_2(\Bbb{R})$ within the bounds $|\sigma_2|<10$. The discrepancy of black points in Figure \ref{fig:plot3} is also due to $\mu$ appearing in the denominator of $\sigma_2(\mu,t)$: near the $t$-axis the map $(\mu,t)\mapsto\left(\sigma_1(\mu,t),\sigma_2(\mu,t)\right)$ pulls black point of  Figure \ref{fig:plot2} apart. Nevertheless, Figure \ref{fig:plot3} 
(along with Figures \ref{fig:plot1} and \ref{fig:plot2}) can still be viewed as an evidence of connectedness of isentropes in unimodal and $(-+-)$-bimodal regions because the part of the contour plot that is not fully filled with black  lies almost entirely above lines ${\rm{Per}}_1(\pm 1)$ (see Figure \ref{fig:main}) and is thus in the $h_\Bbb{R}\equiv 0$ escape component (Proposition \ref{escape boundary}); so we have not missed any ``dynamically interesting'' part of the contour plot in Figure \ref{fig:plot3}.
\end{remark}

Finally, applying the algorithm from \cite{MR1151977} to the family of $(+-+)$-bimodal maps appeared in \eqref{family-3}
yields Figure \ref{fig:plot4} in the $(\mu,t)$-plane and the corresponding moduli space contour plot in Figure \ref{fig:plot5}.
They serve as the evidence for Conjecture \ref{temp2} on the failure of monotonicity.

\begin{remark}
The study of monotonicity in the $(\mu,t)$-plane rather than in the actual moduli space $\mathcal{M}_2(\Bbb{R})$ does not cause any problem with these conjectures: There is a  continuous map $F$ from the $(\mu,t)$-plane to the 
$\left(\sigma_1,\sigma_2\right)$-plane $\eqref{the transformation}$, and thus monotonicity in unimodal or $(-+-)$-bimodal regions of the former imply the same for the latter. Similarly, the failure of monotonicity in the $(+-+)$-bimodal region of the $(\mu,t)$-plane is equivalent to the same assertion for the moduli space as $F$ is injective over the ``interesting'' part; which is the part of the third quadrant that lies inside the parabola $\mu=1-\frac{t^2}{4}$; cf. Figure \ref{fig:transformation}. 
\end{remark}

\begin{figure}[ht!]
\center
\includegraphics[width=9.5cm]{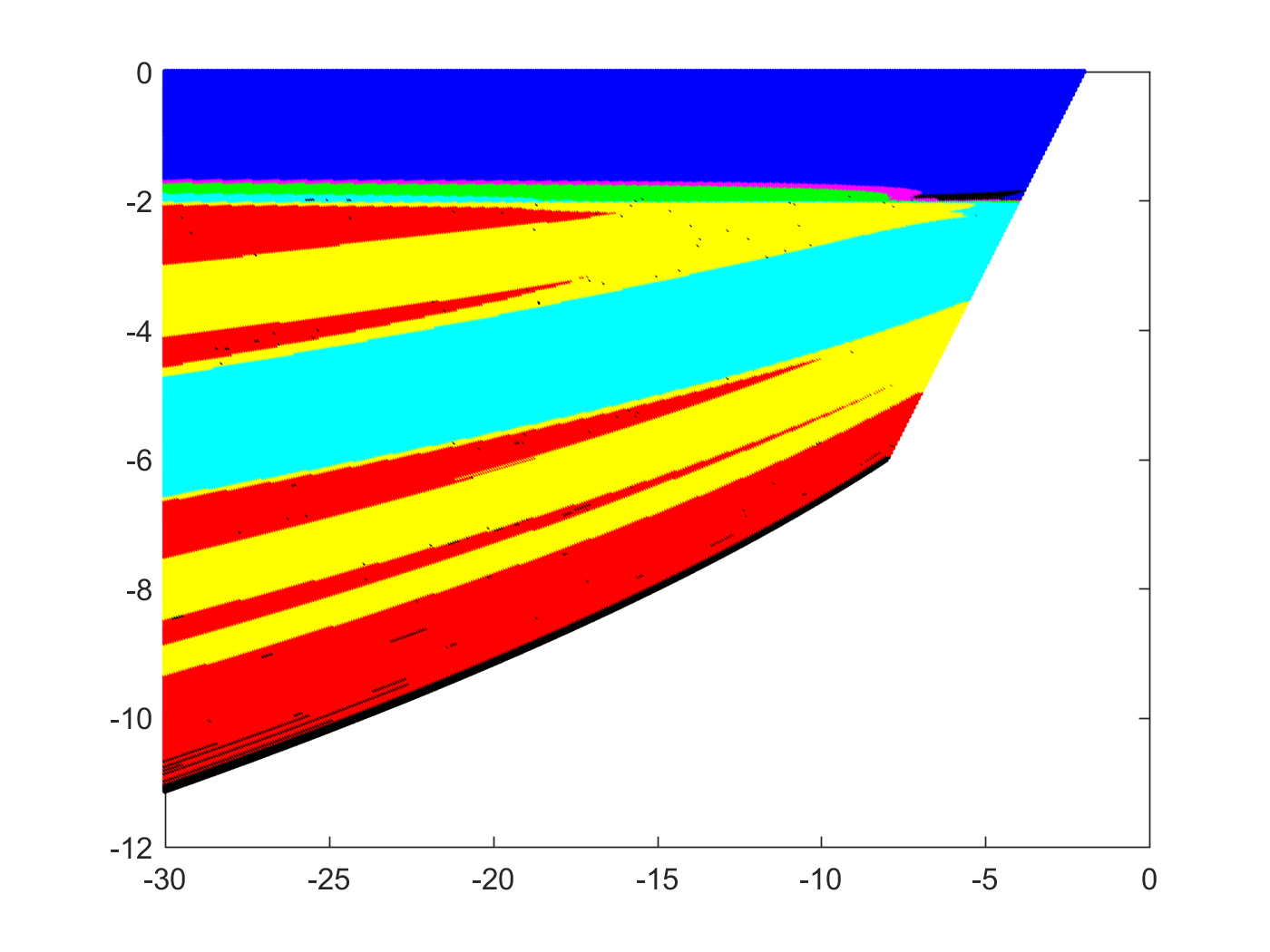}
\caption{An entropy contour plot for the part of the $(+-+)$-bimodal region of the third quadrant of the $(\mu,t)$-plane (Figure \ref{fig:parameter space}) that lies inside the parabola $\mu=1-\frac{t^2}{4}$ and is parametrized by the family $\mathcal{F}_3$ \eqref{family-3}.
Here the  colors blue, magenta, green, cyan, yellow and red correspond to the entropy being in intervals 
$[0,0.05)$, $[0.05,0.2)$, $[0.2,0.3)$, $[0.3,0.5)$, $[0.5,0.66)$ and $[0.66,\log(2)\approx 0.7]$ respectively. The black indicates 
the failure of the algorithm introduced in \cite{MR1151977}.}
\label{fig:plot4}
\end{figure}

\begin{figure}[ht!]
\center
\includegraphics[width=9.5cm]{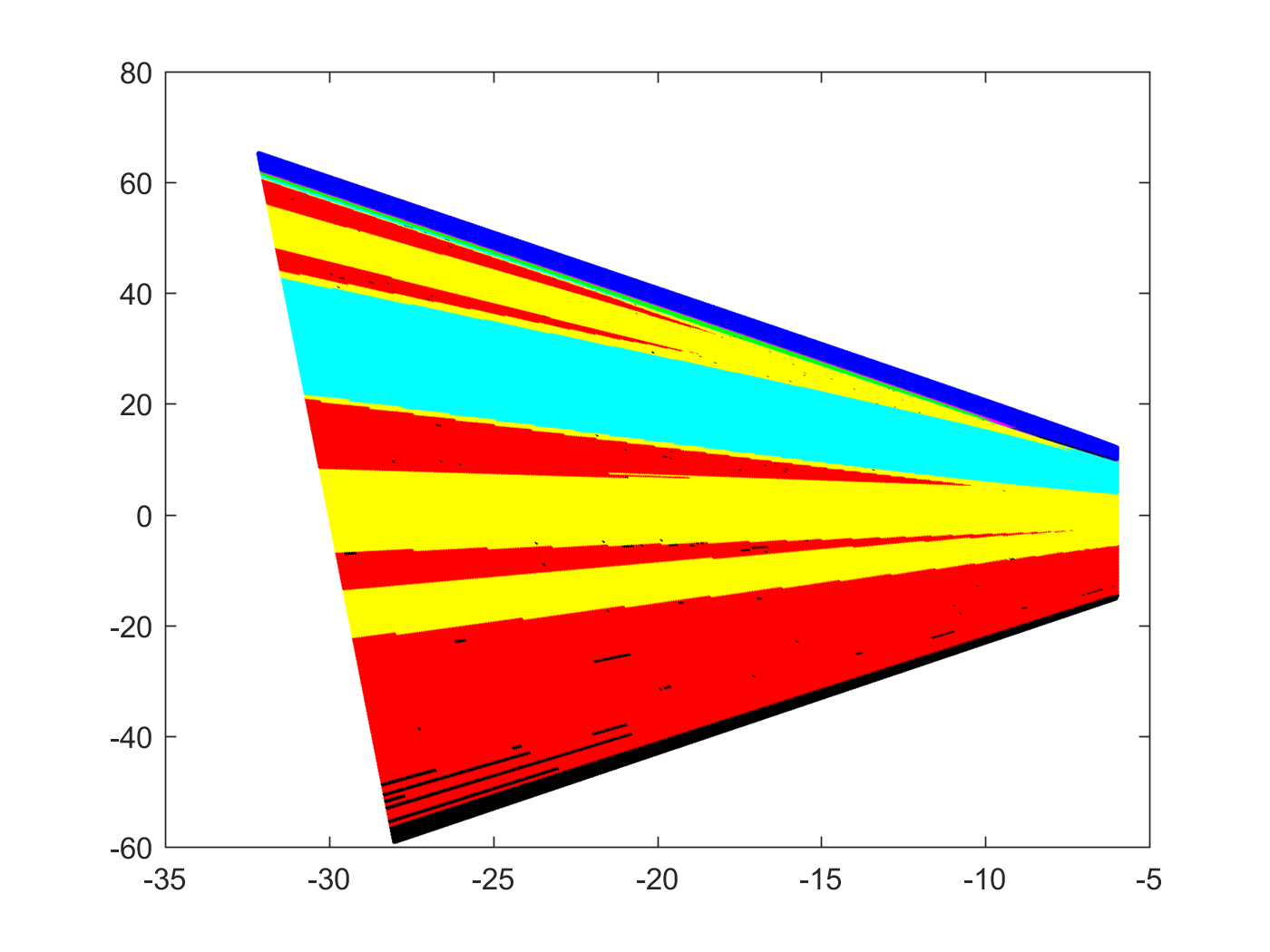}
\caption{Projecting the contour plot illustrated in Figure \ref{fig:plot4} from the $(\mu,t)$-parameter space to the moduli space via the map $F$ \eqref{the transformation} yields a contour plot in the $(+-+)$-bimodal region of the moduli space (Figure \ref{fig:main}). The lower skew boundary line is ${\rm{Per}}_{1}(1): 2\sigma_1-\sigma_2=3$ below which the real entropy is $\log(2)$ (observation \hyperref[i]{4.i}).}
\label{fig:plot5}
\end{figure}

\section{A monotonicity result}\label{S6}

\subsection{The straightening theorem}\label{S6.1}
According to observation \hyperref[f]{4.f} and Lemma \ref{attracting fixed point}, the hypotheses of Theorem \ref{temp1} hold precisely below the line ${\rm{Per}}_1(1)$ and guarantee the existence of a real attracting fixed point. 
Therefore, one can deduce Theorem \ref{temp1}  from Theorem \ref{monotonicity 1} below whose proof is the main goal of this subsection and will establish the connectedness of isentropes for the projection of Figure \ref{fig:plot1} to the moduli space.  

\begin{theorem}\label{monotonicity 1}
Restricted to  the part of the component of degree zero maps in $\mathcal{M}_2(\Bbb{R})-\mathcal{S}(\Bbb{R})$
that lies strictly below the line ${\rm{Per}}_1(1)$  the level sets of the function  $h_\Bbb{R}$ are connected. 
\end{theorem}
The  proof  uses the Douady and Hubbard theory of polynomial-like mappings \cite{MR816367}. We first show that outside the escape locus\footnote{Recall that by observation \hyperref[h]{4.h} the  family $\mathcal{F}_1$  \eqref{family-1} -- on which  Figure \ref{fig:plot1} is based --  is away from the escape locus.} one can quasi-conformally perturb a real attracting fixed point to make it  super-attracting without changing the real entropy.

\begin{theorem}\label{straightening}
Let $f(z)\in\Bbb{R}(z)$ be a real quadratic rational map that is not in the escape locus and admits a real attracting fixed point. 
Then, there exists a  real quadratic polynomial $p$ which is quasi-conformally conjugate to $f$ outside a neighborhood of the attracting fixed point via a homeomorphism $h$ satisfying $\bar\partial h=0$ on the filled Julia set of $p$ (i.e. a hybrid equivalence). Such a polynomial has the same real entropy as $f$ and is unique up to an affine conjugacy.  
\end{theorem}

\begin{proof}
We construct a polynomial-like map out of $f$ following an idea that has been alluded to on
\cite[p. 482]{MR1806289}. Pick a compact topological disk $N$ in the basin of attraction of the real fixed point such that $f(N)\subset{\rm{int}}(N)$, $N$ is invariant under complex conjugation, and $\partial N$ passes through a critical point of $f$.
Then 
\begin{equation}\label{poly-like}
\hat{\Bbb{C}}-f^{-1}(N)\stackrel{f}{\rightarrow}\hat{\Bbb{C}}-N.
\end{equation}
is a degree two polynomial-like map whose filled Julia set is connected and commutes with complex conjugation. Invoking the \textit{Douady-Hubbard straightening theorem} \cite[Theorem 1]{MR816367}, there exists a quasi-conformal homeomorphism  $h:\hat{\Bbb{C}}\rightarrow\hat{\Bbb{C}}$ -- also commuting with the complex conjugation -- that  induces a conjugacy from
the dynamics of a real quadratic polynomial $p$ on its filled Julia set onto  the dynamics  of \eqref{poly-like} on its filled Julia set, and satisfies $\bar\partial h=0$ on the filled Julia set $\mathcal{K}_p$. Forming intersections of the aforementioned filled Julia sets with the real axis, we conclude that $h$ restricts to a conjugacy between two dynamical systems of topological entropies  $h_{\Bbb{R}}(f)$ and $h_{\Bbb{R}}(p)$ respectively; so  $h_{\Bbb{R}}(f)=h_{\Bbb{R}}(p)$. 
The uniqueness part follows from  and the fact that filled Julia set of \eqref{poly-like} is connected due to  \cite[Proposition 2]{MR816367}.
\end{proof}

 Given $\lambda\in(-1,1)$, Theorem \ref{straightening} above assigns to each point of ${\rm{Per}}_1(\lambda)(\Bbb{R})-E$ ($E$ the escape component, see \S\ref{S3.1})
a real quadratic polynomial of the same real entropy that corresponds to a real point of the Mandelbrot set (keep in mind that the filled Julia sets of the maps appearing in the proof of Theorem \ref{straightening} are connected). One needs to keep track of how the projections ${\rm{Per}}_1(\lambda)(\Bbb{R})-E\rightarrow\mathbf{M}\cap\Bbb{R}$ vary as $\lambda$ changes in $(-1,1)$. 
To do so, one can consider the more general setting of the complex quadratic rational maps possessing an attracting fixed point of multiplier $\lambda\in\Bbb{D}$. 
It is more convenient to use the fixed-point normal form \eqref{fixed-point normal form}
\begin{equation}\label{fixed-point normal form 1}
f_{\lambda,\mu}(z):=\frac{z^2+\mu z}{\lambda z+1}
\end{equation}
where two fixed points  $\infty,0$ of multipliers respectively $\lambda,\mu\neq\frac{1}{\lambda}$ are prescribed.
We assume that the former is attracting, i.e. $\lambda\in\Bbb{D}$. Then, just like the case of quadratic polynomials, there is a dichotomy regarding the Julia set: It is connected if and only if the attracting basin $\mathcal{A}_{\lambda,\mu}(\infty)$ contains precisely one critical point; otherwise, both critical orbits tend to $\infty$ and $f_{\lambda,\mu}$ thus lies in the escape locus $E$ with a Cantor Julia set \cite[Lemma 8.2]{MR1246482}. Hence it makes sense to define ``the filled Julia set'' $\mathcal{K}_{\lambda,\mu}$ of 
$f_{\lambda,\mu}$
as
\begin{equation}\label{filled Julia set 1}
\mathcal{K}_{\lambda,\mu}:=\hat{\Bbb{C}}-\mathcal{A}_{\lambda,\mu}(\infty);
\end{equation}
and the ``connectedness locus'' in the $\mu$-parameter plane $\Bbb{C}_{\lambda}:=\Bbb{C}-\{\frac{1}{\lambda}\}$ 
as 
\begin{equation}\label{Mandelbrot variant}
\mathbf{M}_{\lambda}:=\left\{\mu\in\Bbb{C}_{\lambda}\mid \left\langle f_{\lambda,\mu}\right\rangle\notin E\right\}=
\left\{\mu\in\Bbb{C}_{\lambda}\mid \mathcal{K}_{\lambda,\mu} \text{ is connected}\right\}
\end{equation}  
for any $\lambda$ in the open unit disk. Obviously, for $\lambda=0$ the map in 
\eqref{fixed-point normal form 1} is a polynomial with $\mathcal{K}_{0,\mu}$ its filled Julia set in the usual sense, and 
$\mathbf{M}_0$ 
would be the connectedness locus in the parameter space of quadratic polynomials $z^2+\mu z$ 
(\cite[Figure 29]{MR2193309})
which is  a branched double cover of the  Mandelbrot  set via the map
\begin{equation}\label{double cover}
\mathbf{M}_0\rightarrow\mathbf{M}:\mu\mapsto c:=\frac{\mu}{2}-\frac{\mu^2}{4};
\end{equation}
under which both subintervals $[1,4]$ and $[-2,1]$
of $\mathbf{M}_0\cap\Bbb{R}=[-2,4]$
biject onto the real slice $\mathbf{M}\cap\hat{\Bbb{R}}=\left[-2,\frac{1}{4}\right]$ of the Mandelbrot set. In particular, invoking the monotonicity of entropy for quadratic polynomials \cite[Corollary 13.2]{MR970571}, the function 
$$\mu\mapsto h_\Bbb{R}\left(z\mapsto z^2+\mu z\right)$$
is monotonic on each of subintervals $[-2,1]$ and $[1,4]$, but not on their union.

The following theorem of Uhre is all we need to control the straightening in a family: 
\begin{theorem}\label{Uhre}
\cite[Theorem 8.1]{MRUhre} There is a holomorphic motion
$\Phi:\Bbb{D}\times \mathbf{M}_0\rightarrow\Bbb{C}$
 of the connectedness locus $\mathbf{M}_0$ where the base point of the motion is $\lambda=0$ and $\Phi$ bijects the slice $\{\lambda\}\times\mathbf{M}_0$ onto $\mathbf{M}_{\lambda}$. This motion moreover respects the dynamics: There is a
 q.c. homeomorphism $h_{\lambda,\mu}:\hat{\Bbb{C}}\rightarrow\hat{\Bbb{C}}$ preserving the origin and the point at infinity that conjugates the dynamics of 
 $f_{0,\mu}$ in a neighborhood of $\mathcal{K}_{0,\mu}$ with the dynamics of $f_{\lambda,\Phi(\lambda,\mu)}$ in a neighborhood of $\mathcal{K}_{\lambda,\Phi(\lambda,\mu)}$, i.e.
 $h_{\lambda,\mu}\circ f_{0,\mu}=f_{\lambda,\Phi(\lambda,\mu)}\circ h_{\lambda,\mu}$ near 
 $\mathcal{K}_{0,\mu}$.
\end{theorem}

\begin{remark}
The proof of Theorem \ref{Uhre} in \cite{MRUhre} is based on the \textit{Branner-Hubbard holomorphic motion} \cite{MR2348954}. A similar situation is discussed in \cite[\S 3]{MR1062965} as well: Consider  quadratic rational maps 
$\frac{1}{\lambda}\left(z+\frac{1}{z}+b\right)$ 
with an attracting fixed point at $\infty$. For different multipliers $\lambda\in\Bbb{D}-\{0\}$ the connectedness loci in the $b$-plane are all homeomorphic. 
\end{remark}

\begin{proof}[Proof of Theorem \ref{monotonicity 1}]
The half  of the degree zero component of 
$\mathcal{M}_2(\Bbb{R})-\mathcal{S}(\Bbb{R})$
that in Figure \ref{fig:colored} lies below the line
$${\rm{Per}}_1(1):\sigma_2=2\sigma_1-3$$
is characterized by the existence of three real fixed points; cf. observation \hyperref[f]{4.f}. 
By Lemma \ref{attracting fixed point}, at least one of these fixed points should be attracting. Hence, after a real change of coordinates, one can always assume that there is an attracting fixed point of multiplier $\lambda\in (-1,1)$ at $\infty$ and a fixed point at $0$ whose multiplier belongs to $\Bbb{R}-\{1\}$. 
Therefore,  there exists a representative of the form \eqref{fixed-point normal form 1} for any such point of $\mathcal{M}_2(\Bbb{R})-\mathcal{S}(\Bbb{R})$. 
There are two choices $\mu,\mu'$ for the multiplier of the fixed point $0$ that, according to \eqref{fixed point formula}, satisfy 
$$
\frac{1}{1-\mu}+\frac{1}{1-\mu'}=\frac{\lambda}{\lambda-1}.
$$
The formula suggests that when $0\leq\lambda<1$ either $\mu>1$ or $\mu'>1$ while when $-1<\lambda\leq 0$ either $\mu<1$ or $\mu'<1$.
This discussion results in the following:
\begin{equation}\label{auxiliary 1}
\begin{split}
&\left\{\langle f\rangle\in\mathcal{M}_2(\Bbb{R})\mid f\in {\rm{Rat}}_2(\Bbb{R}) \text{ with three distinct real fixed points one of them attracting} \right\}\\
&=\left\{\langle f_{\lambda,\mu}\rangle \mid \lambda\in [0,1), \mu>1 \right\}\bigcup 
\left\{\langle f_{\lambda,\mu}\rangle \mid \lambda\in (-1,0], \mu<1\right\}.
\end{split}
\end{equation}
Away from the component of degree zero maps, if the restriction
$f\restriction_{\hat{\Bbb{R}}}:\hat{\Bbb{R}}\rightarrow\hat{\Bbb{R}}$
is a covering map with a real attracting fixed point, it must belong to the escape locus $E$: the critical points are complex conjugate and so if  under iteration one of them converges to a real cycle, the other critical orbit behaves the same way. Hence \eqref{auxiliary 1} can be rewritten as:
\begin{equation}\label{auxiliary 1'}
\begin{split}
&\text{the part of the degree zero component of } \mathcal{M}_2(\Bbb{R})-\mathcal{S}(\Bbb{R}) \text{  below } {\rm{Per}}_1(1) \text{ and away from }E\\
&=\left\{\langle f_{\lambda,\mu}\rangle \mid \lambda\in [0,1), \mu>1 \right\}\bigcup 
\left\{\langle f_{\lambda,\mu}\rangle \mid \lambda\in (-1,0], \mu<1\right\}-E.
\end{split}
\end{equation}
\indent
For $\lambda,\mu$ real and $\langle f_{\lambda,\mu}\rangle$ away from the escape locus $E$,  we have already constructed in Theorem \ref{straightening} a quasi-conformal conjugacy in vicinity of filled Julia sets which preserves the real entropy. 
So  for $\lambda\in(-1,1)$ and $\mu\in\mathbf{M}_0\cap\Bbb{R}$ in Theorem \ref{Uhre},  $\Phi(\lambda,\mu)$ should be real and the real quadratic rational map $f_{\lambda,\Phi(\lambda,\mu)}$ and the real quadratic polynomial $z^2+\mu z$ are of the same real entropy. In particular, given $\lambda\in (-1,1)$, $\Phi$ induces a homeomorphism 
\begin{equation}\label{auxiliary 2}
\Phi^{\lambda}_\Bbb{R}:[-2,4]=\mathbf{M}_0\cap\Bbb{R}\rightarrow\mathbf{M}_\lambda\cap\Bbb{R}
\end{equation}
satisfying
\begin{equation}\label{auxiliary 3}
h_\Bbb{R}\left(z\mapsto z^2+\mu z\right)=h_\Bbb{R}\left(f_{\lambda,\Phi^{\lambda}_\Bbb{R}(\mu)}\right).
\end{equation}
This will be proved in Lemma \ref{last}.
Next notice that in general the fixed point $0$ of multiplier $\mu'$ belongs to the filled Julia set $\mathcal{K}_{\lambda,\mu'}$ of $f_{\lambda,\mu'}$ (as defined in \eqref{filled Julia set 1}) and the topological behavior around a parabolic point is different from that for an attracting or repelling point. Thus, in Theorem \ref{Uhre}, the mere fact that  the
dynamics of $f_{0,\mu}$ over an open neighborhood of $\mathcal{K}_{0,\mu}$ is conjugate to the dynamics of 
$f_{\lambda,\Phi(\lambda,\mu)}$ over an open neighborhood of $\mathcal{K}_{\lambda,\Phi(\lambda,\mu)}$ guarantees that 
$\Phi(\lambda,1)\equiv 1$.  We conclude that 
$\Phi^\lambda_\Bbb{R}$  bijects $[-2,1)$, $(1,4]$   onto 
$\mathbf{M}_\lambda\cap(-\infty,1)$,
$\mathbf{M}_\lambda\cap(1,\infty)$
respectively; and furthermore, for  any arbitrary real  entropy value $h_0\in\left[0,\log(2)\right]$ and any fixed multiplier $\lambda_0\in (-1,1)$:
\begin{equation}\label{auxiliary 4}
\begin{split}
&\left\{\mu<1 \mid \langle f_{\lambda_0,\mu}\rangle\notin E,  h_{\Bbb{R}}(f_{\lambda_0,\mu})=h_0 \right\}
=\Phi\left(\left\{\lambda_0\right\}\times \left\{\mu_0\in[-2,1) \mid h_{\Bbb{R}}\left(z\mapsto z^2+\mu_0 z\right)=h_0 \right\}\right);\\
&\left\{\mu>1 \mid \langle f_{\lambda_0,\mu}\rangle\notin E, h_{\Bbb{R}}(f_{\lambda_0,\mu})=h_0 \right\}
=\Phi\left(\left\{\lambda_0\right\}\times \left\{\mu_0\in (1,4] \mid h_{\Bbb{R}}\left(z\mapsto z^2+\mu_0 z\right)=h_0 \right\}\right).
\end{split}
\end{equation}
Combining \eqref{auxiliary 1'} and \eqref{auxiliary 4} implies that the set
\begin{equation}\label{auxiliary 5}
\left\{\langle f\rangle \text{ in the degree zero component of }  \mathcal{M}_2(\Bbb{R})-\mathcal{S}(\Bbb{R}) 
\text{ below } {\rm{Per}}_1(1)\, \big|\, \langle f\rangle\notin E,  h_{\Bbb{R}}(f)=h_0\right\}
\end{equation}
can be written as the union of 
$$
\left\{\langle f_{\lambda,\mu}\rangle\,|\, \lambda\in[0,1), \mu\in\Phi\left([0,1)\times \left\{\mu_0\in(1,4]\, \big| \,
 h_{\Bbb{R}}\left(z\mapsto z^2+\mu_0 z\right)=h_0 \right\}\right)\right\}$$
and 
$$
\left\{\langle f_{\lambda,\mu}\rangle\,|\, \lambda\in(-1,0], \mu\in\Phi\left(\left(-1,0\right]\times \left\{\mu_0\in[-2,1)\, \big| \,
 h_{\Bbb{R}}\left(z\mapsto z^2+\mu_0 z\right)=h_0 \right\}\right)\right\}.$$
Each of these subsets is connected being homeomorphic to a (possibly degenerate) rectangle. 
Besides, they intersect due to the fact that along the subinterval
$$
\left\{\langle f_{0,\mu_0}\rangle\,|\,  \mu_0\in [-2,1)\right\}
=\left\{\langle f_{0,\mu_0}\rangle\,|\,  \mu_0\in (1,4]\right\}
=\left\{\langle z^2+c\rangle\,\Big|\, c\in \left[-2,\frac{1}{4}\right)\right\}
$$ 
of the polynomial line every entropy value in $h_0\in \left[0,\log(2)\right]$ is achieved. We conclude that the isentrope \eqref{auxiliary 5} is connected. 
In \eqref{auxiliary 5}, only the lower real escape component  below ${\rm{Per}}_1(1)$ should be excluded, and this is the real escape component over which $h_\Bbb{R}\equiv \log(2)$ (the maps have just one real fixed point in the  $h_\Bbb{R}\equiv 0$ escape component ; see \S\ref{S3.2}).  So to finish the proof, it suffices to establish the connectedness of the isentrope $h_\Bbb{R}=\log(2)$  in the portion of the component of degree zero maps described in the theorem. We claim that this isentrope is the closure of the $h_\Bbb{R}\equiv \log(2)$ escape component. 
The classes in the isentrope $h_\Bbb{R}=\log(2)$ should either be in the $h_\Bbb{R}\equiv \log(2)$ escape component of the component of degree zero maps in $\mathcal{M}_2(\Bbb{R})-\mathcal{S}(\Bbb{R})$ or in the form of 
$\left\langle f_{\lambda,\Phi^{\lambda}_\Bbb{R}(\mu)}\right\rangle$ where the second subscript belongs to the connectedness locus in the parameter plane $\Bbb{C}_{\lambda}$. 
Equation \eqref{auxiliary 3} describes the real entropy of this class as $h_\Bbb{R}\left(z\mapsto z^2+\mu z\right)$ which obtains its maximum possible value  $\log(2)$ only for  $f_{0,-2}(z)=z^2-2z, f_{0,4}(z)=z^2+4z$. These polynomials  are conjugate to the normalized Chebyshev polynomial $z^2-2$ and correspond to endpoints of $\mathbf{M}_0\cap\Bbb{R}$. Under $\Phi^\lambda_{\Bbb{R}}$, they are mapped onto the endpoints of the interval $\mathbf{M}_\lambda\cap\Bbb{R}$. Therefore, when $h_0=\log(2)$, the level set \eqref{auxiliary 5} can be written as 
$$
\left\{\langle f_{\lambda,\Phi(\lambda,4)}\rangle \mid \lambda\in [0,1)\right\}\bigcup 
\left\{\langle f_{\lambda, \Phi(\lambda,-2)}\rangle \mid \lambda\in (-1,0]\right\},
$$
which is a union of two connected curves intersecting  at 
$$
\langle f_{0,4}\rangle=\langle f_{0,-2}\rangle=\langle z^2-2\rangle;
$$
and is thus connected. Compare to Proposition \ref{escape boundary}: this is a part of the boundary of the
 $h_\Bbb{R}\equiv \log(2)$ escape component with the rest  of the boundary lying on the line ${\rm{Per}}_1(1)$. 
This concludes the proof.
\end{proof}

\begin{lemma}\label{last}
Notations as in Theorem \ref{Uhre}, for any $\lambda\in (-1,1)$: $\Phi(\lambda,\mu)\in\Bbb{R}\Leftrightarrow \mu\in\Bbb{R}$. 
\end{lemma}

\begin{proof}
Suppose $\lambda\in (-1,1)$. Conjugating all maps appearing in 
$h_{\lambda,\mu}\circ f_{0,\mu}=f_{\lambda,\Phi(\lambda,\mu)}\circ h_{\lambda,\mu}$
with $z\mapsto\bar{z}$ (cf. \eqref{involution}), one obtains 
$\widetilde{h_{\lambda,\mu}}\circ f_{0,\bar{\mu}}=f_{\lambda,\overline{\Phi(\lambda,\mu)}}\circ\widetilde{h_{\lambda,\mu}}$. If $\Phi(\lambda,\mu)\in\Bbb{R}$, then $\widetilde{h_{\lambda,\mu}}^{-1}\circ h_{\lambda,\mu}$ yields a hybrid equivalence between $z^2+\mu z$ and $z^2+\bar{\mu}z$ preserving the fixed point at the origin. Thus $\mu=\bar{\mu}$. Conversely, if $\mu\in\Bbb{R}$, then away from the basin of the attracting point at infinity,  
$\widetilde{h_{\lambda,\mu}}\circ h_{\lambda,\mu}^{-1}$ yields a hybrid equivalence between the polynomial-like maps corresponding to $f_{\lambda,\Phi(\lambda,\mu)}$ and $f_{\lambda,\overline{\Phi(\lambda,\mu)}}$ (cf. \eqref{poly-like}). But this conjugacy preserves the fixed point at the origin. Thus the multipliers there should coincide: 
$\Phi(\lambda,\mu)=\overline{\Phi(\lambda,\mu)}$.  
\end{proof}

\begin{proof}[Proof of Theorem \ref{temp1}]
Theorem \ref{monotonicity 1} -- that we just proved -- is simply a reformulation of Theorem \ref{temp1}:  According to  observation \hyperref[f]{4.f}, in the component of degree zero maps (defined in \S\ref{S2.4} by the absence of non-real critical points) the maps have three distinct real fixed points precisely when the corresponding classes are below the line  ${\rm{Per}}_1(1)$.
\end{proof}

\begin{corollary}\label{monotonicity 1'}
The entropy is monotonic on both the $(-+-)$-bimodal region and the part of the unimodal region that lies strictly below the line ${\rm{Per}}_1(1)$ and also throughout their union.
\end{corollary}
\begin{proof}
As is apparent from Figure \ref{fig:colored}, the half of the component of degree zero maps in 
$\mathcal{M}_2(\Bbb{R})-\mathcal{S}(\Bbb{R})$
which is strictly below ${\rm{Per}}_1(1)$
consists of certain portions of monotone increasing and  $(+-+)$-bimodal regions 
along with the whole  $(-+-)$-bimodal region and the aforementioned  part of the unimodal region. The latter two are of interest in the corollary  while the entropy is zero or $\log(2)$ over the former two (see observation \hyperref[i]{4.i}). Thinking of the union under consideration here as a closed subspace of the region appearing in Theorem \ref{monotonicity 1}, the first part of  Lemma \ref{point-set topology} finishes the proof. Applying the lemma once more, the regions comprising this union have the common boundary $\left\{\left\langle z^2+c\right\rangle\right\}_{c<0}$ restricted to which $h_\Bbb{R}$ is monotonic by the monotonicity of entropy for quadratic polynomials \cite[Corollary 13.2]{MR970571}. So $h_\Bbb{R}$ is monotonic  on both the $(-+-)$-bimodal region and the lower half of unimodal region.  
\end{proof}

\subsection{Bones in the unimodal region}\label{S6.2}
To understand Conjecture \ref{monotonicity 2} better, in this subsection we study the monotonicity of entropy in the whole unimodal region (whose interior consists of the classes of maps in $\mathcal{F}_2$ \eqref{family-2}) instead of only in the part of it that lies below ${\rm{Per}}_1(1)$, the part which we have already analyzed in Theorem \ref{monotonicity 1} by exploiting the convenient property of the existence of an attracting real fixed point. 

For a map $f$ from the unimodal region of $\mathcal{M}_2(\Bbb{R})$, only one real critical point, say $c_1$ 
($\frac{-2}{\mu+t}$ in \eqref{family-2}), 
contributes to the real entropy because the other critical point $c_2$ ($\frac{2}{\mu-t}$ in \eqref{family-2}) is outside the interval $f(\hat{\Bbb{R}})$. The region is confined between the polynomial line $f(c_2)=c_2$ $\sigma_1=2$ and the line $f(c_1)=c_2$ $\sigma_1=-6$ (that in family \eqref{family-2}  respectively correspond to lines $\mu-t=2$ and $\mu-t=-2$ on the boundary of the domain \eqref{domain-2}). The itinerary of $c_1$ under the unimodal interval map $f\restriction_{f(\hat{\Bbb{R}})}:f(\hat{\Bbb{R}})\rightarrow f(\hat{\Bbb{R}})$ determines its topological entropy and hence $h_\Bbb{R}(f)$. The change in the kneading coordinate of $c_1$ occurs once we hit parameters for which the critical point is periodic. The algebraic curves $f^{\circ n}(c_1)=c_1$ capturing such post-critical relations are known to be  smooth (but possibly disconnected)  subvarieties of $\mathcal{M}_2(\Bbb{C})\cong\Bbb{C}^2$ \cite{MR3145126}; their real loci thus consist of finitely many disjoint one-dimensional closed submanifolds of the plane $\mathcal{M}_2(\Bbb{R})\cong\Bbb{R}^2$.
Each component of such a real locus is either a Jordan curve or a closed subset of the plane diffeomorphic to the real line.  Following the terminology of \cite{MR1351522,MR1736945}, we call the former a \textit{bone-loop} and the latter a \textit{bone-arc}; see Figure \ref{fig:bone}. By a \textit{bone}, we mean a connected component of the real locus of a curve 
$f^{\circ n}(c_1)=c_1$; so a bone is either a bone-loop or a bone-arc. We are mainly interested in such critical orbit relations in the (open) unimodal region of $\mathcal{M}_2(\Bbb{R})$ where there is a distinguished critical point $c_1$ whose preimages are real. The \textit{order type} of the periodic critical point $c_1$ (meaning the ordering of its itinerary) persists along each bone.\footnote{One should think of the order type as a cyclic permutation obtained from writing the cycle containing $c_1$ in the ascending order. The definition of ``bone'' in \cite{MR1351522,MR1736945} is slightly different from ours as in those articles bones are associated to cyclic permutations while here we take them to be the connected components.}\\
\indent
A bone-arc crossing the unimodal region must intersect each of the lines $f(c_2)=c_2$ and $f(c_1)=c_2$ at precisely one point: The bone would be a closed, non-compact, and hence unbounded curve in the plane $\mathcal{M}_2(\Bbb{R})$
whose intersection points with the polynomial line are  PCF quadratic polynomials  with  prescribed kneading data 
that are of course outside the escape locus.  
Such an intersection point   must be unique due to the Thurston rigidity for quadratic polynomials 
\cite[Lemma 13.4]{MR970571}. In fact, a polynomial relation such as $f(c_2)=c_2$ is the only possible critically periodic relation in the escape component; and by Proposition \ref{escape boundary}, we know that the complement in the unimodal region of the real escape components is bounded.  We conclude that after hitting the polynomial line, the bone-arc cannot remain confined in the unimodal region and must hit the other boundary line $f(c_1)=c_2$ as well. In Figure \ref{fig:curves}, we have  detected points $(\mu,t)$ from the domain $\mathbf{U}_2$ \eqref{domain-2} of the family $\mathcal{F}_2$ \eqref{family-2} for which the corresponding map 
$$
f(x):=\frac{2\mu x(tx+2)}{\mu^2x^2+(tx+2)^2}
$$
satisfies $f^{\circ n}(c_1)=c_1$  where $c_1$ is the critical point
$\frac{-2}{\mu+t}$  
and $2\leq n\leq 7$. Observe that  all of the resulting curves are bone-arcs that, in accord with the preceding discussion,  connect a point on the line $\mu-t=2$ to a point on the line $\mu-t=-2$ (keep in mind that by observations \hyperref[b]{4.b} and \hyperref[c]{4.c},  these lines respectively correspond to  certain rays of the polynomial line and the $f(c_1)=c_2$ line in the moduli space). This motivates the conjecture below that mimics the  \textit{connected bone conjecture} for cubic polynomials in \cite{MR1351522,MR1736945}.

\begin{figure}[ht!]
\includegraphics[width=15cm, height=9cm]{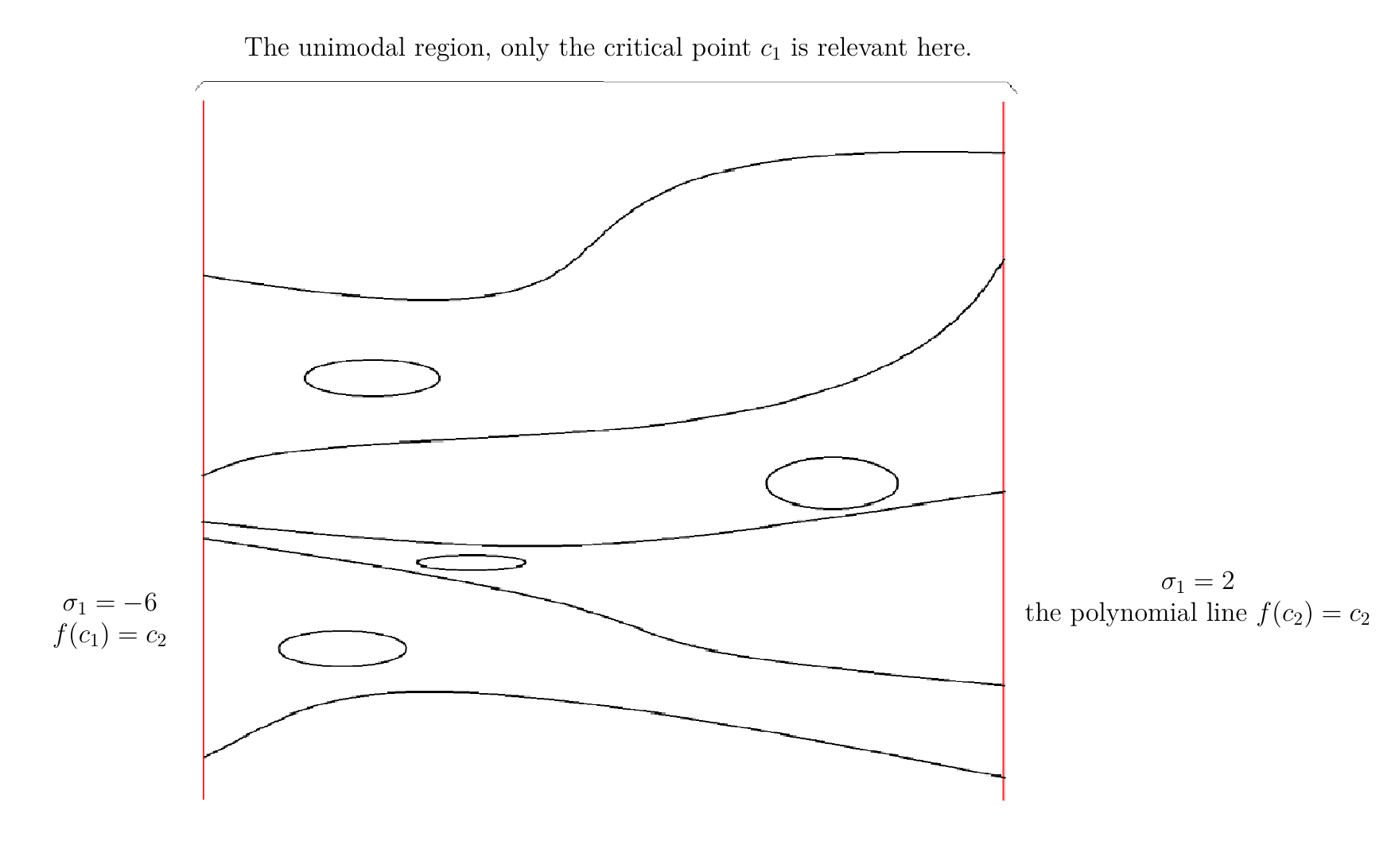}
\caption{A schematic picture of bones (in black)  defined by the post-critical relations $f^{\circ n}(c_1)=c_1$ in the unimodal region of the $(\sigma_1,\sigma_2)$-plane.  They are disjoint one-dimensional submanifolds. There are  bone-arcs connecting PCF quadratic polynomials to PCF quadratic rational maps on the $f(c_1)=c_2$ line. Any other bone in this region should be a Jordan curve, a bone-loop.}
\label{fig:bone}
\end{figure}

\begin{figure}[ht!]
\centering
\includegraphics[width=10cm]{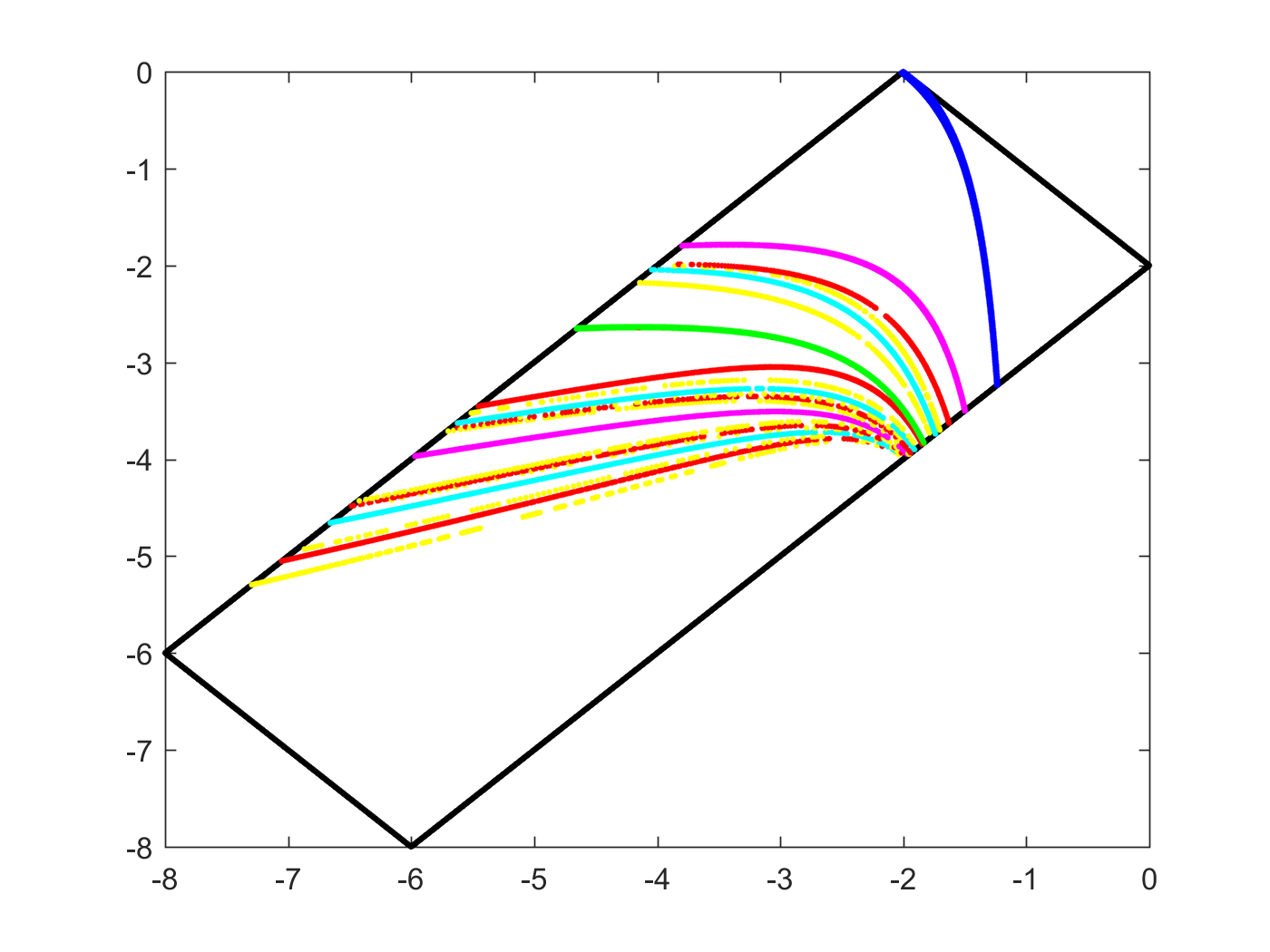}
\caption{In the rectangular region \eqref{domain-2} of the third quadrant of the $(\mu,t)$-plane  that involves all interesting unimodal parameters, we have numerically plotted the curves presenting parameters for which the unique turning point of the unimodal interval map is periodic for the first few periods. The colors indicate the periods:\,\, 
blue:$2$\,\, green:$3$\,\, magenta:$4$\,\, cyan:$5$\,\, red:$6$\,\, yellow:$7$. Compare with  Figure \ref{fig:plot2}.}
\label{fig:curves}
\end{figure}

\begin{namedthm}{No Bone-Loop Conjecture}[]\label{no bone-loop}
There is no bone-loop in the unimodal region.
\end{namedthm}
\noindent
As the kneading data, and hence the real entropy, cannot change without crossing a bone (the argument for 
\cite[Lemma 9.1]{MR1736945} can be applied with no change), there is no surprise that a parallel  conjecture on the absence of bone-loops has been utilized in the proof of monotonicity for cubic polynomials. In our context, as we shall see shortly, a weaker version suffices: 

\begin{namedthm}{Weak No Bone-Loop Conjecture}[]\label{weak no bone-loop}
There is no infinite nested sequence of bone-loops in the unimodal region.
\end{namedthm}
\noindent
It has to be mentioned  that for each critically periodic relation $f^{\circ n}(c_1)=c_1$ the critical point $c_1$ is stable, and hence the bones can be separated by  open neighborhoods. Nonetheless, there is a possibility of a sequence of bone-loops accumulating to a complicated fractal object; the possibility that the preceding conjecture rules out.\\
\indent
For cubic polynomials, the connected bone conjecture on the non-existence of bone-loops was settled in the thesis \cite{MR2695156}. The proof of monotonicity in  \cite{MR1736945} invokes this result while the earlier version \cite{MR1351522} relies on the stronger  \textit{density of hyperbolicity conjecture} along with a kneading theory explanation for why there is no PCF cubic on a   bone-loop. Now that the density of hyperbolicity for real maps has been established in the utmost generality  \cite[Theorem 2]{MR2342693}, it is suggestive to pose a parallel conjecture for real PCF quadratic rational maps.

\begin{remark}\label{density}
Below, we shall need the density of hyperbolicity for real quadratic rational maps with real critical points. This follows from the arguments used in the proof of density of hyperbolicity for real polynomials by Kozlovski, Shen and van Strien in \cite{MR2335796}. In that paper, the density for real polynomials is deduced from the \textit{rigidity}. The argument is based on a result of Shen \cite{MR1992673} that implies a real rational map with real non-degenerate critical points supports no \textit{invariant line filed} on its Julia set. The argument from \cite[\S 2]{MR2335796} carries over and implies the density of hyperbolicity for real quadratic rational maps provided that \textit{quasi-symmetric rigidity} holds for these maps. But quasi-symmetric rigidity is recently established in a very general setting by Clark and van Strien \cite{2018arXiv180509284C}.
\end{remark}

\begin{namedthm}{All PCF's  on Bone-Arcs Conjecture}[]\label{PCF}
Every PCF unimodal map lies on a bone-arc. 
\end{namedthm}

Next, we discuss how these conjectures are related to each other.
\begin{proposition}\label{conjectures}
The following implications hold: 
\small
\begin{equation*}\label{diagram1}
\xymatrix{ \text{no bone-loop conjecture }\ref{no bone-loop}\ar@{<=>}[r] \ar@{=>}[d] 
&\text{all PCF's on bone-arcs conjecture } \ref{PCF} \ar@{=>}[d] \\
\text{weak no bone-loop conjecture }\ref{weak no bone-loop}\ar@{=>}[r] 
&\text{monotonicity in the unimodal region } \ref{monotonicity 2}\ar@{=>}[d]\\
 &\text{monotonicity in unimodal and } (-+-) \text{-bimodal regions } \ref{monotonicity 3}
 }
\end{equation*}
\normalsize
\end{proposition}
\begin{proof}
If there is no bone-loop then, by the discussion after Conjecture \ref{no bone-loop}, over the unimodal region the kneading coordinate of the relevant critical point and hence the real entropy can vary only when one crosses the bone-arcs which are simple curves connecting the boundary lines $\sigma_1=2$, $\sigma_1=-6$ of the region; cf. Figure \ref{fig:bone}. Consequently, each isentrope $h_\Bbb{R}=h_0$ in the unimodal region deformation retracts to its intersection with the polynomial line which is connected due to the monotonicity of the entropy of real quadratic polynomials \cite{MR970571}. The weaker claim of Conjecture \ref{weak no bone-loop} implies Conjecture \ref{monotonicity 2} too: If there are bone-loops, the function $h_\Bbb{R}$ would still be monotonic on the complement of their interiors in the unimodal region because, by a simple planar topology argument, given a compact connected subset $A$  of the plane (say the closure of a connected component of the complement  of the unimodal region with respect to bone-arcs) and countably many Jordan curves $C_1,C_2,\dots$ (e.g. bone-loops) whose interior regions $D_i$'s are contained in  ${\rm{int}}(A)$, the compact subset $A-\bigcup_iD_i$ must be connected as well. Consequently, the failure of monotonicity can only be caused by having an entropy value realized  inside a bone-loop but not on it. Inside a bone-loop, the entropy can vary only upon hitting other bones there which are necessarily smaller bone-loops enclosed by the original one. So if there is no infinite nested sequence of bone-loops, the entropy should be constant inside each bone-loop by a simple continuity argument; hence Conjecture \ref{monotonicity 2}.\\
\indent
Invoking the second part of Lemma \ref{point-set topology}, the monotonicity of entropy throughout the whole unimodal region along with the monotonicity of entropy over the $(-+-)$-bimodal region (Corollary \ref{monotonicity 1'}) imply the monotonicity for the restriction of  $h_\Bbb{R}$ to the union of unimodal and $(-+-)$-bimodal regions since these two have the segment $\left\{\left\langle z^2+c\right\rangle\right\}_{c<0}$
in common along which all entropy values in $\left[0,\log(2)\right]$ are attained. \\
\indent
Establishing 
$\text{Conjecture \ref{no bone-loop}}\Leftrightarrow\text{Conjecture \ref{PCF}}$
will finish the proof. The right implication is  trivial, and the left one can be immediately deduced from  the density of hyperbolicity (see Remark \ref{density}): It makes sense to speak of the interior of a bone-loop $C$ because of the Jordan curve theorem. There is a hyperbolic component $\mathcal{U}\subset\mathcal{M}_2(\Bbb{C})$ whose real locus $\mathcal{U}_\Bbb{R}:=\mathcal{U}\cap\mathcal{M}_2(\Bbb{R})$
intersects this interior and it must be of type capture since, with notations as the rest of this subsection, in the open unimodal region  the critical point $c_2$ has imaginary preimages and thus cannot be periodic. Hence, there are positive integers $m,n$ such that for maps in $\mathcal{U}$ the critical point $c_1$ is in the immediate basin of an attracting $n$-cycle while the orbit of $c_2$ lands there after precisely $m$ iterations. This indicates that 
$f^{\circ n}(c_1)=c_1$, $f^{\circ m}(c_2)=c_1$
are the only possible post-critical relations in this component. These equations cut two connected\footnote{Compare with the analogous statement on hyperbolic components in the parameter space of real cubic polynomials employed in the proof of \cite[Lemma 6]{MR1351522}. The real locus of the curve 
 $f^{\circ n}(c_1)=c_1$  intersects the real hyperbolic component $\mathcal{U}_\Bbb{R}$ along a connected subset since,  in the real $\mathcal{J}$-stable component $\mathcal{U}_\Bbb{R}$,
one can  quasi-conformally deform a map from this intersection in a way that such a critical orbit relation  in the Fatou set persists.} curves in $\mathcal{U}_\Bbb{R}$ that intersect  at the unique PCF rational map in $\mathcal{U}$, the center of the component which is real as well by Proposition \ref{classification of components}. The proposition also says that $\mathcal{U}_\Bbb{R}$ is connected; so it should either intersect the Jordan curve $C$ or must be completely contained in its interior. In the former case the center of $\mathcal{U}_\Bbb{R}$ lies on $C$ while in the latter the part of the bone $f^{\circ n}(c_1)=c_1$ that is contained in $\mathcal{U}_\Bbb{R}$ is inside $C$; this bone cannot be a bone-arc as otherwise it must leave the interior of $C$ and this is not possible without hitting $C$, which cannot occur as different bones do not intersect. We conclude that if there is a bone-loop, then there exist bone-loops with PCF maps on them. 
\end{proof}

As mentioned in \S\ref{S1}, the monotonicity of entropy for unimodal real quadratic rational maps is further addressed in \cite{2020arXiv200903797G,2020arXiv200910147B}. We finish the section by pointing out that in our monotonicity results  in both \S\ref{S6.1} and \S\ref{S6.2}, the monotonicity of entropy for quadratic polynomials played a significant role. Conjecture \ref{no bone-loop} implies that the isentropes in the unimodal region 
 are topological rectangles or line segments built on (perhaps degenerate) intervals of constant entropy on the polynomial line $\sigma_1=2$ that expand leftward until they hit the other boundary line $\sigma_1=-6$. For the isentropes in the smaller green unimodal region of Figure \ref{fig:colored}, this has been established in the course of the proof of Theorem \ref{monotonicity 1}.  In contrast, the next section is about an exotic ``non-polynomial-like'' behavior. 

\section{The essentially non-polynomial case}\label{S7}
 
Following the terminology of \cite{MR1806289}, real quadratic maps with three repelling fixed points on the Riemann sphere can be called \textit{essentially non-polynomial}. This is precisely the situation for $(+-+)$-bimodal maps in family $\mathcal{F}_3$ \eqref{family-3} for which, based on Figure \ref{fig:plot4}, Conjecture \ref{temp2} claims that the entropy is not monotonic: The real fixed point of multiplier $\mu<-1$ at the origin  is the only real fixed point (observation \hyperref[f]{4.f}); and the non-real fixed points should of course be repelling as well due to the fact that critical points are real. In this brief section, we try to explore this ``non-polynomial'' attribution and its ramifications to the variation of entropy. For a detailed treatment of Conjecture \ref{temp2}, see \cite{1930-5311_2020_16_225}.

Figure \ref{fig:bifurcation} depicts the bifurcation in the third quadrant of the $(\mu,t)$-parameter space. The coloring scheme is based on the periods of the attracting cycles toward which the critical values $\pm 1$ converge. This is a superimposed diagram:  an ``average'' of different colors has been used wherever the critical points are attracted by cycles of different periods. The fact that inside the parabola -- where the ``interesting'' dynamics takes place -- the colors coincide simply means that type \textbf{D} hyperbolic components  where critical orbits are absorbed by different cycles are too small to be visible here; compare with \cite[p. 66]{MR1246482} and also with \cite{MR1764925} which establishes that the type \textbf{D} hyperbolic components  disjoint from the polynomial line are bounded. 

The ``forking'' of colored channels in Figure \ref{fig:bifurcation} attests to a bifurcation qualitatively different from what the polynomial family goes through. For instance, there is a light blue-green color corresponding to period 5 that splits into three channels before undergoing a period-doubling bifurcation to period 10 (in a magenta color); see Figure 
\ref{fig:bifurcation zoomed}. This suggests that the entropy may go up in various directions and hence, is probably non-monotonic. This is in contrast with the heuristic interpretation of monotonicity of entropy for polynomials that \cite{MR3289915} puts forward: ``Families of real polynomial maps undergo bifurcations in the simplest possible way.''

\begin{figure}[ht!]
\includegraphics[width=14cm, height=7cm]{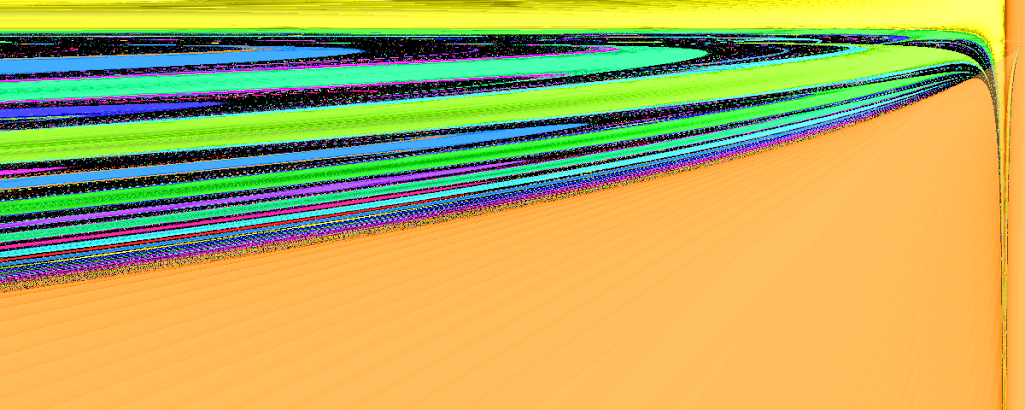}
\caption{The bifurcation diagram for the family 
$\left\{x\mapsto\frac{2\mu x(tx+2)}{\mu^2x^2+(tx+2)^2}:[-1,1]\rightarrow[-1,1]\right\}_{(\mu,t)}$ where 
$-50<\mu<0, -20<t<0$. The colors indicate periods of attracting cycles; orange is $1$, yellow is $2$, etc. The interesting bifurcation behavior for $(+-+)$-bimodal maps takes place inside the parabola $\mu=1-\frac{t^2}{4}$. 
 The $(+-+)$-bimodal parameters $(\mu,t)$ outside the parabola  are colored orange meaning that both critical points converge to an attracting fixed point and hence $(\mu,t)$ lies in the escape locus, as predicted in observation \hyperref[i]{4.i}.}
\label{fig:bifurcation}
\end{figure}

\begin{figure}[ht!]
\includegraphics[width=14cm, height=7cm]{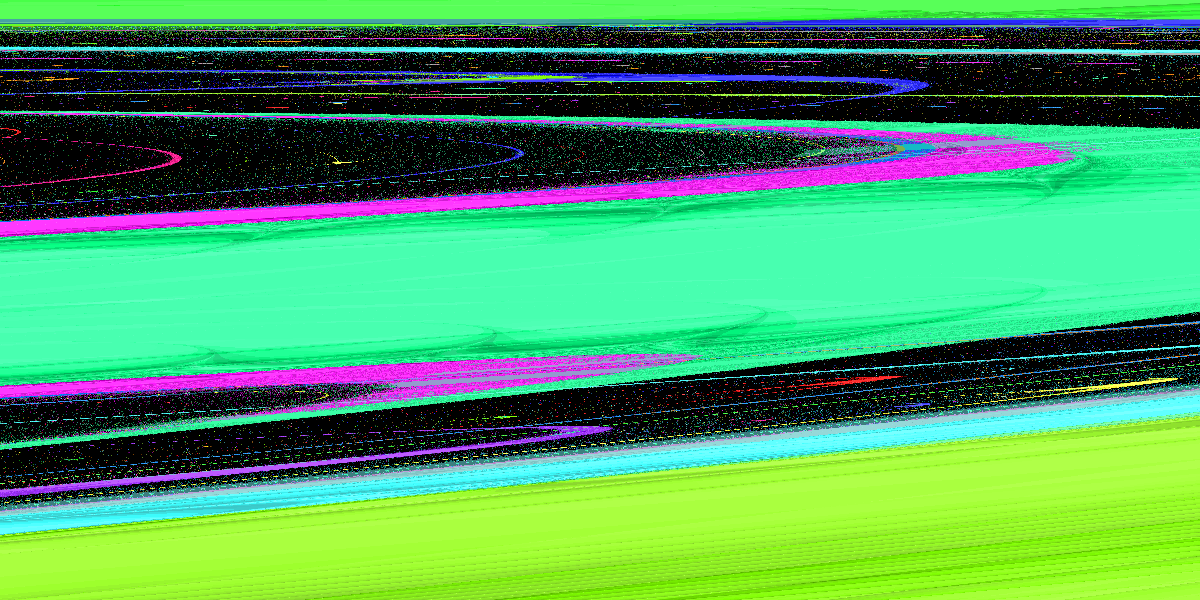}
\caption{Zooming on the part of Figure \ref{fig:bifurcation} defined by $-26<\mu<-19, -5<t<-1$,  period-doubling bifurcations from a $5$-cycle to a $10$-cycle are visible as the transition from green to magenta.}
\label{fig:bifurcation zoomed}
\end{figure}

Finally, we point out that the kneading theory of $(+-+)$-bimodal maps is much more complicated than that of cubic polynomials analyzed in \cite{MR1351522,MR1736945}. There are compatibility relations for itineraries of turning points of a multimodal interval map that include the itineraries of the endpoints \cite[\S7.4]{MR1963683}. Unlike the cubic interval maps in those articles, the interval map $f\restriction_{f(\hat{\Bbb{R}})}:f(\hat{\Bbb{R}})\rightarrow f(\hat{\Bbb{R}})$ 
obtained from the restriction of a real quadratic map $f$ with real critical points is almost never boundary-anchored; quite the contrary, the boundary points are the critical values and their orbits therefore capture all the kneading data. Moreover, a cubic interval map with real critical points can achieve any entropy value between $0$ and $\log(3)$ while for a $(+-+)$-bimodal interval map, although of the same modality, the maximum possible entropy is $\log(2)$. This issue of  ``high modality versus low entropy'' indicates that symbolic sequences realized as the itineraries of critical points of  $(+-+)$-bimodal maps are highly restricted. By contrast, associated with a cubic map is a stunted sawtooth map that comes with the same kneading data, and there is a thorough characterization of symbolic sequences admissible as itineraries under these stunted sawtooth models
\cite[Theorem 5.2]{MR1736945}. This assignment of piecewise-linear models with the same kneading invariants to polynomial interval maps is a crucial part of the treatment of Milnor's conjecture in \cite{MR1351522, MR1736945, MR3264762,MR3999686} which is absent from our context of real quadratic rational maps due to their mysterious combinatorics. As a matter of fact, it is not clear what  piecewise-linear model captures the kneading theory of $(+-+)$-bimodal maps: In the case of degree $d$ polynomials, stunted sawtooth maps are obtained by chopping off a piecewise-linear map of the same shape with $d$ laps of slope $\pm d$ and with the maximum possible entropy $\log(d)$; the entropy which up to affine conjugacy is realized only by Chebyshev polynomials of degree $d$. But in the context of rational maps, there is a continuum of conformally distinct  maps that maximize the real entropy, e.g. maps $z+\frac{1}{z}+t$ that for $t>2$ are $(+-+)$-bimodal and on the boundary of the $h_\Bbb{R}\equiv\log(2)$ escape locus; see Proposition \ref{escape boundary}. Hence it is not clear what is the natural piecewise-linear $(+-+)$-bimodal map of entropy $\log(2)$ that has to be modified in order to construct an appropriate class of stunted maps. 
\begin{question}
Is there a piecewise-linear model that captures the combinatorics of $(+-+)$-bimodal real quadratic rational maps?
\end{question}

\section{Higher degree rational maps with two critical points}\label{S8}
The paper \cite{MR1806289} of Milnor is devoted to the study of rational maps with precisely two critical points. A degree $n$ map with this property must have two distinct critical points at each of which the critical value is attained with multiplicity $n$. The corresponding complex moduli space formed by the Möbius conjugacy classes of degree $n$ \textit{bicritical} rational maps will be denoted by $\mathcal{M}_n^{\rm{bicrit}}(\Bbb{C})$. This space can be regarded as the quotient of the moduli space of the same type of maps with marked critical points by the involution that interchanges the critical points, the description that can be used to identify $\mathcal{M}_n^{\rm{bicrit}}(\Bbb{C})$ with the affine plane $\Bbb{C}^2$
 \cite[\S1]{MR1806289}. It has to be mentioned that, for $n=2$,  this affine structure on $\mathcal{M}_2^{\rm{bicrit}}(\Bbb{C})=\mathcal{M}_2(\Bbb{C})$ differs from the one determined by $\sigma_1,\sigma_2$ with which we have been working so far \cite[Remark 6.3]{MR1246482}. The parity of $n$ comes into play when one considers the locus in $\mathcal{M}_n^{\rm{bicrit}}(\Bbb{C})$ determined by real maps:
\begin{itemize}
\item for $n$ odd, the underlying affine plane $\mathcal{M}_n^{\rm{bicrit}}(\Bbb{R})\cong\Bbb{R}^2$ involves classes that cannot be represented by real bicritical maps (i.e. the \textit{field of moduli} is not a \textit{field of definition}); these are \textit{antipodal} maps that are absent in even degrees due to topological obstructions;
\item for  $n$ odd, the restriction of a bicritical map of degree $n$ with real coefficients to the real circle $\hat{\Bbb{R}}$ is not interesting from the entropy point of view due the fact that the induced self-map of $\hat{\Bbb{R}}$ is an $n$-sheeted covering, and thus of entropy $\log(n)$, if the critical points are not real, whereas when the critical points lie on $\hat{\Bbb{R}}$ they would be inflection points of the induced real map rather than its extrema and the entropy is thus zero;
\end{itemize}
see \cite[\S5]{MR1806289} for details. Consequently, we will take $n$ to be even hereafter. 

Many of the results obtained in previous sections for real quadratic maps can potentially carry over to the case of real bicritical maps of degree $n$. Again, one can simply concentrate on the case where the critical points are real; and then there is a convenient normal form   with critical values $\pm 1$ and a fixed  point of multiplier $\mu$ at the origin:  
\begin{equation}\label{new normal form bicritical}
\left\{x\mapsto\frac{((tx+n)+\mu x)^n-((tx+n)-\mu x)^n}{((tx+n)+\mu x)^n+((tx+n)-\mu x)^n}:[-1,1]\rightarrow [-1,1]\right\}_{\mu\in\Bbb{R}-\{0\}, \mu t\geq 0},
\end{equation}
that generalizes \eqref{new normal form}.
The corresponding auxiliary $(\mu,t)$-parameter space admits a finite-to-one map onto the complement of the covering regions in $\mathcal{M}_n^{\rm{bicrit}}(\Bbb{R})$, and in analogy with Figure \ref{fig:parameter space} can be partitioned to monotone increasing, monotone decreasing, unimodal and finally bimodal regions of shapes $(+-+)$ or $(-+-)$ based on how the critical points $\frac{n}{\pm\mu-t}$ of \eqref{new normal form bicritical} are located with respect to $[-1,1]$. On the other hand, in higher (even) degrees the parabola $\mu=1-\frac{t^2}{4}$ in that picture should be replaced with a higher degree curve in the $(\mu,t)$-plane, and Lemma \ref{attracting fixed point} is not true anymore since there will be more than three fixed points generically. Nevertheless, assuming the claim of the lemma on the existence of an attracting fixed point; that is, in the presence of a ``polynomial-like  behavior''\footnote{To be precise, the paper \cite{MR1806289} uses this terminology only when there is just one attracting fixed point; and calls bicritical maps with more than one attracting fixed point to be in the \textit{principal hyperbolic component} where the Julia set is a quasi-circle. In this component the real entropy is the same as that of the ``center'' $z\mapsto z^n$ of the component which is zero; cf. Theorem \ref{entropy constant over hyperbolic}.}, it is conceivable that a monotonicity result similar to Theorem \ref{monotonicity 1} can be deduced from the monotonicity of entropy for the family $\left\{z^n+c\right\}_{c\in\Bbb{R}}$ of \textit{unicritical} polynomials (which can be established based on the ideas developed in \cite{MR1764936}). With this direction of further inquiry, we conclude the paper. 

\begin{question}
Let $n\geq 2$ be an even integer. Consider the region of the moduli space $\mathcal{M}_n^{\rm{bicrit}}(\Bbb{R})\cong\Bbb{R}^2$ of real bicritical maps of degree $n$ where the critical points are real and there exists a real attracting fixed point. Are the isentropes connected in this region? 
\end{question}

\bibliographystyle{alpha}
\bibliography{bibcomprehensive}

\begin{thebibliography}{DGMT95}

\bibitem[BK92]{MR1151977}
L.~Block and J.~Keesling.
\newblock Computing the topological entropy of maps of the interval with three
  monotone pieces.
\newblock {\em J. Statist. Phys.}, 66(3-4):755--774, 1992.

\bibitem[BKLP89]{MR1002478}
L.~Block, J.~Keesling, S.~Li, and K.~Peterson.
\newblock An improved algorithm for computing topological entropy.
\newblock {\em J. Statist. Phys.}, 55(5-6):929--939, 1989.

\bibitem[BMS20]{2020arXiv200910147B}
Araceli {Bonifant}, John {Milnor}, and Scott {Sutherland}.
\newblock {The W. Thurston Algorithm for Real Quadratic Rational Maps}.
\newblock {\em arXiv e-prints}, page arXiv:2009.10147, September 2020.

\bibitem[Bre72]{MR0413144}
G.~E. Bredon.
\newblock {\em Introduction to compact transformation groups}.
\newblock Academic Press, New York-London, 1972.
\newblock Pure and Applied Mathematics, Vol. 46.

\bibitem[BS02]{MR1963683}
M.~Brin and G.~Stuck.
\newblock {\em Introduction to dynamical systems}.
\newblock Cambridge University Press, Cambridge, 2002.

\bibitem[BvS15]{MR3264762}
H.~Bruin and S.~van Strien.
\newblock Monotonicity of entropy for real multimodal maps.
\newblock {\em J. Amer. Math. Soc.}, 28(1):1--61, 2015.

\bibitem[CvS18]{2018arXiv180509284C}
T.~Clark and S.~van Strien.
\newblock {Quasisymmetric rigidity in one-dimensional dynamics}.
\newblock {\em arXiv e-prints}, page arXiv:1805.09284, May 2018.

\bibitem[DGMT95]{MR1351522}
S.~P. Dawson, R.~Galeeva, J.~W. Milnor, and C.~Tresser.
\newblock A monotonicity conjecture for real cubic maps.
\newblock In {\em Real and complex dynamical systems ({H}iller\o d, 1993)},
  volume 464 of {\em NATO Adv. Sci. Inst. Ser. C Math. Phys. Sci.}, pages
  165--183. Kluwer Acad. Publ., Dordrecht, 1995.

\bibitem[DH85a]{MR762431}
A.~Douady and J.~H. Hubbard.
\newblock {\em \'{E}tude dynamique des polyn\^{o}mes complexes. {I}\& II}.
\newblock Publications Math\'{e}matiques d'Orsay. 1984,1985.

\bibitem[DH85b]{MR816367}
A.~Douady and J.~H. Hubbard.
\newblock On the dynamics of polynomial-like mappings.
\newblock {\em Ann. Sci. \'Ecole Norm. Sup. (4)}, 18(2):287--343, 1985.

\bibitem[dMvS93]{MR1239171}
W.~de~Melo and S.~van Strien.
\newblock {\em One-dimensional dynamics}, volume~25 of {\em Ergebnisse der
  Mathematik und ihrer Grenzgebiete (3) [Results in Mathematics and Related
  Areas (3)]}.
\newblock Springer-Verlag, Berlin, 1993.

\bibitem[Dou95]{MR1351519}
A.~Douady.
\newblock Topological entropy of unimodal maps: monotonicity for quadratic
  polynomials.
\newblock In {\em Real and complex dynamical systems ({H}iller\o d, 1993)},
  volume 464 of {\em NATO Adv. Sci. Inst. Ser. C Math. Phys. Sci.}, pages
  65--87. Kluwer Acad. Publ., Dordrecht, 1995.

\bibitem[Eps00]{MR1764925}
A.~L. Epstein.
\newblock Bounded hyperbolic components of quadratic rational maps.
\newblock {\em Ergodic Theory Dynam. Systems}, 20(3):727--748, 2000.

\bibitem[Fil21]{filom2021real}
K.~Filom.
\newblock Real entropy rigidity under quasi-conformal deformations.
\newblock {\em Conformal Geometry and Dynamics}, 25(1):1--33, 2021.

\bibitem[FLMn83]{MR736568}
A.~Freire, A.~Lopes, and R.~Ma\~n\'e.
\newblock An invariant measure for rational maps.
\newblock {\em Bol. Soc. Brasil. Mat.}, 14(1):45--62, 1983.

\bibitem[FP20]{1930-5311_2020_16_225}
K.~{Filom} and K.~M. {Pilgrim}.
\newblock On the non-monotonicity of entropy for a class of real quadratic
  rational maps.
\newblock {\em Journal of Modern Dynamics}, 16(8):225--254, 2020.

\bibitem[{Gao}20]{2020arXiv200903797G}
Y.~{Gao}.
\newblock {Monotonicity of entropy for unimodal real quadratic rational maps}.
\newblock {\em arXiv e-prints}, page arXiv:2009.03797, September 2020.

\bibitem[GK90]{MR1062965}
L.~R. Goldberg and L.~Keen.
\newblock The mapping class group of a generic quadratic rational map and
  automorphisms of the {$2$}-shift.
\newblock {\em Invent. Math.}, 101(2):335--372, 1990.

\bibitem[Hec96]{MR2695156}
C.~A. Heckman.
\newblock {\em Monotonicity and the construction of quasiconformal conjugacies
  in the real cubic family}.
\newblock ProQuest LLC, Ann Arbor, MI, 1996.
\newblock Thesis (Ph.D.)--State University of New York at Stony Brook.

\bibitem[Koz19]{MR3999686}
O.~Kozlovski.
\newblock On the structure of isentropes of real polynomials.
\newblock {\em J. Lond. Math. Soc. (2)}, 100(1):159--182, 2019.

\bibitem[KR13]{MR3145126}
J.~Kiwi and M.~Rees.
\newblock Counting hyperbolic components.
\newblock {\em J. Lond. Math. Soc. (2)}, 88(3):669--698, 2013.

\bibitem[KSvS07a]{MR2342693}
O.~Kozlovski, W.~Shen, and S.~van Strien.
\newblock Density of hyperbolicity in dimension one.
\newblock {\em Ann. of Math. (2)}, 166(1):145--182, 2007.

\bibitem[KSvS07b]{MR2335796}
O.~Kozlovski, W.~Shen, and S.~van Strien.
\newblock Rigidity for real polynomials.
\newblock {\em Ann. of Math. (2)}, 165(3):749--841, 2007.

\bibitem[LSvS20]{MR4115082}
G.~Levin, W.~Shen, and S.~van Strien.
\newblock Positive transversality via transfer operators and holomorphic
  motions with applications to monotonicity for interval maps.
\newblock {\em Nonlinearity}, 33(8):3970--4012, 2020.

\bibitem[Mil92]{MR1181083}
J.~W. Milnor.
\newblock Remarks on iterated cubic maps.
\newblock {\em Experiment. Math.}, 1(1):5--24, 1992.

\bibitem[Mil93]{MR1246482}
J.~W. Milnor.
\newblock Geometry and dynamics of quadratic rational maps.
\newblock {\em Experiment. Math.}, 2(1):37--83, 1993.
\newblock With an appendix by the author and Lei Tan.

\bibitem[Mil00]{MR1806289}
J.~W. Milnor.
\newblock On rational maps with two critical points.
\newblock {\em Experiment. Math.}, 9(4):481--522, 2000.

\bibitem[Mil06]{MR2193309}
J.~W. Milnor.
\newblock {\em Dynamics in one complex variable}, volume 160 of {\em Annals of
  Mathematics Studies}.
\newblock Princeton University Press, Princeton, NJ, third edition, 2006.

\bibitem[Mis95]{MR1372979}
M.~Misiurewicz.
\newblock Continuity of entropy revisited.
\newblock In {\em Dynamical systems and applications}, volume~4 of {\em World
  Sci. Ser. Appl. Anal.}, pages 495--503. World Sci. Publ., River Edge, NJ,
  1995.

\bibitem[Mn83]{MR736567}
R.~Ma\~n\'e.
\newblock On the uniqueness of the maximizing measure for rational maps.
\newblock {\em Bol. Soc. Brasil. Mat.}, 14(1):27--43, 1983.

\bibitem[MS80]{MR579440}
M.~Misiurewicz and W.~Szlenk.
\newblock Entropy of piecewise monotone mappings.
\newblock {\em Studia Math.}, 67(1):45--63, 1980.

\bibitem[MT88]{MR970571}
J.~W. Milnor and W.~P. Thurston.
\newblock On iterated maps of the interval.
\newblock In {\em Dynamical systems ({C}ollege {P}ark, {MD}, 1986--87)}, volume
  1342 of {\em Lecture Notes in Math.}, pages 465--563. Springer, Berlin, 1988.

\bibitem[MT00]{MR1736945}
J.~W. Milnor and C.~Tresser.
\newblock On entropy and monotonicity for real cubic maps.
\newblock {\em Comm. Math. Phys.}, 209(1):123--178, 2000.
\newblock With an appendix by Adrien Douady and Pierrette Sentenac.

\bibitem[PT06]{MR2348954}
C.~L. Petersen and Lei Tan.
\newblock Branner-{H}ubbard motions and attracting dynamics.
\newblock In {\em Dynamics on the {R}iemann sphere}, pages 45--70. Eur. Math.
  Soc., Z\"{u}rich, 2006.

\bibitem[Ree90]{MR1047139}
M.~Rees.
\newblock Components of degree two hyperbolic rational maps.
\newblock {\em Invent. Math.}, 100(2):357--382, 1990.

\bibitem[She03]{MR1992673}
W.~Shen.
\newblock On the measurable dynamics of real rational functions.
\newblock {\em Ergodic Theory Dynam. Systems}, 23(3):957--983, 2003.

\bibitem[Sil98]{MR1635900}
J.~H. Silverman.
\newblock The space of rational maps on {$\bold P^1$}.
\newblock {\em Duke Math. J.}, 94(1):41--77, 1998.

\bibitem[Tsu00]{MR1764936}
M.~Tsujii.
\newblock A simple proof for monotonicity of entropy in the quadratic family.
\newblock {\em Ergodic Theory Dynam. Systems}, 20(3):925--933, 2000.

\bibitem[Uhr03]{MRUhre}
E.~Uhre.
\newblock {\em Construction of a holomorphic motion in part of the parameter
  space for a family of quadratic rational maps}.
\newblock Master Thesis. Roskilde University, 2003.

\bibitem[vS14]{MR3289915}
S.~van Strien.
\newblock Milnor's conjecture on monotonicity of topological entropy: results
  and questions.
\newblock In {\em Frontiers in complex dynamics}, volume~51 of {\em Princeton
  Math. Ser.}, pages 323--337. Princeton Univ. Press, Princeton, NJ, 2014.

\end{thebibliography}

\end{document}